\documentclass{amsart}

\theoremstyle{plain}
\newtheorem{theorem}{Theorem}
\newtheorem{lemma}[theorem]{Lemma}
\newtheorem{proposition}[theorem]{Proposition}
\newtheorem{corollary}[theorem]{Corollary}

\theoremstyle{definition}

\newtheorem{example}{Example}

\theoremstyle{remark}
\newtheorem{remark}{Remark}

\usepackage{amsfonts}

\usepackage[enableskew,vcentermath]{youngtab}

\usepackage{mathrsfs}

\newcommand{\G}{\mathfrak{g}}

\newcommand{\bo}{\mathfrak{b}}

\newcommand{\tligne}[2]{
#2 \\
\cline{1-#1}}

\newenvironment{tyoung}[1]{
\begin{array}{|c|c|c|c|c|c|c|c|c|c|c|c|}
\cline{1-#1} }{\end{array}}

\newcommand{\ttrait}{-\!\!\!-\!\!\!-\!\!\!-\!\!\!-\!\!\!-}
\newcommand{\trait}{-\!\!\!-\!\!\!-}

\newcommand{\niveau}[1]{{ }^{{ }^{{ }^{\displaystyle{#1}}}}}

\newcommand{\itilde}{\tilde{\imath}}

\begin{document}

\title[Springer fibers in the hook, 2-row and 2-column cases]{A unified approach on Springer fibers in the hook, two-row and
two-column cases\thanks{Work supported in part by Minerva grant, No. 8596/1}}

\author{Lucas Fresse}
\address{Department of Mathematics \\ the Weizmann Institute of
Science \\ Rehovot 76100, Israel}
\email{lucas.fresse@weizmann.ac.il }


\maketitle

\begin{abstract}
We consider the Springer fiber over a nilpotent endomorphism. Fix
a Jordan basis and consider the standard torus relative to this.
We deal with the problem to describe the flags fixed by the torus
which belong to a given component of the Springer fiber. We solve
the problem in the hook, two-row and two-column cases. We provide
two main characterizations which are common to the three cases,
and which involve dominance relations between Young diagrams and
combinatorial algorithms. Then, for these three cases, we deduce
topological properties of the components and their
intersections.
\end{abstract}

\section{Introduction}

\label{introduction-introduction}

Let $G$ be a connected reductive algebraic group over $\mathbb{C}$
and let $\G$ be its Lie algebra. The set ${\mathcal B}$ of the
Borel subalgebras $\bo\subset \G$ is a projective algebraic
variety, called {\em the flag variety}. For a nilpotent element
$u\in\G$, the set
$${\mathcal B}_u=\{\bo\in {\mathcal B}:u\in\bo\}$$
is a closed subvariety of ${\mathcal B}$. The variety ${\mathcal
B}_u$ is called a {\em Springer fiber} since it is the fiber over
$u$ of the Springer resolution ${\mathcal X}\rightarrow{\mathcal
N}$, $(\bo,u)\mapsto u$, where ${\mathcal N}\subset \G$ is the
subset of nilpotent elements and ${\mathcal
X}=\{(\bo,u)\in{\mathcal B}\times{\mathcal N}: u\in \bo\}$ (see
\cite{Springer1}).

The study of Springer fibers involves different fields as
algebraic geometry, representation theory and combinatorics. The
origin of the study dates back to the geometric realization, due
to T.A. Springer, of the irreducible representations of the Weyl groups
in the cohomology of Springer fibers (see \cite{Springer2}).
D. Kazhdan and G. Lusztig gave a topological construction of Springer
representations (see \cite{Kazhdan-Lusztig2}) and conjectured a
link between the configuration of the irreducible components of
Springer fibers and the construction of bases for the
representations of the Hecke algebra (see \cite[\S
6.3]{Kazhdan-Lusztig}). These have been strong motivations which
have made Springer fibers be an important topic of study in modern
algebra.

However, up to now, few questions have been solved. Even for the
type $A$, advances have been done only in few particular cases.
One of the major difficulties seems to be that the geometry of
${\mathcal B}_u$ strongly depends on the Jordan form of $u$, and
the study in each case is very specific.

\medskip
Throughout this article, we study the Springer
fibers for $G=GL(\mathbb{C}^n)$: we set $V=\mathbb{C}^n$, and
$u:V\rightarrow V$ is a nilpotent endomorphism. We identify
${\mathcal B}$ with the variety of {\em complete flags}, i.e.
chains of subspaces $(0=V_0\subset V_1\subset\ldots\subset V_n=V)$
with $\dim V_i=i$ for every $i=0,\ldots,n$. We identify ${\mathcal
B}_u$ with the closed subvariety of complete flags
$(V_0,\ldots,V_n)\in{\mathcal B}$ which are stable by $u$, i.e.
$u(V_i)\subset V_i$ for every $i=0,\ldots,n$.

It is known from N. Spaltenstein \cite{Spaltenstein1},
\cite{Spaltenstein} and R. Steinberg \cite{Steinberg} that the
irreducible components ${\mathcal K}^T\subset{\mathcal B}_u$ are
parameterized by a set of standard tableaux $T$. We recall
Spaltenstein's construction in \ref{Spaltenstein-construction}.

We study the components from the point of view of a set of special
flags $F_\tau\in{\mathcal B}_u$ which are parameterized by a set
of row-standard tableaux $\tau$, and depend on the initial choice
of a Jordan basis. Formally, they are the elements in ${\mathcal
B}_u$ which are fixed by the standard torus relative to the basis,
for its linear action on flags.

The main part of this article concerns the problem to determine
whether the flag $F_\tau$ lies in a given component ${\mathcal
K}^T$, and to interpret this problem into a combinatorial one
involving the pair $(\tau,T)$. We formulate the problem in an
intrinsic way in the sequel of the present section.

First, we establish necessary or sufficient criteria in the
general case (sections \ref{2}, \ref{3} and \ref{4}).

Next, we give the full answer in three particular cases: the hook,
two-row and two-column cases. We provide two main
characterizations, which are common to the three cases. The first
one involves a series of dominance relations between Young
diagrams (section \ref{5}). The second one relies on an algorithm,
whose aim is to construct the tableau $\tau$ according to certain
rules depending on $(\tau,T)$: if the construction succeeds, then the
pair $(\tau,T)$ is said to be constructible, and we show that
$F_\tau\in{\mathcal K}^T$ exactly in this case (section \ref{6}).

Finally, in these three cases, we provide a connection between the
combinatorics involved in our criteria and the existing
combinatorics involved in the previous studies of the components
of the Springer fibers in \cite{Fung}, \cite{Ml-p},
\cite{Melnikov-Pagnon1}, \cite{Melnikov-Pagnon}, \cite{Vargas}
(section \ref{7}), and we deduce topological properties of the
components and their intersections, especially the intersections
in codimension one, which play a crucial role in Kazhdan-Lusztig's
conjecture (section \ref{8}). Especially we provide a new
characterization of the pairs of components intersecting in
codimension one in the two-row case (Theorem
\ref{theorem-codim1-section-2row}). In the three cases, we give a
simple necessary condition for having an intersection in
codimension one (Theorem \ref{lastsection-maintheorem}).

\subsection{Spaltenstein's construction of the components of ${\mathcal B}_u$}

\label{young-diagram} \label{Spaltenstein-construction}

Recall that a {\em Young diagram} is a set of empty boxes
displayed along left-justified rows, whose length decreases to the
bottom. If $Y$ is a Young diagram with $n$ boxes, a {\em standard
tableau} of shape $Y$ is a numbering of $Y$ by $1,...,n$ such that
numbers in the rows increase to the right and numbers in the
columns increase to the bottom.

Up to isomorphism, the Springer fiber ${\mathcal B}_u$ only
depends on the Jordan form of $u$, which can be represented by a
Young diagram: let $\lambda_1\geq ...\geq\lambda_r$ be the sizes
of the Jordan blocks of $u$, then we denote by $Y(u)$ the Young diagram
of rows of lengths $\lambda_1,...,\lambda_r$. Let $\lambda^*_1\geq
...\geq\lambda^*_s$ be the lengths of the columns of $Y(u)$. We know that 
(see \cite[\S II.5.5]{Spaltenstein})
$$\nonumber \dim\,{\mathcal
B}_u=\sum_{j=1}^s\frac{1}{2}\,\lambda^*_j(\lambda^*_j-1).$$

Let us recall from \cite{Spaltenstein} that the irreducible
components of ${\mathcal B}_u$ are parameterized by the standard
tableaux of shape $Y(u)$. For $T$ standard, the shape of the
subtableau $T[1,...,i]$ of entries $1,...,i$ is a subdiagram
$Y_i^T\subset Y(u)$. In this manner we associate to $T$ an
increasing sequence of subdiagrams $Y_1^T\subset ...\subset
Y_n^T=Y(u)$. On the other hand, let $F=(V_0,...,V_n)\in{\mathcal
B}_u$. The restriction of $u$ to $V_i$ is a nilpotent endomorphism
of $V_i$, whose Jordan form is represented by the Young diagram
$Y(u_{|V_i})\subset Y(u)$. Thus, we associate to the flag
$F$ an increasing sequence of subdiagrams $Y(u_{|V_1})\subset
...\subset Y(u_{|V_n})=Y(u)$. Then, we define
$${\mathcal B}_u^T=\{F=(V_0,...,V_n)\in{\mathcal B}_u:Y(u_{|V_i})=Y_i^T\ \forall i=1,...,n\}.$$
By \cite[\S II.5]{Spaltenstein}, the ${\mathcal B}_u^T$'s form a
partition of ${\mathcal B}_u$ into locally closed, irreducible,
nonsingular subsets, and $\dim\,{\mathcal B}_u^T=\dim\,{\mathcal B}_u$
for every $T$. Set
$${\mathcal K}^T=\overline{{\mathcal B}_u^T}$$
the Zariski closure. It follows that ${\mathcal K}^T$ is an
irreducible component of the variety ${\mathcal B}_u$, and each
component of ${\mathcal B}_u$ is obtained in that way. In
particular $\dim\,{\mathcal K}^T=\dim\,{\mathcal B}_u$ for every
$T$.

\subsection{Flags $F_\tau\in{\mathcal B}_u$ and statement of the problem}

\label{section_base_Jordan}

If $Y$ is a Young diagram with $n$ boxes, then we call {\em
row-standard tableau of shape $Y$} a numbering of $Y$ by $1,..,n$
such that numbers in the rows increase to the right.

We parameterize some special elements $F_\tau\in{\mathcal B}_u$
with the row-standard tableaux of shape $Y(u)$. To do this, we fix
a Jordan basis of $u$. Since the lengths of the rows of $Y(u)$
correspond to the sizes of the Jordan blocks of $u$, we can index
the basis $(e_x)_{x\in Y(u)}$ on the boxes of $Y(u)$, in such a
way that the following is true:
\begin{itemize}
\item
$u(e_x)=0$ \ when $x$ lies in the first column of $Y(u)$,
\item
$u(e_x)=e_{x'}$, \ where $x'$ is the box just on the left of $x$,
otherwise.
\end{itemize}
Let $\tau$ be a row-standard tableau of shape $Y(u)$. For
$i=1,...,n$, the numbers $1,...,i$ correspond to boxes $x_1,...,x_i\in
Y(u)$ according to the numbering of $Y(u)$ given by $\tau$. Then
we set $V_i=\langle e_{x_1},...,e_{x_i}\rangle$, and we define
$$F_\tau=(V_0=0,V_1,...,V_n).$$
As $\tau$ is row-increasing, it is clear by construction that
$F_\tau\in {\mathcal B}_u$. Moreover, it is easy to see that the
flags $F_\tau$, for $\tau$ running over the set of row-standard
tableaux of shape $Y(u)$, are all the elements of ${\mathcal B}_u$
which are fixed by the action on flags of the torus of the
automorphisms diagonal in the Jordan basis.

\medskip

In this article, we deal with the following problem: let $\tau$ be
a row-standard tableau and let $T$ be a standard tableau of shape
$Y(u)$, when does the flag $F_\tau$ belong to the component ${\mathcal
K}^T$?

\subsection{Orbits ${\mathcal Z}_\tau\subset{\mathcal B}_u$ and an intrinsic statement of the problem}

\label{propositions} \label{introduction-combinatorics}

The definition of the flags $F_\tau$ depends on the choice of a
Jordan basis of $u$. We relate these flags to a notion of special
orbits which is intrinsic.

Let $\tau$ be a row-standard tableau of shape $Y(u)$. A basis
$(e_i)_{i=1,...,n}$ of $V$ is said to be a {\em $\tau$-basis} if
we have
\begin{itemize}
\item
$u(e_i)=0$ \ when $i$ is in the first column of $\tau$,
\item
$u(e_i)=e_j$, \ where $j$ is the number just on the left of $i$ in
$\tau$, otherwise.
\end{itemize}
For $\underline{e}=(e_i)_{i=1,...,n}$ a $\tau$-basis, we denote by
$F(\underline{e})$ the flag defined by
$$F(\underline{e})=(V_0=0,V_1,...,V_n)$$ with $V_i=\langle
e_1,...,e_i\rangle$. It is clear that
$F(\underline{e})\in{\mathcal B}_u$. Finally we denote by
${\mathcal Z}_\tau$ the set of flags $F(\underline{e})$ for
$\underline{e}$ running over the set of $\tau$-bases.

Let $Z(u)=\{g\in GL(V):gu=ug\}$. The group $Z(u)$ is connected,
and its natural action on flags leaves ${\mathcal B}_u$ and every
component of ${\mathcal B}_u$ stable. Observe that $Z(u)$
acts transitively on the set of $\tau$-bases. It follows
that ${\mathcal Z}_\tau$ is a $Z(u)$-orbit of ${\mathcal
B}_u$. In addition, the flag $F_\tau$ belongs to the set
${\mathcal Z}_\tau$. Indeed, if we write $e_i:=e_{x_i}$ for every
$i=1,...,n$, where $e_{x_i}$ is the basic vector involved in the
definition of $F_\tau$, then we get a $\tau$-basis
$\underline{e}=(e_1,...,e_n)$ such that
$F_\tau=F(\underline{e})$. Thus, we have shown:

\begin{proposition}
\begin{itemize}
\item[(a)] The subset ${\mathcal Z}_\tau\subset {\mathcal B}_u$ is the $Z(u)$-orbit of
the flag $F_\tau$.
\item[(b)] Let ${\mathcal K}^T\subset {\mathcal B}_u$ be an irreducible component. The following conditions are
equivalent: \ (i) ${\mathcal Z}_\tau\cap {\mathcal
K}^T\not=\emptyset$; (ii) ${\mathcal Z}_\tau\subset{\mathcal
K}^T$; (iii) $F_\tau\in{\mathcal K}^T$.
\end{itemize}
\end{proposition}

Let us introduce an equivalence relation on row-standard tableaux:
we write $\tau\equiv\tau'$ if $\tau'$ can be obtained from $\tau$
by switching one or several couples of rows of same length.
Observe that, for $\tau\equiv\tau'$, $\tau$-bases coincide with
$\tau'$-bases, which implies ${\mathcal Z}_\tau={\mathcal
Z}_{\tau'}$. On the other hand, suppose that $i$ is on the left of
$j$ in $\tau$, and let $F=(V_0,...,V_n)\in {\mathcal Z}_\tau$.
Then $i$ is minimal such that $u(V_j)\subset V_i+u(V_{j-1})$. It follows
that ${\mathcal Z}_\tau\cap{\mathcal Z}_{\tau'}=\emptyset$
whenever $\tau\not\equiv \tau'$. We have then proved:

\begin{proposition}
We have ${\mathcal Z}_\tau={\mathcal Z}_{\tau'}$ if and only if
$\tau\equiv \tau'$.
\end{proposition}

\begin{remark}
In general, the ${\mathcal Z}_\tau$'s are not all the
$Z(u)$-orbits of ${\mathcal B}_u$. Suppose for example
$\displaystyle{Y(u)=\mbox{\scriptsize $\yng(3,1)$}}$, let
$(e,e',e'',f)$ with $u(e'')=e'$, $u(e')=e$, $u(e)=u(f)=0$ be a
Jordan basis, then the flag $F=(0\subset\langle
e\rangle\subset\langle e,e'+f\rangle\subset\langle
e,e',f\rangle\subset V)$ does not belong to any ${\mathcal
Z}_\tau$. Nevertheless, the ${\mathcal Z}_\tau$'s are all the
$Z(u)$-orbits of ${\mathcal B}_u$ when the diagram $Y(u)$ has two
columns (see \cite{FM}).
\end{remark}

\medskip

Let $(\tau,T)$ be a pair formed by a row-standard tableau and a
standard tableau of shape $Y(u)$. Observe that the property that
the flag $F_\tau$ belongs to the component ${\mathcal K}^T$ (or
equivalently that the orbit ${\mathcal Z}_\tau$ lies in ${\mathcal
K}^T$) does not depend on $u$. Indeed, if $u':V'\rightarrow V'$ is
a nilpotent endomorphism of same Jordan form as $u$, then we have
$u'=gug^{-1}$ for some linear isomorphism $g:V\rightarrow V'$, and
$g$ induces an isomorphism ${\mathcal B}_u\rightarrow {\mathcal
B}_{u'}$ which sends the component of ${\mathcal B}_u$ associated
to $T$ to the component of ${\mathcal B}_{u'}$ associated to $T$
on one hand, and makes correspond the $\tau$-bases of $V$ with the
$\tau$-bases of $V'$ on the other hand.

Let $\mathrm{sh}(\tau)$ and $\mathrm{sh}(T)$ denote the shapes of
$\tau$ row-standard or $T$ standard, and let $|\tau|$ (resp.
$|T|$) denote the number of boxes in $\tau$ (resp. $T$). Define
$\mathbf{Y}$ to be the set of pairs $(\tau,T)$ formed by $\tau$
row-standard and $T$ standard, with
$\mathrm{sh}(\tau)=\mathrm{sh}(T)$. For $(\tau,T)\in\mathbf{Y}$,
we write $\tau\in T$ if the flag $F_\tau$ belongs to the component
${\mathcal K}^T$ in the Springer fiber ${\mathcal B}_u$ over any
$u$ nilpotent such that $Y(u)=\mathrm{sh}(T)$. Define $\mathbf{K}$
to be the subset of pairs $(\tau,T)\in\mathbf{Y}$ such that
$\tau\in T$.

\medskip

Thus, the main purpose of this article is to study the
combinatorial properties of the set $\mathbf{K}$.

\subsection{The hook, two-row and two-column cases}

A Young diagram $Y$ is said to be of {\em hook type} if it
contains at most one row of length $\geq 2$. It is said to be of
{\em two-row type} if it contains at most two rows. It is said to
be of {\em two-column type} if it contains at most two columns. A
nilpotent endomorphism $u\in \mathrm{End}(V)$ is said to be of
{\em hook} (resp. {\em two-row}) (resp. {\em two-column}) type,
according to its Young diagram $Y(u)$. The most often in this
article, we specialize in these three cases. One common
specificity, which makes the study easier in these cases, is that
the dominance order on diagrams is here linear. Moreover, we point
out some special properties held by Springer fibers in these three
cases (cf. Proposition \ref{proposition-counterexample} and the
remarks in section \ref{last-section}).

Up to now, the description of the components of Springer fibers
has essentially been confined to these three cases.
For instance, in each case, an equational characterization of the components
has been provided (see \cite{F}, \cite{Fung}, \cite{Vargas}) and
sometimes we shall rely on these previous studies.
Nevertheless, these equations are a bit cumbersome and make
each case be very specific.
The aim in studying the components under the point of view of the belonging of the special orbits ${\mathcal Z}_\tau$,
is to get more convenient and unified combinatorial descriptions in the three cases.

\subsection{Notation}

The notation of the previous subsections is kept. Throughout this
article, the flags are written $(V_0,...,V_n)$ or
$(V_0\subset...\subset V_n)$ or $F$. The standard tableaux are
usually written $T,T',S,...$ The row-standard tableaux are written
$\tau,\tau',...$ We will write interchangeably $\tau\in T$ or
$(\tau,T)\in\mathbf{K}$. We denote by $\# A$ the number of
elements in a finite set $A$. The reader can find an index of
notation at the end of this article.


\section{Some basic properties of the set $\mathbf{K}$, and connection to the problem
of intersections of components}

\label{2}

First in this section, we introduce the standardization
$\mathrm{st}(\tau)$ of a row-standard tableau $\tau$ (i.e. the
standard tableau obtained by putting the entries in each column of
$\tau$ in right order). Then, we point out quite natural
properties involving $\mathrm{st}(\tau)$, for instance that the flag $F_\tau$ belongs
to the set ${\mathcal B}_u^{\mathrm{st}(\tau)}$, and thus to the component ${\mathcal K}^{\mathrm{st}(\tau)}$ (hence
$\tau\in \mathrm{st}(\tau)$, or equivalently
$(\tau,\mathrm{st}(\tau))\in\mathbf{K}$). Of course, it does not
solve our initial problem, because there can also exist
$T\not=\mathrm{st}(\tau)$ such that $\tau\in T$. We prove also: $\tau\in T$ $\Rightarrow$
$\mathrm{st}(\tau)\in T$. Finally, we show
that two given components ${\mathcal K}^T,{\mathcal K}^{T'}$ have
a nonempty intersection if and only if there is a common $\tau$
which is standard such that $\tau\in T$ and $\tau\in T'$.

\subsection{Standardization of a row-standard tableau}

\label{standardization}

Let $\tau$ be a row-standard tableau. Arrange the
entries in each column of $\tau$ in increasing order to the
bottom. Then, we see that the numbers in the rows of the tableau
so-obtained remain increasing to the right. In other words, this
tableau is standard. We denote it by $\mathrm{st}(\tau)$.
For example
\[
\tau=\young(348,167,25) \qquad
\mathrm{st}(\tau)=\young(147,258,36)
\]

Embed ${\mathcal Z}_\tau\subset {\mathcal B}_u$. Recall the
partition ${\mathcal B}_u=\bigsqcup_{T}{\mathcal B}_u^T$
parameterized by standard tableaux, introduced in section
\ref{young-diagram}.

\begin{lemma}
\label{lemma-standardization} We have ${\mathcal Z}_\tau\subset
{\mathcal B}_u^T$ for $T=\mathrm{st}(\tau)$.
\end{lemma}

\begin{proof}
Let $F=(V_0,...,V_n)\in {\mathcal Z}_\tau$, and let
$\underline{e}=(e_1,...,e_n)$ be a $\tau$-basis such that
$F=F(\underline{e})$. To prove $F\in{\mathcal B}_u^T$, fix
$i\in\{1,...,n\}$ and let us show $Y(u_{|V_i})=Y_i^T$.

Let $q\geq 1$. The number of boxes in the first $q$ columns of
$Y_i^T$ is equal to the number of entries $j\leq i$ in the first
$q$ columns of $T$, which coincide with the entries $j\leq i$ in
the first $q$ columns of $\tau$ because $T=\mathrm{st}(\tau)$. The
lengths of the rows of the diagram $Y(u_{|V_i})$ are the sizes of
the Jordan blocks of $u_{|V_i}$, it follows that the number of
boxes in the first $q$ columns of $Y(u_{|V_i})$ is equal to
$\mathrm{dim}\,\mathrm{ker}\,(u_{|V_i})^q$. Observe that the
subspace $\mathrm{ker}\,(u_{|V_i})^q$ is generated by the basic
vectors $e_j$ associated to the entries $j\leq i$ in the first $q$
columns of $\tau$. Thus both diagrams $Y(u_{|V_i})$ and $Y_i^T$
contain the same number of boxes in their first $q$ columns, for any $q$.
Therefore, we have $Y(u_{|V_i})=Y_i^T$ for any $i$.  
\end{proof}

As the component ${\mathcal K}^T\subset {\mathcal B}_u$ is the
closure of the subset ${\mathcal B}_u^T$, we deduce the following

\begin{proposition}
\label{proposition-standardization} We have
$(\tau,\mathrm{st}(\tau))\in\mathbf{K}$, for any $\tau$
row-standard.
\end{proposition}

Next, consider a pair $(\tau,T)\in\mathbf{Y}$. As above, to the
row-standard tableau $\tau$, we associate the standard tableau
$\mathrm{st}(\tau)$. The tableau $\mathrm{st}(\tau)$ is a fortiori
row-standard, and we may consider the pair
$(\mathrm{st}(\tau),T)\in\mathbf{Y}$. Then we show the following

\begin{proposition}
\label{proposition-rectify-standard} $(\tau,T)\in \mathbf{K}$
$\Rightarrow$ $(\mathrm{st}(\tau),T)\in \mathbf{K}$.
\end{proposition}

\begin{proof}
Write for simplicity $S=\mathrm{st}(\tau)$. Embed ${\mathcal
Z}_\tau,{\mathcal Z}_S,{\mathcal K}^T$ in ${\mathcal B}_u$ for
some $u:V\rightarrow V$ nilpotent. To prove the proposition, it is
sufficient to show ${\mathcal Z}_S\subset \overline{{\mathcal
Z}_\tau}$.

Let us consider the elementary operation to arrange in relative
order two rows $p<q$ of $\tau$: let $a_1<...<a_r$ (resp.
$b_1<...<b_s$) be the entries of the $p$-th row (resp. $q$-th row)
of the tableau $\tau$. For $i=1,...,s$ write $\tilde
a_i=\min(a_i,b_i)$ and $\tilde b_i=\max(a_i,b_i)$. For
$i=s+1,...,r$ write $\tilde a_i=a_i$. Then let $\tilde\tau$ be the
tableau obtained by replacing $a_i$ by $\tilde a_i$ for
$i=1,...,r$ and $b_i$ by $\tilde b_i$ for $i=1,...,s$.
{\tiny\[
\tau=\begin{array}{|c|c|c|c|c|} \hline
\multicolumn{5}{|c|}{\qquad\vdots\qquad } \\
\hline
a_1 & a_2 & \multicolumn{2}{|c|}{\cdots} & a_r \\
\hline
\multicolumn{5}{|c|}{\qquad \vdots\qquad } \\
\hline
b_1 & b_2 & \cdots & b_s \\
\cline{1-4}
\multicolumn{4}{|c|}{\quad \vdots\quad } \\
\cline{1-4}
\end{array}
\ \ \longrightarrow\ \ \tilde\tau=\begin{array}{|c|c|c|c|c|}
\hline
\multicolumn{5}{|c|}{\qquad\vdots\qquad } \\
\hline
\tilde a_1 & \tilde a_2 & \multicolumn{2}{|c|}{\cdots} & \tilde a_r \\
\hline
\multicolumn{5}{|c|}{\qquad \vdots\qquad } \\
\hline
\tilde b_1 & \tilde b_2 & \cdots & \tilde b_s \\
\cline{1-4}
\multicolumn{4}{|c|}{\quad \vdots\quad } \\
\cline{1-4}
\end{array}
\]}
It is clear that this tableau remains row-standard. Observe that
$\mathrm{st}(\tilde \tau)=\mathrm{st}(\tau)=S$. Moreover, the
tableau $S$ is obtained from $\tau$ after a sequence of operations
as $\tau\mapsto \tilde\tau$, hence it is sufficient to show the
inclusion ${\mathcal Z}_{\tilde\tau}\subset \overline{{\mathcal
Z}_\tau}$.

Let $F\in{\mathcal Z}_{\tau}$, and let
$\underline{e}=(e_1,...,e_n)$ be a $\tau$-basis such that
$F=F(\underline{e})$. For $i\notin\{a_1,...,a_r,b_1,...,b_s\}$ set
$\tilde e_i=e_i$. Set in addition $\tilde e_{a_i}=e_{\tilde a_i}$
for $i=1,...,r$, and $\tilde e_{b_i}=e_{\tilde b_i}$ for
$i=1,...,s$. Then $\underline{\tilde e}:=(\tilde e_1,...,\tilde
e_n)$ is a $\tilde\tau$-basis, and consequently $\tilde
F:=F(\underline{\tilde e})$ belongs to ${\mathcal
Z}_{\tilde\tau}$. For $t\in\mathbb{C}$ let $d_t:V\rightarrow V$ be
the automorphism defined by $d_t(e_i)=e_i$ for
$i\notin\{b_1,...,b_s\}$, and $d_t(e_{b_i})=e_{b_i}+te_{a_i}$.
Then $(d_t)_{t\in\mathbb{C}}$ is a one-parameter subgroup of
$Z(u)=\{g\in GL(V):gu=ug\}$. Recall that ${\mathcal Z}_\tau$ is a
$Z(u)$-orbit. Thus the curve $\{d_tF:t\in\mathbb{C}\}$ lies in
${\mathcal Z}_\tau$. 
For all $i=1,\ldots,n$, we have
\[
d_t\langle e_1,...,e_i\rangle = \langle e_j:j\leq i,\ j\notin\{b_1,...,b_s\};\ e_{b_l}:a_l<b_l\leq i;\ e_{b_l}+te_{a_l}:b_l\leq i<a_l\rangle,
\]
whence
$\lim_{t\rightarrow\infty} d_t\langle e_1,...,e_i\rangle = \langle \tilde e_j:1\leq j\leq i\rangle$.
It follows $\tilde F=\lim_{t\rightarrow\infty}d_tF\in\overline{{\mathcal Z}_\tau}$.
Therefore, ${\mathcal Z}_{\tilde\tau}=Z(u)\tilde F\subset \overline{{\mathcal Z}_\tau}$.
\end{proof}

\begin{remark}
The converse of Proposition \ref{proposition-rectify-standard} is not true.
For instance, if 
\[\tau=\young(23,1)\quad\mathrm{st}(\tau)=\young(13,2)\quad T=\young(12,3)\,,\]
then $(\tau,T)\notin\mathbf{K}$, whereas $(\mathrm{st}(\tau),T)\in\mathbf{K}$.
This can be seen by applying the criteria which we provide in this paper,
or more directly by computing that the component ${\mathcal K}^T$
is in this case the set of flags $(0\subset V_1\subset V_2\subset V_3=\mathbb{C}^3)$
with $V_1=\mathrm{Im}\,u$.
\end{remark}

\subsection{Connection to the problem of intersections of components}

\label{criterion-nonempty-intersection}

We connect the problem to determine pairs $(S,T)\in\mathbf{K}$
with $T,S$ both standard to the problem to determine nonempty
intersections of components of the Springer fiber. Consider the
Springer fiber ${\mathcal B}_u$ and an irreducible component
${\mathcal K}^T\subset {\mathcal B}_u$ associated to the standard
tableau $T$. Let $S$ be a standard tableau of same shape as $T$,
and let ${\mathcal B}_u^S\subset {\mathcal B}_u$ be the
Spaltenstein subset corresponding to $S$ (cf.
\ref{young-diagram}). The following lemma will also be used in
section \ref{8}.

\begin{lemma}
\label{lemma-criterion-nonempty-intersection} ${\mathcal
B}_u^S\cap {\mathcal K}^T\not=\emptyset$ $\Leftrightarrow$
$(S,T)\in\mathbf{K}$.
\end{lemma}

\begin{proof} The implication $(\Leftarrow)$ follows from Lemma \ref{lemma-standardization}.
The proof of the other implication relies on the following
construction.

Let $Y=Y(u)$ be the Young diagram representing the Jordan form
of $u$ (see \ref{young-diagram}). As in section
\ref{section_base_Jordan}, we consider a Jordan basis $(e_x)_{x\in
Y}$ indexed on the boxes of $Y$, with $u(e_x)=0$ for $x$ in the first
column of $Y$, and $u(e_x)=e_{x'}$, where $x'$ is the box on the
left of $x$, otherwise. Let $H\subset GL(V)$ be the subgroup of
automorphisms which are diagonal in the basis $(e_x)_{x\in Y}$.
The flags $F_\tau$, parameterized by row-standard
tableaux of shape $Y$, introduced in section
\ref{section_base_Jordan}, are exactly the
elements of the Springer fiber ${\mathcal B}_u$ which are fixed by
the natural action of $H$ on flags. However, the action of $H$ on
flags does not leave ${\mathcal B}_u$ stable.

We construct a one-parameter subgroup $H'\subset H$ which leaves
${\mathcal B}_u$ stable. To do this, we write $x_{p,q}$ the
$p$-th box of the $q$-th column of $Y$. Let $e_{p,q}=e_{x_{p,q}}$.
Thus
\begin{itemize}
\item $u(e_{p,q})=0$ for $q=1$,
\item $u(e_{p,q})=e_{p,q-1}$ otherwise.
\end{itemize}
Set $\epsilon_{p,q}=nq-p$. Then, we define $h_t:V\rightarrow V$ by
setting $h_t(e_{p,q})=t^{\epsilon_{p,q}}e_{p,q}$. Thus
$H':=(h_t)_{t\in\mathbb{C}^*}$ is a one-parameter subgroup of $H$.
One easily checks that $h_t(u(e_{p,q}))=t^{-n}u(h_t(e_{p,q}))$ for
any $t$, hence the action of $H'$ leaves ${\mathcal B}_u$ and
every component of ${\mathcal B}_u$ stable. Moreover, as the
$\epsilon_{p,q}$'s are pairwise distinct, the flags $F_\tau$ are
exactly the fixed points of ${\mathcal B}_u$ for the action of
$H'$.

Let $F\in{\mathcal K}^T$. Then for any $t\in{\mathbb
C}^*$ we have $h_tF\in{\mathcal K}^T$. It follows
$\lim_{t\rightarrow \infty}h_tF\in{\mathcal K}^T$. Moreover, the
limit $\lim_{t\rightarrow \infty}h_tF$ is necessarily a fixed
point of $H'$, hence $\lim_{t\rightarrow \infty}h_tF=F_\tau$ for
some row-standard tableau $\tau$. 
To show $(\Rightarrow)$, it is now
sufficient to show:
$$F\in{\mathcal B}_u^S\ \Rightarrow\ F_\tau:=\lim_{t\rightarrow \infty}h_tF\in{\mathcal B}_u^S.$$
Indeed, we then get that, if ${\mathcal B}_u^S\cap{\mathcal K}^T$ is nonempty, it contains an
element of the form $F_\tau$. So it follows from Lemma \ref{lemma-standardization}
that $S=\mathrm{st}(\tau)$, and then from Proposition \ref{proposition-rectify-standard} that
$(S,T)\in\mathbf{K}$.

Write $\{(p,q):$ $p=1,...,r$,
$q=1,...,\lambda_p\}=:\{(p_i,q_i):i=1,...,n\}$ so that we have
$\epsilon_{p_1,q_1}<...<\epsilon_{p_n,q_n}$. Write
$e_i=e_{p_i,q_i}$. Let $B\subset GL(V)$ 
be the Borel subgroup of upper triangular automorphisms. in the basis $(e_1,...,e_n)$.
Recall that ${\mathcal B}$ denotes the variety of complete flags
on $V$. Let $S(\tau)\subset {\mathcal B}$ be the $B$-orbit of
$F_\tau$ in ${\mathcal B}$. This is a Schubert cell. Then, it is a
classical fact that
\[S(\tau)=\{F\in{\mathcal B}:\lim_{t\rightarrow \infty}h_tF=F_\tau\}.\]
Let $P=\{g\in GL(V):g(\mathrm{ker}\,u^q)=\mathrm{ker}\,u^q\
\forall q\}$. This is a parabolic subgroup of $GL(V)$. Notice that
${\mathcal B}_u^S$ is the intersection between ${\mathcal B}_u$
and a $P$-orbit of the flag variety, namely: 
\[
{\mathcal P}:=\left\{(V_0,\ldots,V_n)\in{\mathcal B}:
\dim\,V_i\cap\ker\,u^q=c_{i,q}^S\ \ \forall i=1,...,n,\ \forall q\right\},
\]
where $c_{i,q}^S$ denotes the number of entries $j\leq i$ in the
first $q$ columns of $S$ (cf. the proof of Lemma
\ref{lemma-standardization}). Observe that $B\subset P$. Thus,
$F\in {\mathcal B}_u^S$ implies
\[
\lim_{t\rightarrow \infty}h_tF\in(BF)\cap {\mathcal
B}_u\subset{\mathcal B}_u^S.
\]
The proof is now complete.  
\end{proof}

We deduce the following

\begin{proposition}
\label{proposition-nonempty-intersection}
Let ${\mathcal K}^T,{\mathcal K}^{T'}\subset {\mathcal B}_u$ be
two components corresponding to standard tableaux $T$ and $T'$. We
have ${\mathcal K}^T\cap {\mathcal K}^{T'}\not=\emptyset$ if and
only if there is $S$ standard such that $S\in T$ and $S\in T'$.
\end{proposition}

\begin{proof}
As ${\mathcal B}_u$ is the union of the ${\mathcal B}_u^S$'s, the
intersection ${\mathcal K}^T\cap {\mathcal K}^{T'}$ is nonempty if
and only if there is some $S$ standard such that ${\mathcal B}_u^S\cap {\mathcal K}^T\cap {\mathcal K}^{T'}\not=\emptyset$.
By the previous lemma, it is equivalent to have
$(S,T),(S,T')\in\mathbf{K}$.  
\end{proof}

\section{An inductive property of the set $\mathbf{K}$ and stability by the Sch\"utzenberger involution}

\label{3}

In this section, we show two combinatorial properties of the set
$\mathbf{K}$ of pairs $(\tau,T)$ such that $\tau\in T$. First, in
the case where $\tau',T'$ are respective subtableaux of $\tau,T$
of same shape, we relate the belonging of $(\tau',T')$ in
$\mathbf{K}$ to the belonging of $(\tau,T)$. Next, we show that
the set $\mathbf{K}$ is stable by a combinatorial transformation
$(\tau,T)\mapsto (S\cdot \tau,T^S)$ involving the Sch\"utzenberger
involution $T\mapsto T^S$.

\subsection{Inductive property}

\label{subsection-inductive-property}

Let $T,T'$ be standard tableaux. We write $T'\subseteq T$ if $T'$
is a subtableau of $T$, i.e. is obtained from $T$ by deleting
entries $i,...,|T|$ for some $i\geq 1$. More generally, let
$\tau,\tau'$ be row-standard tableaux. We write $\tau'\subseteq
\tau$ if $|\tau'|\leq|\tau|$, and $\tau$ and $\tau'$ satisfy the
following property: if $i,j$ are entries of $\tau'$, then they lie in the same row of $\tau'$
if and only if they lie in the same row of $\tau$. For example:
$$\young(234,1)\ \subseteq\ \young(167,234,58)$$

We have the following

\begin{theorem}
\label{theorem-induction} Let $(\tau,T),(\tau',T')\in\mathbf{Y}$
and we suppose $\tau'\subseteq \tau$ and $T'\subseteq T$.
\begin{itemize}
\item[(a)] $(\tau,T)\in\mathbf{K}\Rightarrow (\tau',T')\in\mathbf{K}$.
\item[(b)] Moreover, if every $i\in\{|T'|+1,...,|T|\}$ lies in the same
column in $\tau$ and $T$, then
$(\tau,T)\in\mathbf{K}\Leftrightarrow (\tau',T')\in\mathbf{K}$.
\end{itemize}
\end{theorem}

\begin{proof}[Proof of Theorem \ref{theorem-induction}.]
We first set the notation. 
Let $n=|T|$ and let $V=\mathbb{C}^n$.
The component ${\mathcal K}^T$ is the closure of the subset ${\mathcal B}_u^T$ 
in the Springer fiber ${\mathcal B}_u$, for $u\in\mathrm{End}(V)$ nilpotent of appropriate Jordan form.
Let $\underline{e}=(e_1,\ldots,e_n)$ be a $\tau$-basis and let
$F_\tau=(\langle e_1,\ldots,e_i\rangle)_{i=0,\ldots,n}\in{\mathcal Z}_\tau\subset {\mathcal
B}_u$ be the corresponding adapted flag.
Denote $n'=|T'|$ and $V'=\langle e_1,\ldots,e_{n'}\rangle$.
Let $u'=u_{|V'}\in\mathrm{End}(V')$ be the restriction of $u$ to $V'$
and let ${\mathcal B}_{u'}$ be the variety of $u'$-stable, complete flags of $V'$.
The component ${\mathcal K}^{T'}$ associated to $T'$ is the closure of the subset
${\mathcal B}_{u'}^{T'}\subset {\mathcal B}_{u'}$. 
Moreover, the flag $F_{\tau'}:=(\langle e_1,\ldots,e_i\rangle)_{i=0,\ldots,n'}$ belongs to
the orbit ${\mathcal Z}_{\tau'}\subset{\mathcal B}_{u'}$.

\smallskip

Let us show part (a) of the theorem: 
we suppose that $F_{\tau}\in{\mathcal K}^{T}$
and we have to show that $F_{\tau'}\in {\mathcal K}^{T'}$.
We let $Y'=\mathrm{sh}(T')$ and $V''=\langle e_{n'+1},\ldots,e_n\rangle$.
Let ${\mathcal G}_{n'}(V)$ be the variety of $n'$-dimensional
subspaces of $V$. Let ${\mathcal G}_{Y'}(V)\subset {\mathcal G}_{n'}(V)$ be the subset
of subspaces $W$ which are
stable by $u$ and such that $Y(u_{|W})=Y'$.
This is a locally closed subvariety of ${\mathcal G}_{n'}(V)$,
which contains $V'$, and an open neighborhood of $V'$ in ${\mathcal G}_{Y'}(V)$
is formed by $\Omega=\{W\in {\mathcal G}_{Y'}(V):W\cap V''=0\}$.
We rely on the following

\medskip
\noindent
{\bf Claim.}
{\em There is an algebraic map $\beta:\Omega\rightarrow GL(V)$,
$W\mapsto \beta_W$ with the properties:
\begin{itemize}
\item[(1)] $\beta_W(W)=V'$ and $\beta_W\circ u(x)=u\circ \beta_W(x)$ for all $x\in W$;
\item[(2)] $\beta_{V'}(x)=x$ for all $x\in V'$.
\end{itemize}}

\smallskip
\noindent
Assume the claim is true. Let ${\mathcal O}=\{(V_0,\ldots,V_n)\in {\mathcal B}_u:V_{n'}\in\Omega\}$. Then, the map
\[\Phi:{\mathcal O}\rightarrow {\mathcal B}_{u'},\ \ 
(V_0,\ldots,V_n)\mapsto(\beta_{V_{n'}}(V_0),\ldots,\beta_{V_{n'}}(V_{n'}))\]
is well defined and algebraic.
On one hand, we see that
$\Phi({\mathcal O}\cap {\mathcal B}_u^T)\subset {\mathcal B}_{u'}^{T'}\subset {\mathcal K}^{T'}$.
Note that ${\mathcal O}\cap {\mathcal B}_u^T$ is open in ${\mathcal B}_u^T$, hence it is a dense subset
of the component ${\mathcal K}^T$. Thus, $\Phi({\mathcal O}\cap{\mathcal K}^T)\subset {\mathcal K}^{T'}$.
On the other hand, $\Phi(F_{\tau})=F_{\tau'}$.
Therefore, we obtain $F_{\tau'}\in {\mathcal K}^{T'}$. It remains to show the claim, to complete
the proof of part (a).

\medskip

\noindent
{\em Proof of the claim.}
For a subset $S\subset V$, let $\langle S\rangle_u=\langle u^l(x):x\in S,\ l\geq 0\rangle$
be the smallest $u$-stable subspace containing $S$.
Thus, the Jordan block decomposition of $V$ can be written
$V=\bigoplus_{j=1}^r\langle e_{i_j}\rangle_u$ where $1\leq i_1,\ldots,i_r\leq n$ are pairwise distinct.
In turn, the Jordan block decomposition of $V'$ can be written
$V'=\bigoplus_{j=1}^k\langle e_{i'_j}\rangle_u$ with $1\leq i'_1,\ldots,i'_k\leq n'$.
We may suppose that $\langle e_{i'_j}\rangle_u\subset \langle e_{i_j}\rangle_u$ for all $j$,
and, denoting $l_j=\dim \langle e_{i'_j}\rangle_u$, that we have $l_1\geq\ldots\geq l_k$.
Notice that the numbers $l_1,\ldots,l_k$ are the sizes of the rows of the diagram $Y'$.

Suppose that, for each $j\in\{1,\ldots,k\}$, we are given an element $f_j\in V''\cap\ker u^{l_j}$.
This implies 
\begin{equation}
\label{formula-unit}
f_j\in\langle e_{i_{j+1}}\rangle_u\oplus\ldots\oplus\langle e_{i_r}\rangle_u
\end{equation} 
hence
the subspaces $\langle e_{i'_j}+f_j\rangle_u$ (for $j=1,\ldots,k$)
are in direct sum.
Also we have $\dim \langle e_{i'_j}+f_j\rangle_u=l_j$ for all $j$, thus
$W(f_1,\ldots,f_k):=\langle e_{i'_j}+f_j:j=1,\ldots,k\rangle_u\in \Omega$.

Now, we fix $W\in \Omega$. For each $j\in \{1,\ldots,k\}$, there is a unique $f_j\in V''$ 
(depending algebraically on $W$)
such that $e_{i'_j}+f_j\in W$.
Let us show that
\begin{equation}
\label{formula-noyau}
f_j\in \ker u^{l_j}\ \ \mbox{for all $j\in\{1,\ldots,k\}$}.
\end{equation}
Suppose by contradiction that there is $j$ such that $f_j\notin \ker u^{l_j}$. We take $j$ minimal, so that
$f_1,\ldots,f_{j-1}$ satisfy (\ref{formula-unit}). Let $j'\leq j$ be minimal such that
$f_j\notin \ker u^{l_{j'}}$.
The subspaces $\langle e_{i'_t}+f_t\rangle_u$, for $t\in\{1,\ldots,j'-1,j\}$, are in direct sum,
and all have dimension bigger than $l_{j'}$. Hence $u_{|W}$ has $j'$ Jordan blocks of length bigger than $l_{j'}$.
This contradicts the fact that $Y(u_{|W})=Y'$. Therefore, we have shown (\ref{formula-noyau}).

From (\ref{formula-noyau}), we derive that $W=W(f_1,\ldots,f_k)$.

Then, the vectors $u^l(e_{i'_j})$ and $u^l(e_{i'_j}+f_j)$, for $j=1,\ldots,k$, $l=0,\ldots,l_j-1$, form bases of $V'$
and $W$ respectively. Each one of both is completed into a basis of $V$ by adding
the vectors $e_{n'+1},\ldots,e_n$.
The change of basis isomorphism from the first basis to the second one algebraically depends on $f_1,\ldots,f_k$,
and moreover, using (\ref{formula-unit}), we see that it has determinant $1$. Therefore, also the inverse change of basis isomorphism
algebraically depends on $f_1,\ldots,f_k$, and this is the desired $\beta_W$.
The claim is proved.

\medskip

Finally, let us show part (b) of the theorem. 
It remains to show the implication $(\Leftarrow)$:
we suppose that $F_{\tau'}\in{\mathcal K}^{T'}$ and we have to show that $F_{\tau}\in{\mathcal K}^{T}$.
Let us write $V_i=\langle e_1,\ldots,e_i\rangle$, so that 
$F_\tau=(V_0,\ldots,V_n)$ and $F_{\tau'}=(V_0,\ldots,V_{n'})$.
The map
\[\Psi:{\mathcal B}_{u'}\rightarrow {\mathcal B}_u,\ \ 
(W_0,\ldots,W_{n'})\mapsto(W_0,\ldots,W_{n'},V_{n'+1},\ldots,V_n),\]
is clearly well defined and algebraic.
It is assumed that every $i\in\{n'+1,\ldots,n\}$ is in the same column in $\tau$ and $T$,
this 
easily implies that the Young diagram $Y(u_{|V_i})$ coincides with $Y_i^T$
for all $i=n'+1,\ldots,n$. It follows that $\Psi({\mathcal B}_{u'}^{T'})\subset {\mathcal B}_{u}^{T}$.
Therefore, $\Psi({\mathcal K}^{T'})\subset {\mathcal K}^{T}$, and we derive that
$F_\tau=\Psi(F_{\tau'})\in {\mathcal K}^{T}$.
The proof of the theorem is now complete.
\end{proof}

\subsection{Sch\"utzenberger involution}

\label{section-schutzenberger}

We show that the set $\mathbf{K}$ is stabilized by a combinatorial
operation $(\tau,T)\mapsto (S\cdot\tau,T^S)$ involving the
Sch\"utzenberger transform $T^S$ of $T$.

Let $T$ be a standard tableau, and let $Y=\mathrm{sh}(T)$ be its
shape. For $i=0,...,n$, let $T[i+1,...,n]$ be the skew subtableau
of entries $i+1,...,n$. We refer to \cite{Fulton} for the
definition of jeu de taquin. Jeu de taquin transforms
$T[i+1,...,n]$ into a Young tableau whose shape is a subdiagram
$Y_{n/i}^T\subset Y$. For example:
$$
T=\young(134,257,6)\qquad
T[4,...,7]=\young(::4,:57,6)\stackrel{\mbox{\tiny
jeu\,de\,taquin}}{\longrightarrow}\young(47,5,6)\qquad
Y_{7/3}^T=\yng(2,1,1)
$$
In that way, we obtain an increasing sequence of subdiagrams
\[\emptyset\subset Y_{n/n-1}^T\subset \ldots\subset Y_{n/0}^T=Y.\]
Let $T^S$ be the standard tableau of shape $Y$ obtained by putting
$i$ in the new box of $Y_{n/n-i}^T$ compared to $Y_{n/n-i+1}^T$.
The tableau $T^S$ is called the {\em Sch\"utzenberger transform of
$T$}. The map $T\mapsto T^S$ so-obtained is an involution.
For example
$$
T=\young(134,257,6)\ \ \longrightarrow\ \ T^S=\young(126,357,4)
$$

The Sch\"utzenberger transform $T^S$ has the following
interpretation in terms of components of Springer fibers, which we
recall from \cite{vanLeeuwen}. For $F=(V_0,...,V_n)\in{\mathcal
B}_u$, let $u_{|V_n/V_i}\in\mathrm{End}(V_n/V_i)$ be the nilpotent
endomorphism induced by $u$, and let $Y(u_{|V_n/V_i})$ be the
Young diagram which represents its Jordan form in the sense of
section \ref{young-diagram}. Let $Y_i^T$ be the shape of the subtableau of $T$ with entries
$1,...,i$. Define
\[{\mathcal B}_{u,T}=\{F=(V_0,...,V_n)\in{\mathcal B}_u:Y(u_{|V_n/V_{n-i}})=Y_i^T\ \forall i=1,...,n\}.\]
The ${\mathcal B}_{u,T}$'s form a partition of the variety
${\mathcal B}_u$. By \cite[Theorem 3.3]{vanLeeuwen}, the closure
of ${\mathcal B}_{u,T}$ is the irreducible component ${\mathcal
K}^{T^S}\subset{\mathcal B}_u$ associated to the tableau $T^S$.

\smallskip

Now, let $\tau$ be a row-standard tableau. Let $S\cdot \tau$ denote the
row-standard tableau constructed as follows. First, replace each
entry $i$ in $\tau$ by $n-i+1$. Then, reverse the order of the
entries inside each row, so that the tableau so-obtained becomes
row-standard. The map $\tau\mapsto S\cdot \tau$ is clearly an
involution. For example:
$$\tau=\young(347,156,2)\ \ \longrightarrow\ \ S\cdot\tau=\young(145,237,6)$$
We get finally an involutive map
$\mathbf{Y}\rightarrow\mathbf{Y}$, $(\tau,T)\mapsto (S\cdot
\tau,T^S)$. We show:

\begin{proposition}
\label{proposition-Schutzenberger} We have $(\tau,T)\in
\mathbf{K}$ if and only if $(S\cdot\tau,T^S)\in\mathbf{K}$.
\end{proposition}

\begin{proof}
Embed the component ${\mathcal K}^T$ and the orbit ${\mathcal
Z}_\tau$ in the Springer fiber ${\mathcal B}_u$ corresponding to
some nilpotent $u\in\mathrm{End}(V)$. Let $V^*$ be the dual vector
space of $V$ and let $u^*\in\mathrm{End}(V^*)$ be the dual
(nilpotent) endomorphism. Let ${\mathcal B}_{u^*}$ be the
corresponding Springer fiber. Both $u$ and $u^*$ have the same
Jordan shape, hence ${\mathcal B}_u$ and ${\mathcal B}_{u^*}$ are
isomorphic. Let ${{\mathcal K}^*}^{T^S}\subset {\mathcal B}_{u^*}$
be the component of ${\mathcal B}_{u^*}$ associated to $T^S$, and
let ${\mathcal Z}^*_{S\cdot\tau}\subset {\mathcal B}_{u^*}$ be the
$Z(u^*)$-orbit of ${\mathcal B}_{u^*}$ associated to $S\cdot\tau$. For a
subspace $W\subset V$, write $W^\perp=\{\phi\in
V^*:\phi(w)=0\ \forall w\in W\}$. Let $\Phi:{\mathcal
B}_u\rightarrow {\mathcal B}_{u^*}$ be the map which sends
$(V_0,...,V_n)$ to $(V_n^\perp,...,V_0^\perp)$. It is an
involutive isomorphism. 
On one hand, notice that whenever $W\subset V$ is a $u$-stable subspace, the restriction
$u_{|W}\in\mathrm{End}(W)$ and the map $(u^*)_{|V^*/W^\perp}\in\mathrm{End}(V^*/W^\perp)$ induced by $u^*$
have the same Jordan form.
This implies that $\Phi({\mathcal B}_u^T)={\mathcal B}_{u^*,T}$,
whence the equality $\Phi({\mathcal K}^T)={{\mathcal K}^*}^{T^S}$. On the other hand,
let $\underline{e}=(e_1,...,e_n)$ be a $\tau$-basis of $V$ and let
$\underline{e^*}=(e_1^*,...,e_n^*)$ be its dual basis. Then
$\underline{{e^*}'}=(e_n^*,...,e_1^*)$ is a $(S\cdot\tau)$-basis
of $V^*$ and we see that
$\Phi(F(\underline{e}))=F(\underline{{e^*}'})$. It follows
$\Phi({\mathcal Z}_\tau)={\mathcal Z}^*_{S\cdot\tau}$. We get
finally ${\mathcal Z}_\tau\subset {\mathcal K}^T$
$\Leftrightarrow$ ${\mathcal Z}^*_{S\cdot \tau}\subset {{\mathcal
K}^*}^{T^S}$. Therefore, $(\tau,T)\in\mathbf{K}$ $\Leftrightarrow$
$(S\cdot \tau,T^S)\in\mathbf{K}$.  
\end{proof}

\section{The general inclusion $\mathbf{K}\subset\mathbf{R}$}

\label{4}

In this section, we associate to a row-standard tableau $\tau$ a
sequence of Young diagrams $Y_{j/i}(\tau)$ for $0\leq i<j\leq n$,
and we associate to a standard tableau $T$ a sequence of Young
diagrams $Y_{j/i}^T$ for $0\leq i<j\leq n$ (cf.
\ref{sequences-subdiagrams}). We write
$$\tau\preceq T\mbox{ \ if \ }Y_{j/i}(\tau)\preceq Y_{j/i}^T\mbox{ for any }i,j,$$
where $\preceq$ is the dominance relation. The set
$$\mathbf{R}:=\{(\tau,T)\in\mathbf{Y}:\tau\preceq T\}$$
has a purely combinatorial definition.

Actually, the diagrams $Y_{j/i}(\tau)$ (resp. $Y_{j/i}^T$)
represent the Jordan forms of the nilpotent maps induced by $u$ on
the subquotients $V_j/V_i$ of a flag $F\in{\mathcal Z}_\tau$
(resp. of a generic element $F\in{\mathcal K}^T$).

We show the general implication
$$\tau\in T\Rightarrow \tau\preceq T,$$
which means that we have the inclusion $\mathbf{K}\subset
\mathbf{R}$ (cf. Theorem \ref{theorem_necessary_condition}). This
provides a combinatorial necessary condition for having $\tau\in
T$. However, this condition is not sufficient whenever $\tau,T$
are not of hook, two-row or two-column type (cf. Proposition
\ref{proposition-counterexample}).

\subsection{Lower semi-continuity of the Jordan form on a subquotient}

\label{semi-continuity}

The set of Young diagrams is partially ordered with the dominance
order: let $Y,Y'$ be two Young diagrams, let $\lambda_1\geq
...\geq \lambda_r$ (resp. $\lambda'_1\geq ...\geq \lambda'_r$) be
the lengths - possibly $0$ - of the rows of $Y$ (resp. of $Y'$),
we write $Y\preceq Y'$ if
$$\lambda_1+...+\lambda_p\leq \lambda'_1+...+\lambda'_p\ \forall p=1,...,r.$$
Equivalently, denoting by $\mu_1\geq ...\geq \mu_s$ (resp.
$\mu'_1\geq ...\geq \mu'_s$) the lengths of the columns of $Y$
(resp. $Y'$), we have $Y\preceq Y'$ if and only if
$$\mu'_1+...+\mu'_q\leq \mu_1+...+\mu_q\ \forall q=1,...,s.$$

Recall from \ref{young-diagram} that, for
$F=(V_0,...,V_n)\in{\mathcal B}_u$, we denote by $Y(u_{|V_i})$ the
diagram representing the Jordan form of the nilpotent
$u_{|V_i}\in\mathrm{End}(V_i)$ induced by $u$. Observe that the
map $F\mapsto Y(u_{|V_i})$ is lower semicontinuous, i.e. the
subset
$$\{F\in{\mathcal B}_u: Y(u_{|V_i})\preceq Y_0\}$$
is closed for any Young diagram $Y_0$. Indeed, we see that the
number of boxes in the first $q$ columns of $Y(u_{|V_i})$ is $\dim
\ker (u_{|V_i})^q=n-\mathrm{rk}(u_{|V_i})^q$, and the map
$F\mapsto \mathrm{rk}\,(u_{|V_i})^q$ is lower semicontinuous.
Therefore, we get an implication:
$$F\in{\mathcal K}^T\ \Rightarrow\ Y(u_{|V_i})\preceq Y_i^T\ \forall i=1,...,n.$$
More generally, for $0\leq i<j\leq n$ and for
$F=(V_0,...,V_n)\in{\mathcal B}_u$, we may consider the nilpotent
map $u_{|V_j/V_i}\in\mathrm{End}(V_j/V_i)$ induced by $u$ on the
subquotient $V_j/V_i$, and its Young diagram $Y(u_{|V_j/V_i})$.

\begin{lemma}
\label{lemma-semi-continuity} The map $F=(V_0,...,V_n)\mapsto
Y(u_{|V_j/V_i})$ is lower semicontinuous.
\end{lemma}

\begin{proof}
It is sufficient to show that the map $F\mapsto
\mathrm{rk}\,u^q_{|V_j/V_i}$ is lower semicontinuous. This is an
easy consequence of the formula
\[
\mathrm{rk}\,u^q_{|V_j/V_i}=\mathrm{dim}\,(V_i+u^q(V_j))-i
\]
which can be proved by using the rank formula (see \cite[Lemma
2.2]{F}).  
\end{proof}

\subsection{A combinatorial necessary condition for having $\tau\in T$}

\label{sequences-subdiagrams} \label{Y-j/i-tau} \label{Y-j/i-T}
\label{section-vanLeeuwen} \label{K-subset-R}

In the purpose to deduce from Lemma \ref{lemma-semi-continuity} a
necessary condition for having $\tau\in T$, we determine
$Y(u_{|V_j/V_i})$ for $F\in{\mathcal Z}_\tau$ and for a generic
$F\in{\mathcal K}^T$.

\smallskip

Let $\tau$ be a row-standard tableau, with $|\tau|=n$ boxes. Let
$0\leq i<j\leq n$. Let $m_1\geq ...\geq m_k$ be the sizes of the
rows of the subtableau $\tau[i+1,...,j]$ of entries $i+1,...,j$.
Define $Y_{j/i}(\tau)$ as the Young diagram of rows of lengths
$m_1,...,m_k$. For example
\[\tau=\young(237,468,15)\qquad
\tau[4,...,7]=\young(::7,46,:5) \qquad Y_{7/3}(\tau)=\yng(2,1,1)\]
Embed ${\mathcal Z}_\tau\subset {\mathcal B}_u$. Let
$F=(V_0,...,V_n)\in {\mathcal Z}_\tau$ and let
$\underline{e}=(e_1,...,e_n)$ be a $\tau$-basis such that
$F=F(\underline{e})$. Then we see that the subbasis
$(e_{i+1},...,e_j)$ forms a Jordan basis of $u_{|V_j/V_i}$ of
Jordan blocks of sizes $m_1,...,m_k$. We have thus
$$Y(u_{|V_j/V_i})=Y_{j/i}(\tau).$$

Let $T$ be a standard tableau, with $|T|=n$ boxes. Let $0\leq
i<j\leq n$. The subtableau $T[i+1,...,j]$ of entries $i+1,...,j$
is a skew tableau. By jeu de taquin, we transform it into a Young
tableau. We denote by $Y_{j/i}^T$ the shape of the so-obtained
Young tableau. For example
\[T=\young(124,358,67)\qquad
T[3,...,6]=\young(::4,35,6)\stackrel{\mbox{\tiny
jeu\,de\,taquin}}{\longrightarrow}\young(34,5,6) \qquad
Y_{6/2}^T=\yng(2,1,1)\] Embed ${\mathcal K}^T\subset {\mathcal
B}_u$. It follows from \cite[Theorem 3.3]{vanLeeuwen} that the set
\[\{(V_0,...,V_n)\in{\mathcal K}^T:Y(u_{|V_j/V_i})=Y_{j/i}^T\}\]
is a nonempty open subset of the component ${\mathcal K}^T$.

\smallskip

Let $(\tau,T)\in\mathbf{Y}$. We write
\begin{center}
$\tau\preceq T$ \ if \ $Y_{j/i}(\tau)\preceq Y_{j.i}^T$ for all
$0\leq i<j\leq n$.
\end{center}
Define $\mathbf{R}$ as the set of pairs $(\tau,T)\in\mathbf{Y}$
such that $\tau\preceq T$. By Lemma \ref{lemma-semi-continuity},
we get:

\begin{theorem}
\label{theorem_necessary_condition} We have $\mathbf{K}\subset
\mathbf{R}$.
\end{theorem}

\subsection{The condition is not sufficient}

\label{counterexample}

The equivalence $\tau\in T\Leftrightarrow \tau\preceq T$ does not
hold in general. As a first example, let
\[
\tau=\young(125,46,3)\qquad T=\young(125,34,6)
\]
Then, we have $\tau\preceq T$ but $\tau\notin T$. Indeed, embed
${\mathcal K}^T,{\mathcal Z}_\tau\subset{\mathcal B}_u$. 
First, note that $\dim \ker u\cap\mathrm{Im}\,u=2$.
On one hand, we have $\ker u\cap\mathrm{Im}\,u\subset V_3$ for all $(V_0,\ldots,V_6)\in{\mathcal B}_u^T$,
hence also for all $(V_0,\ldots,V_6)\in{\mathcal K}^T$.
On the other hand,
$\mathrm{dim}\,(V_3\cap\mathrm{ker}\,u\cap \mathrm{Im}\,u)=1$ for any
$(V_0,...,V_6)\in{\mathcal Z}_\tau$. It follows: ${\mathcal
Z}_\tau\not\subset{\mathcal K}^T$.

Actually, we have the following

\begin{proposition}
\label{proposition-counterexample} Let $Y$ be a Young diagram
containing $\mbox{\scriptsize $\yng(3,2,1)$}$ as a subdiagram.
Then there is $(\tau,T)$ with
$Y=\mathrm{sh}(\tau)=\mathrm{sh}(T)$, such that $\tau\preceq T$
and $\tau\notin T$.
\end{proposition}

\begin{proof}
First, let $Y$ be a Young diagram of rows of lengths
$(\lambda_1,\lambda_2,...,\lambda_r)$ with
$\lambda_1>\lambda_2\geq 2$, $\lambda_3\geq 1$, $r\geq 3$. Let
$R_1$ be a row of $\lambda_1-3$ empty boxes, let $R_2$ be a row of
$\lambda_2-2$ boxes, let $R_3$ be a row of $\lambda_3-1$ boxes,
for $p\geq r$ let $R_p$ be a row of $\lambda_p$ boxes. Fill in the
boxes of $R_1$ with numbers $7,8,...,\lambda_1+3$ from left to
right, number $R_2$ in turn by
$\lambda_1+4,\lambda_1+5,...,1+\lambda_1+\lambda_2$ from left to
right, number $R_3$ by
$\lambda_1+\lambda_2+2,...,\lambda_1+\lambda_2+\lambda_3$ from
left to right, for $p\geq r$ number $R_p$ by
$\lambda_1+...+\lambda_{p-1}+1,...,\lambda_1+...+\lambda_p$ from
left to right. It is straightforward to check that the pair
$(\tau,T)$ with \[ \tau=\mbox{\tiny $\begin{array}{|c|c|c|c|c|}
\hline 1 & 2 & 5 & \multicolumn{2}{|l|}{R_1\ \ } \\
\hline 4 & 6 & \multicolumn{2}{|l|}{R_2\ \ } \\
\cline{1-4} 3 & \multicolumn{2}{|l|}{\ R_3\ } \\
\cline{1-3} \multicolumn{3}{|l|}{\ R_4\ } \\
\cline{1-3} \multicolumn{2}{|l|}{ \mbox{etc.} } \\
\cline{1-2}
\end{array}$}
\qquad T=\mbox{\tiny $\begin{array}{|c|c|c|c|c|}
\hline 1 & 2 & 5 & \multicolumn{2}{|l|}{R_1\ \ } \\
\hline 3 & 4 & \multicolumn{2}{|l|}{R_2\ \ } \\
\cline{1-4} 6 & \multicolumn{2}{|l|}{\ R_3\ } \\
\cline{1-3} \multicolumn{3}{|l|}{\ R_4\ } \\
\cline{1-3} \multicolumn{2}{|l|}{ \mbox{etc.} } \\
\cline{1-2}
\end{array}$}
\]
belongs to $\mathbf{R}$. By Theorem \ref{theorem-induction}, it
does not belong to $\mathbf{K}$.

It is straightforward to deduce that the pair $(\hat\tau,\hat T)$
with \[ \hat\tau=\mbox{\tiny $\begin{array}{|c|c|c|c|c|}
\hline 1 & 2 & 5 & \multicolumn{2}{|l|}{R_1\ \ } \\
\hline 0 & 4 & 6 & \multicolumn{2}{|l|}{R_2\ \ } \\
\cline{1-5} 3 & \multicolumn{2}{|l|}{\ R_3\ } \\
\cline{1-3} \multicolumn{3}{|l|}{\ R_4\ } \\
\cline{1-3} \multicolumn{2}{|l|}{ \mbox{etc.} } \\
\cline{1-2}
\end{array}$}+1
\qquad \hat T=\mbox{\tiny $\begin{array}{|c|c|c|c|c|}
\hline 0 & 1 & 2 & \multicolumn{2}{|l|}{R_1\ \ } \\
\hline 3 & 4 & 5 & \multicolumn{2}{|l|}{R_2\ \ } \\
\cline{1-5} 6 & \multicolumn{2}{|l|}{\ R_3\ } \\
\cline{1-3} \multicolumn{3}{|l|}{\ R_4\ } \\
\cline{1-3} \multicolumn{2}{|l|}{ \mbox{etc.} } \\
\cline{1-2}
\end{array}$}+1
\]
belongs to $\mathbf{R}$ (where ``$+1$'' means that we increase by
1 each entry of the tableau, in order to obtain a row-standard and
a standard tableau). (The dominance relations for $(i,j)$ for
$i\geq 1$ follow from the previous case - relations for $(0,j)$ are
easy.) Let us show $(\hat\tau,\hat T)\notin\mathbf{K}$. Let $(\hat
\tau',\hat T')$ be the pair of subtableaux of entries $1,...,7$:
\[\hat \tau'=\young(236,157,4)\quad \hat T'=\young(123,456,7)
\quad\mbox{hence}\ \ S\cdot\hat\tau'=\young(256,137,4)\quad \hat
T'^S=\young(134,267,5)\] Let $\hat\tau''=S\cdot ((S\cdot
\hat\tau')[1,...,6])$ and $\hat T''=((\hat T'^S)[1,...,6])^S$. We
have
\[\hat\tau''=\young(125,46,3)\qquad
\hat T''=\young(125,34,6)\] As observed previously, we have $(\hat
\tau'',\hat T'')\notin \mathbf{K}$. Combining Theorem
\ref{theorem-induction} and Proposition
\ref{proposition-Schutzenberger}, we deduce
$(\tau,T)\notin\mathbf{K}$. It solves the case
$(\lambda_1,\lambda_2,...,\lambda_r)$ with
$\lambda_1=\lambda_2\geq 3$, $\lambda_3\geq 1$, $r\geq 3$. The
proof is now complete.  
\end{proof}

\section{Equivalence $\tau\in T\Leftrightarrow\tau\preceq T$ in the hook, 2-row and 2-column cases}

\label{5}

In the previous section we have established the general
implication
\[\tau\in T \Rightarrow \tau\preceq T.\]
Proposition \ref{proposition-counterexample} shows that the
inverse implication is not true in general. Our purpose in this
section is to show the following

\begin{theorem}
\label{th-KR-3cas} Let $(\tau,T)\in\mathbf{Y}$ be of hook, two-row
or two-column type. We have the equivalence
$$\tau\in T\Leftrightarrow\tau\preceq T.$$
\end{theorem}

We treat successively the three cases, and Theorem
\ref{th-KR-3cas} will follow by combining Theorems
\ref{theorem-A-hook-case}, \ref{theorem-KR-2-row} and
\ref{theorem-A-2column}. The proof in each case relies on a
further characterization.

\subsection{Equivalence in the hook case}

\label{A-hook-case}

Let us state some notation. We denote by $\mathbf{Y}_{\mbox{\tiny
hk}}$ the subset of pairs $(\tau,T)\in\mathbf{Y}$ such that the
Young shape $\mathrm{sh}(\tau)=\mathrm{sh}(T)$ is of hook type.
Set $\mathbf{K}_{\mbox{\tiny hk}}=\mathbf{K}\cap
\mathbf{Y}_{\mbox{\tiny hk}}$ and $\mathbf{R}_{\mbox{\tiny
hk}}=\mathbf{R}\cap \mathbf{Y}_{\mbox{\tiny hk}}$. We aim to prove
the equality $\mathbf{K}_{\mbox{\tiny hk}}=\mathbf{R}_{\mbox{\tiny
hk}}$.

We rely on another subset $\mathbf{A}_{\mbox{\tiny hk}}\subset
\mathbf{Y}_{\mbox{\tiny hk}}$. Let $(\tau,T)\in
\mathbf{Y}_{\mbox{\tiny hk}}$. Let $a_1<...<a_s$ (resp.
$a'_1<...<a'_s$) be the entries in the first row of $T$ (resp. of
$\tau$). Say $(\tau,T)\in \mathbf{A}_{\mbox{\tiny hk}}$ if we have
\[a'_{q-1}< a_q\leq a'_q\ \ \forall q\in\{2,...,s\}.\]

We show the following

\begin{theorem}
\label{theorem-A-hook-case} We have $\mathbf{K}_{\mbox{\rm\tiny
hk}}=\mathbf{A}_{\mbox{\rm\tiny hk}}=\mathbf{R}_{\mbox{\rm\tiny
hk}}$.
\end{theorem}

\begin{proof}
The inclusion $\mathbf{K}_{\mbox{\tiny hk}}\subset
\mathbf{R}_{\mbox{\tiny hk}}$ follows from Theorem
\ref{theorem_necessary_condition}. The inclusion
$\mathbf{A}_{\mbox{\tiny hk}}\subset \mathbf{K}_{\mbox{\tiny hk}}$
follows from \cite[Theorem 4.1]{Vargas}. Then it remains to show
$\mathbf{R}_{\mbox{\tiny hk}}\subset \mathbf{A}_{\mbox{\tiny
hk}}$. Suppose that $(\tau,T)\in\mathbf{R}_{\mbox{\tiny hk}}$ and
let us show $(\tau,T)\in\mathbf{A}_{\mbox{\tiny hk}}$. Let
$q\in\{2,...,s\}$. Define $\hat{q}\in\{1,...,s+1\}$ by
\begin{itemize}
\item if $a'_s\geq a_q$, then $\hat{q}\in\{1,...,s\}$ is minimal such that $a'_{\hat{q}}\geq a_q$,
\item if $a'_s<a_q$, then $\hat{q}=s+1$.
\end{itemize}
We have to prove $q=\hat{q}$.

First, we use the relation $Y_{a_q-1/0}(\tau)\preceq
Y_{a_q-1/0}^T$. The diagram $Y_{a_q-1/0}^T$ is the shape of the
subtableau of $T$ of entries $i<a_q$, hence its first row contains
$q-1$ boxes. The lengths of the rows of the diagram
$Y_{a_q-1/0}(\tau)$ are the lengths of the rows of the subtableau
of $\tau$ of entries $i<a_q$. The first row of this subtableau
contains $\hat{q}-1$ boxes. By the relation above, we get $\hat{q}\leq q$.

Write $n=|\tau|=|T|$. Next, we use the relation
$Y_{n/a_q-1}(\tau)\preceq Y_{n/a_q-1}^T$. The diagram
$Y_{n/a_q-1}^T$ is the shape of the transform by jeu de taquin of
the skew subtableau of $T$ of entries $a_q,...,n$, that is
\[
\begin{array}{|c|c|c|}
\hline
\!a_q\! & \cdots & \!a_s\! \\
\hline
\!b_{p}\! \\
\cline{1-1}
\vdots \\
\cline{1-1}
\!b_r\! \\
\cline{1-1}
\end{array}
\]
for some $b_{p},...,b_r$. Thus the first row of $Y^T_{n/a_q-1}$ has
length $s-q+1$. On the other hand the first row of the subtableau
of $\tau$ of entries $a_q,...,n$ contains the entries
$a'_{\hat{q}},...,a'_s$. This row is nonempty, since we have already
proved $\hat{q}\leq q(\leq s)$. The first row of the Young diagram
$Y_{n/a_q-1}(\tau)$ has thus length $s-\hat{q}+1$. By the dominance
relation above, we obtain $\hat{q}\geq q$. We get finally $\hat{q}=q$.  
\end{proof}

\subsection{Equivalence in the two-row case}

\label{section-A-2row} \label{A-2row-case}

We denote by $\mathbf{Y}_{\mbox{\tiny 2-r}}$ the subset of pairs
$(\tau,T)\in\mathbf{Y}$ such that the Young shape
$\mathrm{sh}(\tau)=\mathrm{sh}(T)$ has (at most) two rows. Set
$\mathbf{K}_{\mbox{\tiny 2-r}}=\mathbf{K}\cap
\mathbf{Y}_{\mbox{\tiny 2-r}}$ and $\mathbf{R}_{\mbox{\tiny
2-r}}=\mathbf{R}\cap \mathbf{Y}_{\mbox{\tiny 2-r}}$. We shall
prove the equality $\mathbf{K}_{\mbox{\tiny
2-r}}=\mathbf{R}_{\mbox{\tiny 2-r}}$. As previously for the hook
case, we rely on another subset $\mathbf{A}_{\mbox{\tiny
2-r}}\subset \mathbf{Y}_{\mbox{\tiny 2-r}}$ that we define by
induction.

\smallskip

First, we define $\mathbf{\hat{A}}_{\mbox{\tiny 2-r}}$ as the set
of $(\tau,T)\in\mathbf{Y}_{\mbox{\tiny 2-r}}$ such that one of the
following conditions is satisfied (writing $n=|\tau|=|T|$):
\begin{itemize}
\item[(1)] there is $i\in\{2,...,n\}$ (necessarily even) such that
the subtableaux $\tau[1,...,i]$ and $T[1,...,i]$ of entries
$1,...,i$ have a rectangular shape with two rows of length $i/2$;
\item[(2)] the subtableau $T[1,...,i]$ has a non-rectangular shape
for every $i\in\{2,...,n\}$ and $1$ has the same place in $\tau$
and $T$.
\end{itemize}
Let $(\tau,T)\in \mathbf{\hat{A}}_{\mbox{\tiny 2-r}}$. We
associate to the pair $(\tau,T)$ a subset $\eta(\tau,T)\subset
\mathbf{Y}_{\mbox{\tiny 2-r}}$ with one or two elements:

\smallskip\noindent
(a) Suppose that (1) holds for some $i<n$, and choose
$i\in\{2,...,n-1\}$ maximal for which the property holds. Then the
subtableaux $\tau':=\tau[1,...,i]$ and $T':=T[1,...,i]$ have the
same shape. On the other hand, the subtableaux $\tau[i+1,...,n]$
and $T[i+1,...,n]$ of entries $i+1,...,n$ have the same shape
which is a Young diagram with $n-i$ boxes. Write
$\tau''=\tau[i+1,...,n]-i$ and $T''=T[i+1,...,n]-i$, where
``$-i$'' means that each entry $j$ in the tableau is replaced by
$j-i$, so that $\tau''$ and $T''$ are respectively row-standard
and standard. Define $\eta(\tau,T)=\{(\tau',T'),(\tau'',T'')\}$.

\smallskip\noindent
(b) Suppose now that (1) holds for no $i<n$, but holds for
$i=n$. Write $T'=T[1,...,n-1]$. Let $\tilde\tau$ be the tableau
obtained by switching the two rows of $\tau$. The entry $n$ is
either in the second row of $\tau$, or in the second row of
$\tilde\tau$. In the former case, write $\tau'=\tau[1,...,n-1]$.
In the latter case, write $\tau'=\tilde\tau[1,...,n-1]$. Define
$\eta(\tau,T)=\{(\tau',T')\}$.

\smallskip\noindent
(c) Finally, suppose that (2) holds. Write $1=a_1<a_2<...<a_r$
(resp. $b_1<b_2<...<b_s$) the entries of the first (resp. second)
row of $T$. By hypothesis, we have $s<r$ and $b_q>a_{q+1}$ for any
$q=1,...,s$, hence the transform by jeu de taquin of the
subtableau $T[2,...,n]$ contains $a_2,...,a_r$ in its first row
and $b_1,...,b_s$ in its second row. Let $T'$ be the standard
tableau with $a_2-1,...,a_r-1$ in its first row and
$b_1-1,...,b_s-1$ in its second row. Write $1=a'_1<a'_2<...<a'_r$
(resp. $b'_1<b'_2<...<b'_s$) the entries of the first (resp.
second) row of $\tau$. Let $\tau'$ be the row-standard tableau
with $a'_2-1,...,a'_r-1$ in its first row and $b'_1-1,...,b'_s-1$
in its second row. Define $\eta(\tau,T)=\{(\tau',T')\}$.

\smallskip

Now, we define the set $\mathbf{A}_{\mbox{\tiny 2-r}}$ by
induction.
\begin{itemize}
\item[-] Say $(\mbox{\scriptsize \young(1)},\mbox{\scriptsize \young(1)})\in \mathbf{A}_{\mbox{\tiny 2-r}}$.
\item[-] For $(\tau,T)\in\mathbf{Y}$ with $|\tau|=|T|>1$, say $(\tau,T)\in \mathbf{A}_{\mbox{\tiny 2-r}}$
if we have $(\tau,T)\in \mathbf{\hat A}_{\mbox{\tiny 2-r}}$ and
$\eta(\tau,T)\subset \mathbf{A}_{\mbox{\tiny 2-r}}$.
\end{itemize}

\begin{example}
Let
\[\tau=\young(235,14)\qquad T=\young(134,25)\]
We see that $(\tau,T)\in\mathbf{\hat A}_{\mbox{\tiny 2-r}}$, and
$\eta(\tau,T)=\{(\tau',T'),(\tau'',T'')\}$ with
\[\tau'=\young(2,1)\quad T'=\young(1,2)\quad\mbox{ and }\quad\tau''=\young(13,2)\quad T''=\young(12,3)\]
We have $(\tau',T'),(\tau'',T'')\in\mathbf{\hat A}_{\mbox{\tiny
2-r}}$ and
\begin{eqnarray*}
&& \eta(\tau',T')=(\young(1),\young(1))\in{\mathbf A}_{\mbox{\tiny 2-r}},\mbox{ hence }(\tau',T')\in {\mathbf A}_{\mbox{\tiny 2-r}}, \\
&&
\eta(\tau'',T'')=(\young(2,1),\young(1,2))=(\tau',T')\in{\mathbf
A}_{\mbox{\tiny 2-r}},\mbox{ hence }(\tau'',T'')\in {\mathbf
A}_{\mbox{\tiny 2-r}}.
\end{eqnarray*}
Thus we get $\eta(\tau,T)\subset \mathbf{A}_{\mbox{\tiny 2-r}}$,
and it finally results $(\tau,T)\in {\mathbf A}_{\mbox{\tiny
2-r}}$.
\end{example}

We prove the following

\begin{theorem}
\label{theorem-KR-2-row} We have $\mathbf{K}_{\mbox{\rm\tiny
2-r}}=\mathbf{A}_{\mbox{\rm\tiny 2-r}}=\mathbf{R}_{\mbox{\rm\tiny
2-r}}$.
\end{theorem}

\begin{proof}
I) The inclusion $\mathbf{K}_{\mbox{\tiny
2-r}}\subset\mathbf{R}_{\mbox{\tiny 2-r}}$ follows from Theorem
\ref{theorem_necessary_condition}.

\smallskip
\noindent II) Let us show the inclusion $\mathbf{R}_{\mbox{\tiny
2-r}}\subset \mathbf{A}_{\mbox{\tiny 2-r}}$. According to the
inductive definition of $\mathbf{A}_{\mbox{\tiny 2-r}}$, it is
sufficient to check that any $(\tau,T)\in \mathbf{R}_{\mbox{\tiny
2-r}}$ satisfies $(\tau,T)\in\mathbf{\hat A}_{\mbox{\tiny 2-r}}$
and $\eta(\tau,T)\subset\mathbf{R}_{\mbox{\tiny 2-r}}$.

Set $n=|\tau|=|T|$. If for some $i\in\{2,...,n\}$, the subtableau
$T[1,...,i]$ is rectangular, then, using the relation
$Y_{i/0}(\tau)\preceq Y_{i/0}^T$, we get that the subtableau
$\tau[1,...,i]$ is also rectangular. If $T[1,...,i]$ is
rectangular for none $i\in\{2,...,n\}$, then, using the relation
$Y_{n/1}(\tau)\preceq Y_{n/1}^T$, we get that $1$ is in the first
row of $\tau$. Thus $(\tau,T)\in\mathbf{\hat A}_{\mbox{\tiny
2-r}}$.

\noindent In case (a) of the definition of $\eta(\tau,T)$, we have
$Y_{l/k}(\tau')=Y_{l/k}(\tau)$ for any $0\leq k<l\leq i$ and
$Y_{l/k}(\tau'')=Y_{l+i/k+i}(\tau)$ for any $0\leq k<l\leq n-i$ on
one hand, $Y_{l/k}^{T'}=Y_{l/k}^T$ for any $0\leq k<l\leq i$ and
$Y_{l/k}^{T''}=Y_{l+i/k+i}^T$ for any $0\leq k<l\leq n-i$ on the
other hand, therefore $(\tau',T'),(\tau'',T'')\in
\mathbf{R}_{\mbox{\tiny 2-r}}$.

\noindent In case (b) of the definition of $\eta(\tau,T)$, we have
$Y_{l/k}(\tau')=Y_{l/k}(\tau)$ for any $0\leq k<l\leq n-1$ and
$Y_{l/k}^{T'}=Y_{l/k}^T$ for any $0\leq k<l\leq n-1$, hence
$(\tau',T')\in \mathbf{R}_{\mbox{\tiny 2-r}}$.

\noindent In case (c) of the definition of $\eta(\tau,T)$, we have
$Y_{l/k}(\tau')=Y_{l+1/k+1}(\tau)$ for any $0\leq k<l\leq n-1$ and
$Y_{l/k}^{T'}=Y_{l+1/k+1}^T$ for any $0\leq k<l\leq n-1$, hence
$(\tau',T')\in \mathbf{R}_{\mbox{\tiny 2-r}}$.

\noindent In all the cases, we have shown $\eta(\tau,T)\subset
\mathbf{R}_{\mbox{\tiny 2-r}}$. It follows
$\mathbf{R}_{\mbox{\tiny 2-r}}\subset\mathbf{A}_{\mbox{\tiny
2-r}}$.

\smallskip
\noindent III) Now, let us prove the inclusion
$\mathbf{A}_{\mbox{\tiny 2-r}}\subset \mathbf{K}_{\mbox{\tiny
2-r}}$. To do this, it is sufficient to show that
$\eta(\tau,T)\subset \mathbf{K}_{\mbox{\tiny 2-r}}$ implies
$(\tau,T)\in \mathbf{K}_{\mbox{\tiny 2-r}}$, for any $(\tau,T)\in
\mathbf{\hat A}_{\mbox{\tiny 2-r}}$.

\noindent In case (c) of the definition of $\eta(\tau,T)$, we have
$\tau'=S\cdot((S\cdot \tau)[1,...,n-1])$ and
$T'=((T^S)[1,...,n-1])^S$, and $n$ is in the same place of $S\cdot
\tau$ and $T^S$, hence, combining Theorem
\ref{theorem-induction}\,(b) and Proposition
\ref{proposition-Schutzenberger}, we get
$$\eta(\tau,T)\subset \mathbf{K}_{\mbox{\tiny 2-r}}\Rightarrow
(S\cdot \tau,T^S)\in \mathbf{K}_{\mbox{\tiny 2-r}} \Rightarrow
(\tau,T)\in \mathbf{K}_{\mbox{\tiny 2-r}}.$$

\noindent In case (b) of the definition of $\eta(\tau,T)$, we get
the implication $\eta(\tau,T)\subset \mathbf{K}_{\mbox{\tiny
2-r}}\Rightarrow (\tau,T)\in \mathbf{K}_{\mbox{\tiny 2-r}}$, by
applying Theorem \ref{theorem-induction}\,(b).

\noindent Finally, we consider the case (a) of the definition of
$\eta(\tau,T)$. Embed ${\mathcal K}^T,{\mathcal Z}_\tau\subset
{\mathcal B}_u$. Set $V'=\mathrm{ker}\,u^{i/2}$ and let
$u'\in\mathrm{End}(V')$ be the nilpotent map induced by $u$, set
$V''=V/\mathrm{ker}\,u^{i/2}$ and let $u''\in\mathrm{End}(V'')$ be
the nilpotent map induced by $u$, let ${\mathcal B}_{u'}$ and
${\mathcal B}_{u''}$ be the corresponding Springer fibers. The map
\begin{eqnarray}
\Phi:{\mathcal B}_{u'}\times {\mathcal B}_{u''} & \rightarrow & {\mathcal B}_{u}\nonumber \\
(V_0,...,V_i),(W_0,...,W_{n-i}) & \mapsto &
(V_0,...,V_i,\mathrm{ker}\,u^{i/2}+W_1,...,\mathrm{ker}\,u^{i/2}+W_{n-i})\nonumber
\end{eqnarray}
is well defined and algebraic. Easily, we have $\Phi({\mathcal
B}_{u'}^{T'}\times{\mathcal B}_{u''}^{T''})\subset {\mathcal
B}_{u}^{T}$. It follows
\[\Phi({\mathcal K}^{T'}\times {\mathcal K}^{T''})=\Phi(\overline{{\mathcal B}_{u'}^{T'}\times{\mathcal B}_{u''}^{T''}})
\subset {\mathcal K}^{T}.\] On the other hand, we see that
${\mathcal Z}_\tau\subset \Phi({\mathcal Z}_{\tau'}\times{\mathcal
Z}_{\tau''})$. It results:
$$
{\mathcal Z}_{\tau'}\subset{\mathcal K}^{T'},{\mathcal
Z}_{\tau''}\subset{\mathcal K}^{T''}\Rightarrow {\mathcal
Z}_\tau\subset {\mathcal K}^T.
$$
Thus, we get the implication $\eta(\tau,T)\subset
\mathbf{K}_{\mbox{\tiny 2-r}} \Rightarrow
(\tau,T)\in\mathbf{K}_{\mbox{\tiny 2-r}}$. The proof of the
theorem is now complete.  
\end{proof}

\subsection{Equivalence in the two-column case}

\label{A-2column-case} \label{def-A-2column}

Let $\mathbf{Y}_{\mbox{\tiny 2-c}}$ denote the subset of pairs
$(\tau,T)\in\mathbf{Y}$ such that the Young shape
$\mathrm{sh}(\tau)=\mathrm{sh}(T)$ has (at most) two columns. Set
$\mathbf{K}_{\mbox{\tiny 2-c}}=\mathbf{K}\cap
\mathbf{Y}_{\mbox{\tiny 2-c}}$ and $\mathbf{R}_{\mbox{\tiny
2-c}}=\mathbf{R}\cap \mathbf{Y}_{\mbox{\tiny 2-c}}$. As previously
for the hook and two-row cases, we prove the equality
$\mathbf{K}_{\mbox{\tiny 2-c}}=\mathbf{R}_{\mbox{\tiny 2-c}}$,
and, to do this, we rely on another subset
$\mathbf{A}_{\mbox{\tiny 2-c}}\subset\mathbf{Y}_{\mbox{\tiny
2-c}}$.

\smallskip

Let us start with some notation. Let $(\tau,T)\in
\mathbf{Y}_{\mbox{\tiny 2-c}}$. We denote by $C_q(\tau)$ (resp.
$C_q(T)$) the set of the entries in the $q$-th column of $\tau$
(resp. $T$). For $i\in C_2(\tau)$, let $\nu_\tau(i)\in C_1(\tau)$
be the entry on the left of $i$ in $\tau$. For $i\in C_1(\tau)$,
let $\omega_\tau(i)\in C_2(\tau)\cup\{\infty\}$ be either the
entry on the right of $i$ in $\tau$ if there is one, or
$\omega_\tau(i)=\infty$ otherwise. Say by convention $i<\infty$
for any integer $i$. Set $n=|\tau|=|T|$.

As in the two-row case, we define first a subset $\mathbf{\hat
A}_{\mbox{\tiny 2-c}}\subset\mathbf{Y}_{\mbox{\tiny 2-c}}$. Recall
that we denote by $\mathrm{st}(\tau)$ the standard tableau
obtained by putting in increasing order the numbers inside the
columns of $\tau$ (cf. \ref{standardization}). Assume
$T\not=\mathrm{st}(\tau)$. Then there is $i\in\{1,...,n\}$ which
does not lie in the same column of $\tau$ and $T$. Take $i$
minimal. We ask for three conditions:
\begin{itemize}
\item[(1)] We have $i\in C_1(\tau)\cap C_2(T)$.
\item[(2)] There is $j\in\{i+1,...,n\}\cap C_2(\tau)$
such that $\nu_{\tau}(j)\leq i$.
\end{itemize}
Then take $j$ minimal.
\begin{itemize}
\item[(3)] There is $i'\in\{i,...,j-1\}\cap C_1(\tau)$ such that
$\omega_{\tau}(i')>j$.
\end{itemize}
Let $\mathbf{\hat A}_{\mbox{\tiny 2-c}}\subset
\mathbf{Y}_{\mbox{\tiny 2-c}}$ be the subset of pairs $(\tau,T)$
such that $T\not=\mathrm{st}(\tau)$, and satisfying all the
conditions (1), (2) and (3).

\begin{example}
\label{example-eta-1} The tableaux $\tau=\mbox{\scriptsize
$\young(24,17,36,8,5)$}$ and $T=\mbox{\scriptsize
$\young(12,34,56,7,8)$}$ satisfy the conditions with $i=2$ and
$j=4$ and for example $i'=3$. Thus $(\tau,T)\in \mathbf{\hat
A}_{\mbox{\tiny 2-c}}$.
\end{example}

Let us now define a map $\eta:\mathbf{\hat A}_{\mbox{\tiny
2-c}}\rightarrow\mathbf{Y}_{\mbox{\tiny 2-c}}$. Let
$(\tau,T)\in\mathbf{\hat A}_{\mbox{\tiny 2-c}}$, and let $i,j$ be
the entries involved in conditions (1) and (2). We define a
row-standard tableau $\tilde\tau$ with the same shape as $\tau$.
We distinguish two cases.

\smallskip
\noindent (a) Suppose that $\omega_{\tau}(i)<\omega_{\tau}(i')$
for some $i'\in\{i+1,...,j-1\}\cap C_1(\tau)$. Take $i'$ minimal.
Then define $\tilde\tau$ as the tableau obtained from $\tau$ by
switching $i$ and $i'$. By condition (2), we have
$\omega_{\tau}(i)\geq j$, hence $\omega_{\tau}(i)> i'$. By
definition of $i'$ we have $\omega_{\tau}(i')> i$. Hence the
tableau $\tilde\tau$ is row-standard.

\smallskip
\noindent (b) Next suppose that
$\omega_{\tau}(i')\leq\omega_{\tau}(i)$ for any $i'\in
\{i,...,j-1\}\cap C_1(\tau)$. Then define $\tilde\tau$ as the
tableau obtained from $\tau$ by switching $i$ and $j$. By
condition (3), we have $\omega_{\tau}(i)>j$, hence the tableau
$\tilde\tau$ is row-standard.

\smallskip
Then, we set $\eta(\tau,T)=(\tilde\tau,T)$.

\begin{example}
Let $\tau$ and $T$ be as in Example \ref{example-eta-1}. Then
$\tilde\tau=\mbox{\scriptsize $\young(34,17,26,8,5)$}$.
\end{example}

Finally, we define the set $\mathbf{A}_{\mbox{\tiny 2-c}}$ by
induction.
\begin{itemize}
\item[-] Say $(\tau,T)\in \mathbf{A}_{\mbox{\tiny 2-c}}$ whenever $T=\mathrm{st}(\tau)$.
\item[-] For $(\tau,T)\in\mathbf{Y}$ with $T\not=\mathrm{st}(\tau)$,
say $(\tau,T)\in \mathbf{A}_{\mbox{\tiny 2-c}}$ if we have
$(\tau,T)\in \mathbf{\hat A}_{\mbox{\tiny 2-c}}$ and
$\eta(\tau,T)\in \mathbf{A}_{\mbox{\tiny 2-c}}$.
\end{itemize}

Observe that, assuming $(\tilde\tau,T)=\eta(\tau,T)\in\mathbf{\hat A}_{\mbox{\tiny 2-c}}$
and letting $\tilde\imath\in\{1,\ldots,n\}$ be the minimal entry which is not in the
same column of $\tilde\tau$ and $T$, we have either $\tilde\imath>i$
(in case (b) of the definition of $\tilde\tau$ above), or $\tilde\imath=i$ and $\omega_{\tilde\tau}(\tilde\imath)>\omega_{\tau}(i)$
(in case (a)). It follows by induction on the pair $(i,\omega_\tau(i))$
that, applying $\eta$ to the pair $(\tau,T)$ a certain number of times, we obtain an element
which does lie in $\mathbf{\hat A}_{\mbox{\tiny 2-c}}$. Therefore, the definition
of the set $\mathbf{A}_{\mbox{\tiny 2-c}}$ is valid.

\begin{example}
Let $\tau$ and $T$ be as in Example \ref{example-eta-1}. We apply
$\eta$ as many times as possible:
\[\tau=\mbox{\scriptsize $\young(24,17,36,8,5)$}
\ \rightarrow\ \mbox{\scriptsize $\young(34,17,26,8,5)$} \
\rightarrow\ \mbox{\scriptsize $\young(34,17,56,8,2)$} \
\rightarrow\ \mbox{\scriptsize $\young(34,12,56,8,7)$}=:\tau'
\]
(each intermediate tableau belongs to $\mathbf{\hat
A}_{\mbox{\tiny 2-c}}$). Since the last tableau satisfies
$T=\mathrm{st}(\tau')$, we get finally $(\tau,T)\in
\mathbf{A}_{\mbox{\tiny 2-c}}$.
\end{example}

We have the following

\begin{theorem}
\label{theorem-A-2column} We have $\mathbf{K}_{\mbox{\rm\tiny
2-c}}=\mathbf{A}_{\mbox{\rm\tiny 2-c}}=\mathbf{R}_{\mbox{\rm\tiny
2-c}}$.
\end{theorem}

\begin{proof}
The inclusion $\mathbf{K}_{\mbox{\rm\tiny
2-c}}\subset\mathbf{R}_{\mbox{\rm\tiny 2-c}}$ follows from Theorem
\ref{theorem_necessary_condition}.

Let us show the inclusion $\mathbf{A}_{\mbox{\rm\tiny
2-c}}\subset\mathbf{K}_{\mbox{\rm\tiny 2-c}}$. Let
$(\tau,T)\in\mathbf{A}_{\mbox{\rm\tiny 2-c}}$. If
$T=\mathrm{st}(\tau)$, then we have $(\tau,T)\in
\mathbf{K}_{\mbox{\rm\tiny 2-c}}$ by Proposition
\ref{proposition-standardization}. Assume now
$(\tau,T)\in\mathbf{\hat A}_{\mbox{\rm\tiny 2-c}}$. By induction,
it remains to show the implication
$\eta(\tau,T)\in\mathbf{K}_{\mbox{\rm\tiny 2-c}}\Rightarrow
(\tau,T)\in\mathbf{K}_{\mbox{\rm\tiny 2-c}}$. Write
$\eta(\tau,T)=(\tilde\tau,T)$. Embed ${\mathcal K}^T,{\mathcal
Z}_\tau,{\mathcal Z}_{\tilde\tau}\subset {\mathcal B}_u$. Then, it
is sufficient to see that ${\mathcal Z}_\tau$ is contained in the
closure of ${\mathcal Z}_{\tilde\tau}$. This follows from the
Claim inside the proof of \cite[Theorem 2.5]{F}.

Finally, let us show the inclusion $\mathbf{R}_{\mbox{\rm\tiny
2-c}}\subset\mathbf{A}_{\mbox{\rm\tiny 2-c}}$. Let $(\tau,T)\in
\mathbf{R}_{\mbox{\rm\tiny 2-c}}$. If $T=\mathrm{st}(\tau)$, then
it is immediate that $(\tau,T)\in \mathbf{A}_{\mbox{\rm\tiny
2-c}}$. Assume now that $T\not=\mathrm{st}(\tau)$. By induction,
it is sufficient to see that $(\tau,T)\in \mathbf{\hat
A}_{\mbox{\rm\tiny 2-c}}$ and $\eta(\tau,T)\in
\mathbf{R}_{\mbox{\rm\tiny 2-c}}$. The fact that $(\tau,T)\in
\mathbf{\hat A}_{\mbox{\rm\tiny 2-c}}$ is shown in the first part
of the proof of \cite[Theorem 2.5]{F} (before the Claim). The fact
that $\eta(\tau,T)\in \mathbf{R}_{\mbox{\rm\tiny 2-c}}$ is proved
in the last part of the proof of \cite[Theorem 2.5]{F} (after the
Claim).  
\end{proof}

\subsection{A connection between the two-row and two-column cases}

\label{connection-1-2row-2column} \label{transposee}

We provide a complement to Theorem \ref{theorem-A-2column} where
the combinatorial condition required for having $\tau\in T$ is
weakened in the particular case where $\tau$ is a standard tableau
(i.e. rows and also columns of $\tau$ are increasing) and $\tau,T$
have two columns:

\begin{proposition}
\label{proposition-standard-2column} Let $(\tau,T)\in
\mathbf{Y}_{\mbox{\tiny\rm 2-c}}$. Assume that $\tau$ is a
standard tableau. Then, $\tau\in T$ if and only if
$Y_{i/0}(\tau)\preceq Y_{i/0}^T$ for every $i=1,...,n$.
\end{proposition}

\begin{proof}
The implication $(\Rightarrow)$ is immediate, let us show
$(\Leftarrow)$. By Theorem \ref{theorem-A-2column}, it is
sufficient to derive the relation $Y_{j/i}(\tau)\preceq Y_{j/i}^T$
for all $0\leq i<j\leq n$. Let $s_{j/i}^T$ (resp. $s_{j/i}(\tau)$)
denote the length of the second column of $Y_{j/i}^T$ (resp. of
$Y_{j/i}(\tau)$). The relation $Y_{j/i}(\tau)\preceq Y_{j/i}^T$
is immediate if $Y_{j/i}(\tau)$ has only one column.
Otherwise, there are $a,b$ in the same row of $\tau$ such that $i<a<b\leq j$.
As $\tau$ is standard, every $a'\in C_1(\tau)$ with 
$a'\leq a$ has a
neighbor on the right $b'\leq b$. Thus, every $a'\in\{1,...,i\}\cap C_1(\tau)$
has a neighbor on the right $b'\in\{1,...,j\}$ in $\tau$,
hence $\#\{1,...,i\}\cap C_1(\tau)$ is the number of rows of $\tau$ 
of the form $(a',b')$ with $a'\leq i$, $a'<b'\leq j$.
Note that $s_{j/0}(\tau)$ (resp. $s_{j/i}(\tau)$) is the number of rows of $\tau$ of the form
$(a',b')$ with $a'<b'\leq j$ (resp. with $i<a'<b'\leq j$). It follows 
$$s_{j/0}(\tau)=s_{j/i}(\tau)+\#\{1,...,i\}\cap C_1(\tau).$$
Similarly, $s_{j/0}^T$ is the number of rows of length $2$ in the subtableau $T[1,...,j]$.
The number $\#\{1,\ldots,i\}\cap C_1(T)$ is greater than or equal to the number of rows of the form
$(a',b')$ with $a'\leq i$ in the subtableau $T[1,...,j]$.
By definition of jeu de taquin, $s_{j/i}^T$ is greater than or equal to
the number of rows of length $2$ in the skew subtableau
$T[i+1,...,j]$. It follows $$s_{j/0}^T\leq
s_{j/i}^T+\#\{1,...,i\}\cap C_1(T).$$
The relation $Y_{i/0}(\tau)\preceq Y_{i/0}^T$ implies $\#\{1,...,i\}\cap
C_1(\tau)\geq \#\{1,...,i\}\cap C_1(T)$. The
relation $Y_{j/0}(\tau)\preceq Y_{j/0}^T$ implies
$s_{j/0}(\tau)\leq s_{j/0}^T$. Therefore, we get
$s_{j/i}(\tau)\leq s_{j/i}^T$. It follows $Y_{j/i}(\tau)\preceq
Y_{j/i}^T$.
\end{proof}

\begin{remark}
Assume $Y(u)$ has two columns of lengths $r\geq s$. Let $\tau_0$
be the standard tableau of shape $Y(u)$, numbered by $1,...,r$ in
the first column and $r+1,...,r+s$ in the second column. Then it
is clear that $Y_{i/0}(\tau_0)\preceq Y_{i/0}^T$ for any $i$, any
$T$ standard, hence ${\mathcal Z}_{\tau_0}\subset {\mathcal K}^T$
for any $T$. In particular, when $Y(u)$ has two columns, the
intersection of all the components of ${\mathcal B}_u$ is always
nonempty (see also \cite[\S 2.1]{F}).
\end{remark}

\medskip

We provide a connection between the two-row and two-column cases.
Let $Y$ be a Young diagram. Denoting by $\lambda_1>...>\lambda_r$
the lengths of the rows of $Y$, we define its transposed diagram
$Y^t$ as the Young diagram with $\lambda_q$ boxes in the $q$-th
column. Let $T$ be a standard tableau. We define $T^t$ as the
standard tableau obtained by transposition: if
$a_1<...<a_{\lambda_p}$ are the entries of the $p$-th row of $T$,
then the $p$-th column of $T^t$ contains $a_1,...,a_{\lambda_p}$
from top to bottom.

\begin{proposition}
\label{proposition-transposee} Let $T,S$ be two standard tableaux
of common shape with two rows. Then we have the implication:
$(S,T)\in\mathbf{K}$ $\Rightarrow$ $(T^t,S^t)\in\mathbf{K}$.
\end{proposition}

\begin{proof}
Assume that $(S,T)\in\mathbf{K}$. By Theorem
\ref{theorem-KR-2-row}, we obtain in particular the relation
$Y_{i/0}(S)\preceq Y_{i/0}^T$ for any $i$. The diagrams
$Y_{i/0}(S)$ and $Y_{i/0}^T$ are both defined in section
\ref{Y-j/i-T} but not in the same manner. The diagrams
$Y_{i/j}(S)$ and $Y_{i/j}^S$ may be different. Nevertheless, in
the present case $j=0$, the diagrams $Y_{i/0}(S)$ and $Y_{i/0}^S$
both coincide with the shape of the subtableau $S[1,...,i]$. The
relation $Y_{i/0}(S)\preceq Y_{i/0}^T$ implies
$Y_{i/0}(T^t)\preceq Y_{i/0}^{S^t}$ and, as it holds for any $i$,
we conclude that $(T^t,S^t)\in\mathbf{K}$ by applying Proposition
\ref{proposition-standard-2column}.  
\end{proof}

\section{A notion of constructible pairs $(\tau,T)$}

\label{6}

In the previous section, for the hook, two-row and two-column cases,
we give two combinatorial characterizations of the pairs $(\tau,T)$ such that $\tau\in T$.
The purpose of this section is to establish a further common combinatorial characterization,
for these three cases,
by transforming the inductive criterion of the previous section into an iterative one.
The new criterion involves an algorithm.
Starting with a pair $(\tau,T)$ in $\mathbf{Y}_{\mbox{\tiny hk}}$, $\mathbf{Y}_{\mbox{\tiny 2-r}}$
or $\mathbf{Y}_{\mbox{\tiny 2-c}}$, the algorithm aims to construct
the tableau $\tau$ by inserting successively $1,2,3...$ in an empty tableau,
according to rules depending on $\tau$ and $T$. The algorithm
either succeeds or fails to reconstruct $\tau$. We say that the pair $(\tau,T)$
is constructible if it succeeds. Then, our purpose is to show:


\begin{theorem}
Let $(\tau,T)\in\mathbf{Y}_{\mbox{\rm\tiny
hk}}\cup\mathbf{Y}_{\mbox{\rm\tiny
2-r}}\cup\mathbf{Y}_{\mbox{\rm\tiny 2-c}}$. Then, we have $\tau\in
T$ if and only if the pair $(\tau,T)$ is constructible.
\end{theorem}

We treat successively the two-row, two-column and hook cases.
The algorithm in the two-row case is more simple. 
In section \ref{new-charac-int-codim-one}, we will use this algorithm 
in order to study the intersections of components in the two-row case.
By presenting the algorithms in the hook and two-column cases,
we want to highlight that similar iterative procedures hold also in these two cases.
For the reason that we will not invoke the latter two algorithms in the subsequent sections, 
we will spare technical details 
especially in the proof of Theorem \ref{theorem-constructibility-2column}, regarding the two-column case.

The arguments in this section are purely combinatorial, they rely on the
criteria provided in the previous section.

\subsection{Constructibility in the two-row case}

\label{constructibility-2row-case}

Fix $(\tau,T)\in \mathbf{Y}_{\mbox{\tiny 2-r}}$. We define an
algorithm aiming to reconstruct $\tau$ as the final term of a
sequence of tableaux $\theta_1,\theta_2,...$ obtained by inserting
successively the entries $1,2,...$ 
starting with the empty tableau
$\theta_0$. 
For $i\in\{0,...,n\}$ the tableau $\theta_i$ is numbered from
$1,...,i$ and it has the following properties:
\begin{itemize}
\item[(2r-a)] Each row of $\theta_i$ is a sequence of numbered
boxes separated by blanks (some empty boxes). For $p\in\{1,2\}$,
the entries of the $p$-th row of $\theta_i$ are in increasing order
from left to right and they also occur in the $p$-th row of
$\tau$.
\item[(2r-b)] The columns of $\theta_i$ being numbered from left
to right, the number of entries in the $q$-th column of $\theta_i$
equals the number of entries $\leq i$ in the $q$-th column of $T$,
for any $q\geq 1$.
\end{itemize}

\smallskip
\noindent For $i\in\{1,...,n\}$ suppose we have constructed
$\theta_{i-1}$ that satisfies properties (2r-a) and (2r-b).
According to (2r-b), the number of entries in the $q$-th column of
$\theta_{i-1}$ is weakly decreasing from left to right. In
particular the full columns (two entries) are concentrated at the
left of $\theta_{i-1}$. The tableau $\theta_{i-1}$ has the
following aspect:
\[
\theta_{i-1}=
\begin{array}{|c|c|c|c|c|}
\cline{1-1} \cline{3-3} \cline{5-5}
*\ * & \quad & *\ *\ * & \quad & *\ *\ * \\
\hline
\multicolumn{2}{|c|}{*\ *\ *\ *} & \quad & *\ *\ *  \\
\cline{1-2} \cline{4-4}
\end{array}
\]
where the symbol $\begin{tyoung}{1} \tligne{1}{*\ *\ *}
\end{tyoung}$ stands for a strip of numbered boxes. The two strips which begin at the
first column are called {\em ``in-place''} strips (possibly one of
the two is empty). The other strips are called {\em ``to-be-placed''}
strips. We form $\theta_i$ from $\theta_{i-1}$ by inserting $i$
according to the following rule.

\smallskip\noindent (1)
First case: $i$ belongs to the first row of $T$. The entry $i$ is
then put in a new column on the right of $\theta_{i-1}$ in the top
or bottom place depending on whether $i$ is in the first or in the
second row of $\tau$:
\begin{eqnarray}
\theta_{i}=
\begin{array}{|c|c|c|c|c|c|}
\cline{1-1} \cline{3-3} \cline{5-6}
*\,* & \quad & *\,*\,* & \quad & *\,*\,* & i \\
\hline
\multicolumn{2}{|c|}{*\,*\,*\,*} & \quad & *\,*\,*  \\
\cline{1-2} \cline{4-4}
\end{array} \nonumber \\
\mbox{ \ or \ } \theta_{i}=\begin{array}{|c|c|c|c|c|c|}
\cline{1-1} \cline{3-3} \cline{5-5}
*\,* & \quad & *\,*\,* & \quad & *\,*\,*  \\
\hline
\multicolumn{2}{|c|}{*\,*\,*\,*} & \quad & *\,*\,* & \quad & i \\
\cline{1-2} \cline{4-4} \cline{6-6}
\end{array} \nonumber
\end{eqnarray}

\smallskip \noindent (2) Second case: $i$ belongs to the second row of $T$.
Suppose $i$ lies in the $p$-th row of $\tau$, for $p\in\{1,2\}$.
Then $i$ has to be inserted in the $p$-th row of $\theta_{i-1}$.
We state this failure case:

\smallskip
\noindent {\em (Failure case)} {\em (With $i$ in the second row of
$T$ and in the $p$-th row of $\tau$.) If the $p$-th box of the
last column of $\theta_{i-1}$ is nonempty, then the algorithm
fails.}

\smallskip
\noindent Suppose this failure case does not arise. Then, the entry $i$ is put at
the end of the last strip of the $p$-th row. Next, from right to
left, the first entries of the ``to-be-placed'' strips are
successively pushed to the left at the end of the preceding strip
(at the beginning of the row if there is no preceding strip):
\begin{eqnarray}
\begin{array}{|c|c|c|c|c|c|c|c|c|c|c|}
\cline{1-1} \cline{3-4} \cline{7-8}
\!* & \leftarrow\!\!\!-\!\!\!-\!\!\!\!\!\! & a & \!*\,* & \multicolumn{2}{|c|}{\longleftarrow\!\!\!\trait\!\!\!\!\!\!\!\!}  & c & \!*\,*  \\
\cline{1-8} \cline{10-10}
\multicolumn{2}{|c|}{\!*\,*\,*\,*} & \multicolumn{2}{|c|}{\longleftarrow\!\!\!\trait\!\!\!\!\!\!\!\!} & b & \!*\,* & \multicolumn{3}{|c|}{\longleftarrow\!\!\!\ttrait\!\!\!\!\!} & i  \\
\cline{1-2} \cline{5-6} \cline{10-10}
\end{array} \nonumber \\
\ \ \mbox{gives}\ \ \theta_i=
\begin{array}{|c|c|c|c|c|c|c|c|c|c|c|}
\cline{1-2} \cline{5-6} \cline{9-9}
\!* & a & \multicolumn{2}{|c|}{\quad} & \!*\,* & c & \multicolumn{2}{|c|}{\quad} & \!*\,* \\
\hline
\multicolumn{3}{|c|}{\!*\,*\,*\,\,*\,} & b & \multicolumn{2}{|c|}{\quad} & \!*\,* & i \\
\cline{1-4} \cline{7-8}
\end{array} \nonumber
\end{eqnarray}

\smallskip
If no failure case occurs while $i$ runs over $\{1,...,n\}$, then we
get a final tableau $\theta_n$ with entries $1,...,n$. According
to (2r-a) and (2r-b) we have $\theta_n=\tau$. Then we say that the
pair {\em $(\tau,T)$ is constructible.}

\begin{example}
Let $T=\young(12347,5689)$.

\smallskip
\noindent (a) First, we suppose $\tau=\young(23689,1457)$.
\begin{center}
$\niveau{\theta_1=}$\young(1)\quad
$\niveau{\theta_2=}$\young(:2,1)\quad
$\niveau{\theta_3=}$\young(:23,1)\quad
$\niveau{\theta_4=}$\young(:2,1)\young(3,:4)
\end{center}
The number $5$ is in the second row of $T$ and in the second row
of $\tau$. The second box of the last column of $\theta_4$ is not
empty. Due to the failure rule, the algorithm fails.

\smallskip
\noindent
(b) Suppose now $\tau=\young(14689,2357)$. \\
Likewise
\begin{center}
$\niveau{\theta_4=}$\young(1,:2)\young(:4,3)\quad
$\niveau{\theta_5=}$\young(14,2)\young(35)\quad
$\niveau{\theta_6=}$\young(146,23)\young(5)
\quad$\niveau{\theta_7=}$\young(146,23)\young(57)
\end{center}
\begin{center}
$\niveau{\theta_8=}$\young(1468,235)\young(7)\quad
$\niveau{\mbox{and finally\quad
$\theta_9=$}}$\young(14689,2357)$\niveau{\ =\tau.}$
\end{center}
Thus in this example $(\tau,T)$ is constructible.
\end{example}

\medskip

Let $\theta_1,\theta_2,...$ be the intermediate tableaux obtained
by applying the algorithm to the pair $(\tau,T)$. (Implicitly, no
fail has occurred while $\theta_i$ is well defined.) Each tableau
$\theta_i$ is formed by strips that we number in the following
manner:
\[
\theta_i\ =
\begin{array}{|c|c|c|c|c|}
\cline{1-2} \cline{4-4} \multicolumn{2}{|c|}{\ \ \mbox{1st strip}\
\ } & \qquad & \mbox{4th strip} &
\multicolumn{1}{|c}{\qquad} \\
\hline
\ \mbox{2nd strip}\  & \qquad & \mbox{3rd strip} & \qquad \qquad & \mbox{5th strip} \\
\cline{1-1} \cline{3-3} \cline{5-5}
\end{array}
\cdots
\]
or possibly
\[
\theta_i\ =
\begin{array}{|c|c|c|c|c|}
\cline{1-1} \cline{3-3} \cline{5-5}
\ \mbox{1st strip}\  & \qquad & \mbox{3rd strip} & \qquad \qquad & \mbox{5th strip} \\
\hline \multicolumn{2}{|c|}{\ \ \mbox{2nd strip}\ \ } & \qquad &
\mbox{4th strip} &
\multicolumn{1}{|c}{\qquad} \\
\cline{1-2} \cline{4-4}
\end{array}
\cdots
\]
The first two strips are left-justified, possibly one of the two is
empty. We denote by $s_1(\theta_i)$ (resp. by $s_2(\theta_i)$) the
length of the first strip (resp. the second strip) of $\theta_i$.
We show the following

\begin{lemma}
\label{constructibility-2row-lemma} Suppose that there are $i,j$
such that $s_1(\theta_i)<s_2(\theta_i)$ and $s_1(\theta_j)>
s_2(\theta_j)$, then the pair $(\tau,T)$ satisfies condition
\ref{A-2row-case}\,(1).
\end{lemma}

\begin{proof}
%
By the construction of $\theta_k$ from $\theta_{k-1}$, we see
that, for $p\in\{1,2\}$, either $s_p(\theta_k)=s_p(\theta_{k-1})$
or $s_p(\theta_k)=s_p(\theta_{k-1})+1$. Thus we always have
\[
(s_1(\theta_k)-s_2(\theta_k))-(s_1(\theta_{k-1})-s_2(\theta_{k-1}))\in\{-1,0,1\}.
\]
Since $s_1(\theta_i)-s_2(\theta_i)<0$ and
$s_1(\theta_j)-s_2(\theta_j)> 0$, there is
$k\in\{\min(i,j)+1,...,\max(i,j)-1\}$ such that
$s_1(\theta_k)-s_2(\theta_k)=0$. 
The equality $s_1(\theta_k)=s_2(\theta_k)$ implies
that, starting with the third one, all the strips of $\theta_k$ are empty.
Then the shape of $\theta_k$ is a rectangular Young diagram with
two rows of length $k/2$ and, according to (2r-a) and (2r-b), this
is also the shape of the subtableaux $\tau[1,...,k]$ and
$T[1,...,k]$. Therefore, the pair $(\tau,T)$ satisfies condition
\ref{A-2row-case}\,(1). 
\end{proof}

We prove the following result.

\begin{theorem}
\label{theorem-constructibility-2row} Let
$(\tau,T)\in\mathbf{Y}_{\mbox{\rm\tiny 2-r}}$. Then, we have
$\tau\in T$ if and only if the pair $(\tau,T)$ is constructible.
\end{theorem}

\begin{proof}
Write $n=|\tau|=|T|$. First, let us show that, if the pair
$(\tau,T)$ is constructible, then we have $(\tau,T)\in
\mathbf{\hat A}_{\mbox{\rm\tiny 2-r}}$. To do this, we have to
show that $(\tau,T)$ satisfies either \ref{A-2row-case}\,(1) or
\ref{A-2row-case}\,(2). If \ref{A-2row-case}\,(2) does not hold,
then $1$ is in the second row of $\tau$ or there is
$i\in\{2,...,n\}$ such that $T[1,...,i]$ is rectangular. In the
former case $\theta_1=\begin{array}{c} \cdot \\ \young(1)
\end{array}$, which gives $s_1(\theta_1)=0$ and $s_2(\theta_1)=1$.
Since $(\tau,T)$ is constructible, we have $\theta_n=\tau$, hence
$s_1(\theta_n)\geq s_2(\theta_n)$. Then, by Lemma
\ref{constructibility-2row-lemma}, the pair $(\tau,T)$ satisfies
condition \ref{A-2row-case}\,(1). In the latter case it follows
from (2r-b) that $\theta_i$ has a rectangular shape and it follows
from (2r-a) that $\theta_i$ coincides with the subtableau
$\tau[1,...,i]$. Thus $T[1,...,i]$ and $\tau[1,...,i]$ are both
rectangular and condition \ref{A-2row-case}\,(1) holds. Therefore,
$(\tau,T)$ satisfies \ref{A-2row-case}\,(1) or
\ref{A-2row-case}\,(2).

By Theorem \ref{theorem-KR-2-row} and by definition of the set
$\mathbf{A}_{\mbox{\rm\tiny 2-r}}$, in order to prove the theorem,
it is sufficient to show that any $(\tau,T)\in\mathbf{\hat
A}_{\mbox{\tiny 2-r}}$ is constructible if and only if the pair or
the two pairs forming $\eta(\tau,T)$ are constructible. We
distinguish three cases, corresponding to the three cases of the
construction of $\eta(\tau,T)$ in \ref{A-2row-case}.

\medskip
\noindent (a) Suppose that there is $i<n$ such that
$\tau[1,...,i]$ and $T[1,...,i]$ are both rectangular, and choose
$i$ maximal for this property.

Then, let $\eta(\tau,T)=\{(\tau',T'),(\tau'',T'')\}$ be defined as
in \ref{A-2row-case}\,(a). Let $\theta_1,\theta_2,...$ (resp.
$\theta'_1,\theta'_2,...$) (resp. $\theta''_1,\theta''_2,...$) be
the tableaux obtained by applying the algorithm to the pair
$(\tau,T)$ (resp. $(\tau',T')$) (resp. $(\tau'',T'')$). It
immediately follows from the definition of the algorithm that
$\theta_j=\theta'_j$ for $j\leq i$. Therefore $(\tau',T')$ is
constructible if and only if the algorithm for $(\tau,T)$ does not
fail in the first $i$ steps, and in this case $\theta_i$ coincides with
the tableau $\tau'$. Assume that $(\tau',T')$ is constructible. It
follows from the definition of the algorithm that, for $j\geq
i+1$, the tableau $\theta_i$ is obtained by adding
$\theta''_{j-i}$ on the right of $\tau'$ and replacing every entry
$k$ in $\theta''_{j-i}$ by $k+i$. Thus, it is easy to see that
$(\tau'',T'')$ is constructible if and only if the algorithm for
$(\tau,T)$ does not fail also until the end. Finally, we get that
$(\tau,T)$ is constructible if and only if both $(\tau',T')$ and
$(\tau'',T'')$ are constructible.

\medskip
\noindent (b) Suppose that $\tau[1,...,i]$ and $T[1,...,i]$ are
both rectangular only for $i=n$.

If $\tilde\tau$ denotes the tableau obtained by switching the two
rows of $\tau$, then it is immediate from the definition of the
algorithm that $(\tau,T)$ is constructible if and only if
$(\tilde\tau,T)$ is constructible. Up to dealing with $\tilde\tau$
instead of $\tau$, we may suppose that $n$ is the last entry in
the second row of $\tau$. Then $\eta(\tau,T)=\{(\tau',T')\}$ with
$\tau'=\tau[1,...,n-1]$ and $T'=T[1,...,n-1]$. Let
$\theta_1,\theta_2,...$ (resp. $\theta'_1,\theta'_2,...$) be the
tableaux obtained by applying the algorithm to the pair $(\tau,T)$
(resp. $(\tau',T')$). It follows from the definition of the
algorithm that we have $\theta_i=\theta'_i$ for $i\leq n-1$. Thus,
$(\tau',T')$ is constructible if and only if the algorithm for
$(\tau,T)$ does not fail in the first $n-1$ steps, and in this case
$\theta_{n-1}=\tau'$. In the last step of the algorithm relative
to the pair $(\tau,T)$, the entry $n$ has to be inserted in the
second row of $\theta_{n-1}$. The last box in the second row of
$\theta_{n-1}=\tau'$ is empty, hence the algorithm does not fail in
the $n$-th step whenever it has not failed before. Finally, we get
that $(\tau,T)$ is constructible if and only if $(\tau',T')$ is
constructible.

\medskip
\noindent (c) Finally, suppose that \ref{A-2row-case}\,(2) holds.

Let $\eta(\tau,T)=(\tau',T')$ be as in \ref{A-2row-case}\,(c). Let
$\theta_1,\theta_2,...$ (resp. $\theta'_1,\theta'_2,...$) be the
tableaux obtained by applying the algorithm to the pair $(\tau,T)$
(resp. $(\tau',T')$). By hypothesis, $1$ lies in the first row of
$\tau$, hence $s_1(\theta_1)=1$ and $s_2(\theta_1)=0$. By Lemma
\ref{constructibility-2row-lemma}, we get
$s_1(\theta_i)>s_2(\theta_i)$ for any $i$. Thus $\theta_i$ has the
following form:
\[
\theta_i\ =
\begin{array}{|c|c|c|c|c|c|}
\cline{1-3} \cline{5-5} 1 & \multicolumn{2}{|c|}{\ \ A_1 \ } &
\qquad & A_4 &
\multicolumn{1}{|c}{\qquad} \\
\hline
\multicolumn{2}{|c|}{\ A_2\  } & \qquad & A_3 & \qquad \qquad & A_5 \\
\cline{1-2} \cline{4-4} \cline{6-6}
\end{array}
\cdots
\]
with $\#A_1\geq \#A_2$, where $\#A$ denotes the number of boxes in
$A$. Let $\tilde\theta_i$ be the tableau obtained by removing $1$
and moving $A_1,A_3,A_4,...$ by one rank to the left:
\[
\tilde\theta_i\ =
\begin{array}{|c|c|c|c|c|}
\cline{1-2} \cline{4-4} \multicolumn{2}{|c|}{\ \ A_1 \ } & \qquad
& A_4 &
\multicolumn{1}{|c}{\qquad} \\
\cline{1-5}
\multicolumn{1}{|c|}{\ A_2\  } & \quad & A_3 & \qquad \qquad & A_5 \\
\cline{1-1} \cline{3-3} \cline{5-5}
\end{array}
\cdots
\]
Possibly the gap between $A_2$ and $A_3$ in $\tilde\theta_i$ is
empty. The tableau $\tilde\theta_i$ is numbered by $2,...,i$. Let
us show the following claim.

\medskip
\noindent {\bf Claim.} {\em Let $i\in\{2,...,n\}$. The algorithm relative to $(\tau,T)$
does not fail in the first $i$ steps if and only if the
one relative to $(\tau',T')$ does not fail in the first $i-1$ steps.
Moreover, the tableau $\theta'_{i-1}$ is obtained from
$\tilde\theta_i$ by replacing every entry $j$ by $j-1$.}

\medskip

The proof of the claim ends the proof of the Theorem, because it
implies in particular that $(\tau,T)$ is constructible if and only
if $(\tau',T')$ is.

We show the claim by induction on $i\geq 2$. 
Condition \ref{A-2row-case}\,(2) implies in particular that $2$
lies in the first row of $T$, hence the algorithm for $(\tau,T)$
does not fail in the first two steps, and the claim is immediate for
$i=2$. Suppose the claim holds until $i-1\in\{2,...,n-1\}$
and show it for $i$. 
The row number of $i$ in $T$ and the row number of $i-1$ in $T'$ coincide.
Moreover, the row number of $i$ in $\tau$ and the row number of $i-1$ in $\tau'$ are the same, say $p$. 
Assume that $i$ lies in the first row of $T$, then $i-1$ is in the first row of $T'$ and both algorithms do not fail.
Now, assume that $i$ lies in the second row of $T$.
The tableau $T[1,...,i]$ being non-rectangular, it follows from (2r-b)
that also $\tilde\theta_{i-1}$ is non-rectangular, thus the last 
columns of $\theta_{i-1}$ and $\tilde\theta_{i-1}$ coincide.
By induction hypothesis, we get that the $p$-th box of the last column of $\theta_{i-1}$ is
empty if and only if the $p$-th box of the last column of
$\theta'_{i-2}$ is empty. Thus the algorithm for $(\tau,T)$ fails
at the $i$-th step if and only if the algorithm for $(\tau',T')$
fails at the $(i-1)$-th step. In each case, assuming that both algorithms
do not fail, the fact that $\theta'_{i-1}$ coincides with
$\tilde\theta_i$ up to replacing each entry $j$ by $j-1$ follows by
induction hypothesis and the definition of the algorithm. 
\end{proof}

\subsection{Constructibility in the two-column case}

\label{section-constructibility-2column-case}
\label{constructibility-2column-case}

Let $(\tau,T)\in\mathbf{Y}_{\mbox{\tiny 2-c}}$. Write
$n=|\tau|=|T|$. Let $r$ be the common length of the first column
of $\tau,T$. Let $Y$ be the Young diagram with two columns of $r$
empty boxes
\[
Y=\begin{array}{|c|c|} \hline
\ & \ \\
\hline
\ & \ \\
\hline
\vdots & \vdots \\
\hline
\ & \ \\
\hline
\end{array}
\]
The algorithm which we define aims to reconstruct $\tau$ as the
final term of a sequence of tableaux $\theta_1,\theta_2,...$
obtained by inserting successively the entries $1,2,...$ in the
diagram $Y$.

We consider tableaux which are partial numberings of $Y$. The
tabl\-eaux $\tau$ and $T$ and their subtableaux $\tau[i,...,j]$
and $T[i,...,j]$ of entries $i,...,j$, for any integers $1\leq
i\leq j\leq n$, are considered as partial numberings of $Y$. Let
$\theta_0$ be the empty numbering of $Y$. The tableau $\theta_i$
obtained at the $i$-th step of the algorithm is a partial
numbering of entries $1,...,i$. Set
$\overline{\mathbb{N}}=\{0,1,2,...\}\cup\{\infty\}$. By convention
$\infty=\infty+1=\infty-1$ and $a<\infty$ for any
$a\in\mathbb{N}$. Let $P=\{1,...,r\}$ be the rows of $Y$ numbered
from top to bottom. For each step $i=1,2,...$ of the algorithm and
for each row $p\in P$ we define in addition an index
$f_i(p)\in\overline{\mathbb{N}}$. As an initialization, we set
$f_0(p)=\infty$ for any $p\in P$.

For $i\in\{1,...,n\}$ the tableau $\theta_i$ has the following properties:
\begin{itemize}
\item[(2c-a)] For $p\in P$ the entries of the $p$-th row of
$\theta_i$ are in increasing order to the right and they also
occur in the $p$-th row of $\tau$.
\item[(2c-b)] For $q\in\{1,2\}$ the number of entries in the $q$-th column of
$\theta_i$ is equal to the number of entries $\leq i$ in the
$q$-th column of $T$.
\item[(2c-c)] Let $P_i\subset P$ be the subset of rows of $\theta_i$
whose first box is nonempty and let $Q_i\subset P$ be the subset
of the other non-empty rows. For $p\in P$ we have
$f_i(p)\in\mathbb{N}$ if and only if $p\in P_i$. In addition,
$\mathrm{max}_{p\in P_i}f_i(p)=\#Q_i$.
\end{itemize}
For $i\in\{1,...,n\}$ suppose that we have constructed
$\theta_{i-1}$ that satisfies these properties. For instance,
$\theta_{i-1}$ has the following aspect:
\[
\theta_{i-1}=\begin{array}{|c|c|l} \cline{1-2}
 & * & { }^{\infty} \\
 \cline{1-2}
* & * & { }^{f_{i\!-\!1}(2)} \\
\cline{1-2}
 & & { }^{\infty} \\
\cline{1-2}
 * & & { }^{f_{i\!-\!1}(4)} \\
 \cline{1-2}
 * & & { }^{f_{i\!-\!1}(5)} \\
 \cline{1-2}
 \end{array}
\]
The symbol $\begin{tyoung}{1}\tligne{1}{*}\end{tyoung}$ stands for a
numbered box. We have written the values $f_{i-1}(p)$ at the
right. In this example we have $P_{i-1}=\{2,4,5\}$ and
$Q_{i-1}=\{1\}$.

\smallskip
\noindent We form $\theta_i$ from $\theta_{i-1}$ by inserting $i$
according to the following rules.

\smallskip
\noindent (0) Let $p_i\in P$ be the number of the row of $\tau$ to
which $i$ belongs. The entry $i$ has to be put in the $p_i$-th row
of $\theta_{i-1}$. We state this failure case:

\smallskip
\noindent {\em (First failure case) If the second box of the $p_i$-th
row of $\theta_{i-1}$ is not empty, then the algorithm fails.}

\smallskip
\noindent Suppose that no failure of this type arises. Then
put $i$ in the second box of the $p_i$-th row of $\theta_{i-1}$.
Let $\theta'_{i}$ be the tableau so-obtained.

\smallskip
\noindent (1) First case: $i$ belongs to the second column of $T$.
Then define $\theta_i=\theta'_{i}$. Let $p\in P$. If
$f_{i-1}(p)<f_{i-1}(p_i)$, then set $f_i(p)=f_{i-1}(p)+1$. If
$f_{i-1}(p)\geq f_{i-1}(p_i)$, then set $f_i(p)=f_{i-1}(p)$.

\smallskip
\noindent For $\theta_{i-1}$ as in the previous figure, we could
have for example
\[
\theta_{i}=\begin{array}{|c|c|l} \cline{1-2}
 & * & { }^{\infty} \\
 \cline{1-2}
* & * & { }^{f_{i}(2)=f_{i\!-\!1}(2)+1} \\
\cline{1-2}
 & i & { }^{\infty} \\
\cline{1-2}
 * & & { }^{f_{i}(4)=f_{i\!-\!1}(4)+1} \\
 \cline{1-2}
 * & & { }^{f_{i}(5)=f_{i\!-\!1}(5)+1} \\
 \cline{1-2}
\end{array}
\quad\mbox{or}\quad \theta_{i}=\begin{array}{|c|c|l} \cline{1-2}
 & * & { }^{\infty} \\
 \cline{1-2}
* & * & { }^{f_{i}(2)} \\
\cline{1-2}
 & & { }^{\infty} \\
\cline{1-2}
 * & i & { }^{f_{i}(4)=f_{i-1}(4)} \\
 \cline{1-2}
 * &  & { }^{f_{i}(5)} \\
 \cline{1-2}
\end{array}
\]
In the case of the first tableau, the finite indices have been
incremented. In the case of the second tableau, for $p\in\{2,5\}$,
we have $f_i(p)=f_{i-1}(p)$ or $f_i(p)=f_{i-1}(p)+1$ depending on
whether $f_{i-1}(p)\geq f_{i-1}(4)$ or $f_{i-1}(p)< f_{i-1}(4)$.
\\
Properties (2c-a), (2c-b), (2c-c) hold easily for $\theta_i$.

\smallskip
\noindent (2) Second case: $i$ belongs to the first column of $T$.
We construct $\theta_i$ from the intermediate tableau $\theta'_i$
by pushing in addition one entry to the left. We state:

\smallskip
\noindent {\em (Second failure case) Suppose that $i$ is in the first
column of $T$. If $f_{i-1}(p_i)=0$, then the algorithm fails.}

\smallskip
\noindent Suppose no failure arises. Let $F'_i$ be the set of
the entries of the second column of $\theta'_i$ which have an empty
box on their left. The set $F'_i$ is nonempty.
(Indeed, we have either $f_{i-1}(p_i)=\infty$ or $f_{i-1}(p_i)\in\mathbb{N}\setminus\{0\}$.
In the former case, $i\in F'_i$.
In the latter case, due to (2c-c), $Q_{i-1}\not=\emptyset$, 
hence also $F'_i\not=\emptyset$.)
By (2c-a) each $j\in F'_i$
belongs to the first column of $\tau$. If possible, choose $j\in
F'_i$ with a minimal right-neighbor entry in $\tau$. If no element
of $F'_i$ has a right-neighbor entry in $\tau$, then choose $j\in
F'_i$ with $p_j$ minimal. Let $\theta_i$ be the tableau obtained
from $\theta'_i$
by pushing $j$ by one rank to the left. \\
We set $f_i(p_j)=0$. Let $p\in P$ be such that $p\not=p_j$. If
$f_{i-1}(p)< f_{i-1}(p_i)$, then set $f_i(p)=f_{i-1}(p)$. If
$f_{i-1}(p)\geq f_{i-1}(p_i)$, then set $f_i(p)=f_{i-1}(p)-1$.

\smallskip
\noindent For $\theta_{i-1}$ as in the previous figure, we get for
example:
\[
\theta'_{i}=\begin{array}{|c|c|l} \cline{1-2}
 & j & { }^{\infty} \\
 \cline{1-2}
* & * & { }^{f_{i\!-\!1}(2)} \\
\cline{1-2}
 & i & { }^{\infty} \\
\cline{1-2}
 * & & { }^{f_{i\!-\!1}(4)} \\
 \cline{1-2}
 * & & { }^{f_{i\!-\!1}(5)} \\
 \cline{1-2}
\end{array}
\quad\mbox{or}\quad \theta'_{i}=\begin{array}{|c|c|l} \cline{1-2}
 & j & { }^{\infty} \\
 \cline{1-2}
* & * & { }^{f_{i\!-\!1}(2)} \\
\cline{1-2}
 & & { }^{\infty} \\
\cline{1-2}
 * & i & { }^{f_{i\!-\!1}(4)} \\
 \cline{1-2}
 * &  & { }^{f_{i\!-\!1}(5)} \\
 \cline{1-2}
\end{array}
\]
and one entry has yet to be moved to the left. In the case of the
first tableau, no failure can occur, since $f_{i-1}(p_i)=\infty$. We
have $F'_i=\{i,j\}$. We move $i$ or $j$ to the left. If one of
them has an entry on its right in $\tau$ we move the one for which
this is minimal. Otherwise we move $j$, since the row-number of
$j$ is smaller. As $f_{i-1}(p_i)=\infty$, the indices do not
change. In the case of the second tableau, a failure occurs if and
only if $f_{i-1}(4)=0$. If no failure occurs, then $j$ is moved to
the left and we set $f_i(1)=0$. The other indices which are
greater or equal than $f_{i-1}(p_i)$
are reduced by one unit. \\
Easily properties (2c-a), (2c-b) and (2c-c) hold for $\theta_i$.

\smallskip

If no failure case occurs while $i$ runs over $\{1,...,n\}$, then we
get a final tableau $\theta_n$ with entries $1,...,n$. According
to (2c-a) and (2c-b) we have $\theta_n=\tau$. Then we say that the
pair $(\tau,T)$ is {\em constructible.}

\begin{example}
(a) Let $\tau$ and $T$ be the tableaux
\[
\tau=\young(35,14,2) \quad\mbox{and}\quad T=\young(12,34,5)
\]
We get the following $\theta_1,...,\theta_5$. We write the indices
$f_i(p)$ with roman numerals at the right.
\[
\begin{array}{|c|c|c}
\cline{1-2}
 & & {}^{\infty} \\
\cline{1-2}
 1 & \ \, & {}^{0} \\
\cline{1-2}
 & & {}^{\infty} \\
\cline{1-2} \multicolumn{2}{c}{\theta_1}
\end{array}
\quad
\begin{array}{|c|c|l}
\cline{1-2}
 & & {}^{\infty} \\
\cline{1-2}
 1 & & {}^{I} \\
\cline{1-2}
 & 2 & {}^{\infty} \\
\cline{1-2} \multicolumn{2}{c}{\theta_2}
\end{array}
\quad
\begin{array}{|c|c|l}
\cline{1-2}
3 & & {}^{0} \\
\cline{1-2}
1 & & {}^{I} \\
\cline{1-2}
 & 2 & {}^{\infty} \\
\cline{1-2} \multicolumn{2}{c}{\theta_3}
\end{array}
\quad
\begin{array}{|c|c|l}
\cline{1-2}
3 & & {}^{I} \\
\cline{1-2}
1 & 4 & {}^{I} \\
\cline{1-2}
 & 2 & {}^{\infty} \\
\cline{1-2} \multicolumn{2}{c}{\theta_4}
\end{array}
\quad
\begin{array}{|c|c|l}
\cline{1-2}
3 & 5 & {}^{0} \\
\cline{1-2}
1 & 4 & {}^{0} \\
\cline{1-2}
2 & & {}^{0} \\
\cline{1-2} \multicolumn{2}{c}{\theta_5}
\end{array}
\]
No failure has occurred and the pair $(\tau,T)$ is thus
constructible.

\smallskip
\noindent (b) Suppose that $\tau$ and $T$ are the tableaux
\[
\tau=\young(26,35,4,1) \quad\mbox{and}\quad T=\young(12,34,5,6)
\]
\[
\begin{array}{|c|c|c}
\cline{1-2}
 & & {}^{\infty} \\
\cline{1-2}
 & & {}^{\infty} \\
\cline{1-2}
 & & {}^{\infty} \\
\cline{1-2}
 1 & \ \, & {}^{0} \\
\cline{1-2} \multicolumn{2}{c}{\theta_1}
\end{array}
\quad
\begin{array}{|c|c|l}
\cline{1-2}
 & 2 & {}^{\infty} \\
\cline{1-2}
 & & {}^{\infty} \\
\cline{1-2}
 & & {}^{\infty} \\
\cline{1-2}
 1 & & {}^{I} \\
\cline{1-2} \multicolumn{2}{c}{\theta_2}
\end{array}
\quad
\begin{array}{|c|c|l}
\cline{1-2}
 & 2 & {}^{\infty} \\
\cline{1-2}
3 & & {}^{0} \\
\cline{1-2}
 & & {}^{\infty} \\
\cline{1-2}
1 & & {}^{I} \\
\cline{1-2} \multicolumn{2}{c}{\theta_3}
\end{array}
\quad
\begin{array}{|c|c|l}
\cline{1-2}
 & 2 & {}^{\infty} \\
\cline{1-2}
3 & & {}^{I} \\
\cline{1-2}
 & 4 & {}^{\infty} \\
\cline{1-2}
1 &  & {}^{II} \\
\cline{1-2} \multicolumn{2}{c}{\theta_4}
\end{array}
\quad
\begin{array}{|c|c|l}
\cline{1-2}
2 & & {}^{0} \\
\cline{1-2}
3 & 5 & {}^{0} \\
\cline{1-2}
 & 4 & {}^{\infty} \\
\cline{1-2}
1 &  & {}^{I} \\
\cline{1-2} \multicolumn{2}{c}{\theta_5}
\end{array}
\quad
\begin{array}{|c|c|}
\cline{1-2}
2 & 6  \\
\cline{1-2}
3 & 5  \\
\cline{1-2}
 & 4  \\
\cline{1-2}
1 &  \\
\cline{1-2} \multicolumn{2}{c}{\theta'_6}
\end{array}
\]
The entry $6$ is inserted in the first row of $\theta_5$. In
addition $6$ belongs to the second row of $T$ hence one entry has
to be moved to the left. But we have $f_5(1)=0$, hence we are in
the second failure case. Finally $(\tau,T)$ is not constructible.
\end{example}

We have the following result.

\begin{theorem}
\label{theorem-constructibility-2column} Let
$(\tau,T)\in\mathbf{Y}_{\mbox{\rm\tiny 2-c}}$. Then, we have
$\tau\in T$ if and only if the pair $(\tau,T)$ is constructible.
\end{theorem}

%

The proof of Theorem \ref{theorem-constructibility-2column} requires many technical verifications.
For the sake of conciseness, we will underline the main steps of the proof without going into details.
The proof relies on the set
${\mathbf A}_{\mbox{\tiny 2-c}}$ introduced in section
\ref{def-A-2column}. Recall that $\mathrm{st}(\tau)$ is the
standard tableau obtained by putting each column of $\tau$ in
increasing order. Considering the inductive definition of the set
${\mathbf A}_{\mbox{\tiny 2-c}}$ and applying Theorem
\ref{theorem-A-2column}, Theorem
\ref{theorem-constructibility-2column} is a consequence of the
following

\begin{proposition}
\label{proposition-constructibility-2column} Let
$(\tau,T)\in\mathbf{Y}_{\mbox{\rm\tiny 2-c}}$.
\begin{itemize}
\item[(a)] If $T=\mathrm{st}(\tau)$, then the pair $(\tau,T)$ is constructible.
\item[(b)] If $(\tau,T)$ is constructible, then we have $T=\mathrm{st}(\tau)$ or $(\tau,T)\in\mathbf{\hat A}_{\mbox{\rm\tiny 2-c}}$.
\item[(c)] Assume $(\tau,T)\in \mathbf{\hat A}_{\mbox{\rm\tiny 2-c}}$.
Then the pair $(\tau,T)$ is constructible if and only if
$\eta(\tau,T)$ is constructible.
\end{itemize}
\end{proposition}

Proposition
\ref{proposition-constructibility-2column}\,(a) is a consequence
of the following lemma, which easily follows,
by induction, from the definition of the algorithm.

\begin{lemma}
\label{lemme_constructibilite_tauT} Let $k\in\{1,...,n\}$. If each
$l\in\{1,...,k\}$ is in the same column in $\tau$ and $T$, then
the algorithm does not fail in $\{1,...,k\}$. In addition
$\theta_k$ coincides with the subtableau $\tau[1,...,k]$ of
entries $1,...,k$, and we have $f_k(p)=0$ for all $p\in P_k$.
\end{lemma}

The next goal to show Proposition \ref{proposition-constructibility-2column}\,(b)
requires two preliminary observations.
Let $(\tau,T)\in \mathbf{Y}_{\mbox{\rm\tiny 2-c}}$ be such that $T\not=\mathrm{st}(\tau)$.
Take $i\in\{1,...,n\}$
minimal which does not lie in the same column in $\tau$ and
$T$. By Lemma \ref{lemme_constructibilite_tauT}, the algorithm
does not fail in the first $i-1$ steps. Recall that
$C_q(\tau)$ (resp. $C_q(T)$) is the set of entries in the
$q$-th column of $\tau$ (resp. of $T$). 
By Lemma \ref{lemme_constructibilite_tauT} and the second failure case, we have:

\begin{lemma}
\label{lemme_suivant} If the algorithm does not fail at the $i$-th
step, then $i\in C_1(\tau)\cap C_2(T)$.
\end{lemma}

Let $\nu_{\tau}:C_2(\tau)\rightarrow C_1(\tau)$ and
$\omega_{\tau}:C_1(\tau)\rightarrow C_2(\tau)\cup\{\infty\}$ be
the maps introduced in section \ref{def-A-2column}. The sets
$P_k,Q_k$ are those involved in condition (2c-c).
We need the following technical

\begin{lemma}
\label{lemme_Qknonvide} Let $k\geq i$ be such that $\nu_{\tau}(l)>i$
for any $l\in\{i+1,...,k\}\cap C_2(\tau)$. Suppose that the algorithm does not fail
in the first $k$ steps. Then the set $Q_k$ is nonempty and we
have $f_k(p)<\#Q_k= f_k(p')$ for any $p\in P_k\setminus P_i$ and
$p'\in P_i$.
\end{lemma}

We show Proposition \ref{proposition-constructibility-2column}\,(b):
assume that the pair $(\tau,T)$, $T\not=\mathrm{st}(\tau)$ is
constructible, and let us show that $(\tau,T)\in\mathbf{\hat A}_{\mbox{\tiny 2-c}}$. 
Lemma \ref{lemme_suivant} exactly means that $(\tau,T)$ satisfies condition \ref{def-A-2column}\,(1).
As $(\tau, T)$ is constructible, $Q_n=\emptyset$. Then, by Lemma
\ref{lemme_Qknonvide}, there is $j\in \{i+1,...,n\}\cap C_2(\tau)$
such that $\nu_{\tau}(j)\leq i$. This is condition \ref{def-A-2column}\,(2). 
Take $j$ minimal for this property. Also by Lemma \ref{lemme_Qknonvide}, $Q_{j-1}\not=\emptyset$.
Hence there is $i'\in\{i,...,j-1\}\cap C_1(\tau)$ which is in the second column of $\theta_{j-1}$. 
This implies that $\omega_{\tau}(i')>j$.
Condition \ref{def-A-2column}\,(3) is thus satisfied. Finally, we get $(\tau,T)\in\mathbf{\hat A}_{\mbox{\tiny 2-c}}$.

\medskip

It remains to show Proposition \ref{proposition-constructibility-2column}\,(c). Suppose
$(\tau,T)\in\mathbf{\hat A}_{\mbox{\tiny 2-c}}$. Let $j>i$ be the
integers involved in \ref{def-A-2column}\,(1)--(3). Recall that
the pair $\eta(\tau,T)=(\tilde{\tau},T)$ is defined in section
\ref{def-A-2column}, according to the following rule:
\begin{itemize}
\item[(a)] If there is $i'\in\{i+1,...,j-1\}\cap C_1(\tau)$ such
that $\omega_{\tau}(i)<\omega_{\tau}(i')$, then take $i'$ minimal
for this property and $\tilde{\tau}$ is obtained from $\tau$ by
switching $i$ and $i'$.
\item[(b)] Otherwise, $\tilde{\tau}$ is obtained from $\tau$ by switching $i$
and $j$.
\end{itemize}
We set $\itilde=i'$ in case (a) and $\itilde=j$ in case (b), so
that the tableau $\tilde\tau$ is obtained from $\tau$ by switching $i$
and $\itilde$. Let
$\tilde{p}_j,\tilde{\theta}_j,\tilde{f}_j,\tilde{P}_j,\tilde{Q}_j$
be the analogues of $p_j,\theta_j,f_j,P_j,Q_j$ for the pair
$(\tilde\tau,T)$. By Lemma \ref{lemme_constructibilite_tauT}, the
algorithms relative to $(\tau,T)$ and $(\tilde{\tau},T)$ do not
fail in the first $i-1$ steps.
The following lemma compares both algorithms from the $i$-th step.
It can be proved by induction. The proof, which we skip, 
consists of technical verifications, relying on the definition of the algorithm
and on Lemma \ref{lemme_Qknonvide} as another ingredient.

\begin{lemma} \label{lemme_echange_iitilde}
Let $k\in\{i,...,\itilde\}$.
\begin{itemize}
\item[(a)] The algorithm relative to the tableau $\tau$ does not fail
in the first $k$ steps if and only if the algorithm relative to
$\tilde{\tau}$ does not.
\end{itemize}
Suppose now that both algorithms do not fail in the first $k$ steps.
\begin{itemize}
\item[(b)] Case $i\leq k<\itilde$. The entries of
$\theta_k$ and $\tilde{\theta}_k$ have the same place in both
tableaux except $i$ which is in the second box of the $p_i$-th
row of $\theta_k$ and in the second box of the $p_{\itilde}$-th
row of $\tilde{\theta}_k$. We have $f_k(p)=\tilde{f}_k(p)$
$\forall p\notin P_i$.
\item[(c)] Case $k=\itilde$.
The tableaux $\theta_{\itilde}$ and $\tilde{\theta}_{\itilde}$ are
obtained one from the other by switching $i$ and $\itilde$. We
have $f_{\itilde}(p)=\tilde{f}_{\itilde}(p)$ $\forall p\in P$.
\end{itemize}
\end{lemma}

Proposition
\ref{proposition-constructibility-2column}\,(c) then follows from Lemma
\ref{lemme_echange_iitilde}: we have to show that the algorithm relative
to the pair $(\tau,T)$ succeeds if and only if the algorithm relative
to the pair $(\tilde\tau,T)$ succeeds.
By Lemma
\ref{lemme_echange_iitilde}\,(a), we may assume that both
algorithms have not failed in the first $\itilde$ steps. It is
easy to see that the success of the algorithm after the
$\itilde$-th step only depends on the shape of $\theta_{\itilde}$,
on the values of the map $f_{\itilde}$ and on the subtableau
$\tau[\itilde+1,...,n]$ of entries $\itilde+1,...,n$. By Lemma
\ref{lemme_echange_iitilde}\,(c), the tableaux $\theta_{\itilde}$
and $\tilde{\theta}_{\itilde}$ have the same shape, the maps
$f_{\itilde}$ and $\tilde{f}_{\itilde}$ are equal, whereas the
subtableaux $\tau[\itilde+1,...,n]$ and
$\tilde{\tau}[\itilde+1,...,n]$ coincide.
Therefore, both algorithms fail or succeed simultaneously.

\subsection{Constructibility in the hook case}

\label{constructibility-hook-case}

Let $(\tau,T)\in\mathbf{Y}_{\mbox{\tiny hk}}$. Set $n=|\tau|=|T|$.
Let $r$ be the common length of the first column of $\tau,T$ and
let $s$ be the length of their first row. Let $Y$ be the
rectangular Young diagram with $s+1$ columns of length $r$ (the
$(s+1)$-th column may be implicit)
\[
Y=\begin{array}{|c|c|c|c|} \hline
\ & \cdots & \ & \ \\
\hline
\vdots & \ddots & \vdots & \vdots \\
\hline
\ & \cdots & \ & \ \\
\hline
\end{array}
\]
The algorithm which we define aims to reconstruct $\tau$ as the
final term of a sequence of tableaux $\theta_1,\theta_2,...$
obtained by inserting successively the entries $1,2,...$ in the
diagram $Y$. 

We consider tableaux which are partial numberings of $Y$. The
tabl\-eaux $\tau$ and $T$ and their subtableaux $\tau[i,...,j]$
and $T[i,...,j]$ are considered as partial numberings of $Y$. Let
$\theta_0$ be the empty numbering of $Y$. The tableau $\theta_i$
obtained at the $i$-th step of the algorithm is a partial
numbering of entries $1,...,i$,
and it has the following properties:
\begin{itemize}
\item[(h-a)] For $p\in\{1,...,r\}$, the entries of the $p$-th row
of $\theta_i$ belong to the $p$-th row of $\tau$. The entries of
the first row of $\theta_i$ are in increasing order.
\item[(h-b)] For
$q\in\{1,...,s\}$, the number of entries in the $q$-th column of
$\theta_i$ equals the number of entries $\leq i$ in the $q$-th column
of $T$.
\item[(h-c)] 
$\theta_i$ and $\tau[1,\ldots,i]$ differ by at most one entry, whose row number is $\geq 2$.
\end{itemize}
For $i\in\{1,...,n\}$ suppose that we have constructed
$\theta_{i-1}$ satisfying these properties. 
Let $s_{i-1}\in\{1,...,s\}$ be the length
of the first row of the subtableau $\tau[1,\ldots,i-1]$.
Properties (h-a), (h-b) and (h-c)
imply that $\theta_{i-1}$
has the following aspect:
\[
\begin{array}{c}\mbox{a)} \\ \ \\ \ \\ \ \\ \end{array}
\!\!\!\!\!\theta_{i-1}=\begin{tyoung}{3} \tligne{3}{ * & * & \ }
\tligne{3}{ * & \  & \ } \tligne{3}{\ & \  & \ } \tligne{3}{
* & \ & \ }
 \end{tyoung}
\qquad
\begin{array}{c}\mbox{b)} \\ \ \\ \ \\ \ \\ \end{array}
\!\!\!\!\!\theta_{i-1}=\begin{tyoung}{3} \tligne{3}{ * &
* & \ } \tligne{3}{ *   & \ & \ } \tligne{3}{ \  & \
& \ } \tligne{3}{  \ & \ & * }
 \end{tyoung}
\]
\begin{itemize}
\item[a)] either the tableau $\theta_{i-1}$ and $\tau[1,...,i-1]$ coincide,
then the $(s_{i-1}+1)$-th column of $\theta_{i-1}$ is empty,  
\item[b)] or they differ by exactly one entry $j$ which has the same row number
$\geq 2$ in both tableaux, but $j$ is in the first column of $\tau[1,\ldots,i-1]$
and the unique entry in the $(s_{i-1}+1)$-th column of $\theta_{i-1}$.
\end{itemize}
We construct $\theta_i$ from $\theta_{i-1}$ in the following
manner.

\smallskip \noindent (0)
Let $p\in\{1,...,r\}$ be the row number of $i$ in $\tau$. Note
that the $p$-th box of the $(s_{i-1}+1)$-th column of
$\theta_{i-1}$ is empty. 
Put $i$ in this box and denote by
$\theta'_i$ the tableau so-obtained.

For instance:
\[
\begin{array}{c}\mbox{a.1)} \\ \ \\ \ \\ \ \\ \end{array}
\!\!\!\!\!\!\!\!\!\!\theta'_{i}=\begin{tyoung}{3} \tligne{3}{ * &
* & i } \tligne{3}{ * & \  & \ } \tligne{3}{\ & \  & \ } \tligne{3}{ * & \ & \  }
 \end{tyoung}
\quad
\begin{array}{c}\mbox{a.2)} \\ \ \\ \ \\ \ \\ \end{array}
\!\!\!\!\!\!\!\!\!\!\theta'_{i}=\begin{tyoung}{3} \tligne{3}{ * &
* & \ } \tligne{3}{ * & \  & \ } \tligne{3}{\ & \  & i } \tligne{3}{ * & \ & \ }
 \end{tyoung} \quad
\begin{array}{c}\mbox{b.1)} \\ \ \\ \ \\ \ \\ \end{array}
\!\!\!\!\!\!\!\!\!\!\theta'_{i}=\begin{tyoung}{3} \tligne{3}{ * &
* & i } \tligne{3}{ * & \ & \ } \tligne{3}{ \  & \
& \ } \tligne{3}{  \ & \ & \!*\! }
 \end{tyoung}
\quad
\begin{array}{c}\mbox{b.2)} \\ \ \\ \ \\ \ \\ \end{array}
\!\!\!\!\!\!\!\!\!\!\theta'_{i}=\begin{tyoung}{3} \tligne{3}{ * &
* & \ } \tligne{3}{ * & \ & \ } \tligne{3}{ \  & \
& i } \tligne{3}{  \ & \ & * }
 \end{tyoung}
\]

\smallskip \noindent (1) First case: $i$ lies in the first row of $T$. 
We state this failure case:

\smallskip
\noindent {\em (First failure case) If $i$ is in the first row of $T$
whereas the tableaux $\theta_{i-1}$ and
$\tau[1,...,i-1]$ do not coincide, then the algorithm fails.}

\smallskip
\noindent We suppose that the failure case does not occur. 
Then we set $\theta_i=\theta'_i$. 

For example:
\[
\begin{array}{c}\mbox{a.1)} \\ \ \\ \ \\ \ \\ \end{array}
\!\!\!\!\!\!\!\theta_{i}=\begin{tyoung}{3} \tligne{3}{ * &
* & i } \tligne{3}{ * & \  & \ } \tligne{3}{\ & \  & \ } \tligne{3}{ * & \ & \ }
 \end{tyoung}
\qquad
\begin{array}{c}\mbox{a.2)} \\ \ \\ \ \\ \ \\ \end{array}
\!\!\!\!\!\!\!\theta_{i}=\begin{tyoung}{3} \tligne{3}{ * &
* & \ } \tligne{3}{ * & \  & \ } \tligne{3}{\ & \  & i } \tligne{3}{ * & \ & \ }
 \end{tyoung}
\qquad \mbox{b.1), b.2): fail}\]

\smallskip \noindent (2) Second case: $i>1$, and $i$ belongs to the first column of $T$. We
state:

\smallskip
\noindent {\em (Second failure case) If $i>1$ belongs to the first
column of $T$ whereas $s_{i-1}>0$ and the tableaux $\theta'_i$ and $\tau[1,...,i]$
coincide, then the algorithm fails.}

\smallskip
\noindent We suppose that this failure case does not occur. 
Then there are 
nonempty boxes among the last
$r-1$ boxes of the $(s_{i-1}+1)$-th column of $\theta'_i$. 
Choose the entry in the upper row and push it to the first
column. Let $\theta_i$ be the tableau so-obtained.

For example:
\[
\mbox{a.1): fail} \
\begin{array}{c}\mbox{a.2)} \\ \ \\ \ \\ \ \\ \end{array}
\!\!\!\!\!\!\!\!\theta_{i}=\begin{tyoung}{3} \tligne{3}{ * &
* & \ } \tligne{3}{ * & \  & \ } \tligne{3}{i & \  & \ } \tligne{3}{ * & \ & \ }
 \end{tyoung}
\
\begin{array}{c}\mbox{b.1)} \\ \ \\ \ \\ \ \\ \end{array}
\!\!\!\!\!\!\!\!\theta_{i}=\begin{tyoung}{3} \tligne{3}{ * &
* & i } \tligne{3}{ * & \ & \ } \tligne{3}{ \  & \
& \ } \tligne{3}{* & \ & \ }
 \end{tyoung}
\
\begin{array}{c}\mbox{b.2)} \\ \ \\ \ \\ \ \\ \end{array}
\!\!\!\!\!\!\!\!\theta_{i}=\begin{tyoung}{3} \tligne{3}{ * &
* & \  } \tligne{3}{ * & \ & \  } \tligne{3}{ i  & \
& \  } \tligne{3}{  \ & \ & *  }
 \end{tyoung}
\]
In each case, properties (h-a), (h-b), (h-c) are easily satisfied by $\theta_i$.

\smallskip

If no failure case occurs while $i$ runs over $\{1,...,n\}$, then
we get a final tableau $\theta_n$ with entries $1,...,n$.
According to (h-a) and (h-b) we have $\theta_n=\tau$. We say that
the pair $(\tau,T)$ is {\em constructible.}

\begin{example}
(a) Suppose
\[
\tau=\young(245,3,1) \quad\mbox{and}\quad T=\young(134,2,5)
\]
We get successively
\[
\begin{array}{ccccc}
\begin{tyoung}{3} \tligne{3}{\ & \ & \ } \tligne{3}{\ & \ & \ } \tligne{3}{1 & \ & \ } \end{tyoung}
& & \begin{tyoung}{3} \tligne{3}{2 & \ & \ } \tligne{3}{\ & \ & \
} \tligne{3}{1 & \ & \ } \end{tyoung}
& & \begin{tyoung}{3} \tligne{3}{2 & \ & \ } \tligne{3}{\ & 3 & \ } \tligne{3}{1 & \ & \ } \end{tyoung} \\
\theta_1 & & \theta_2 & & \theta_3
\end{array}
\]
There is a failure of first type at the fourth step, since $4$
belongs to the first row of $T$ whereas $\theta_3$ and
$\tau[1,...,3]$ do not coincide.

\smallskip
\noindent (b) Suppose
\[
\tau=\young(134,5,2) \quad\mbox{and}\quad T=\young(125,3,4)
\]
\[
\begin{array}{cccc}
\begin{tyoung}{3} \tligne{3}{1 & \ & \ } \tligne{3}{\ & \ & \ } \tligne{3}{\ & \ & \ } \end{tyoung}\
& \begin{tyoung}{3} \tligne{3}{1 & \ & \ } \tligne{3}{\ & \ & \ }
\tligne{3}{\ & 2 & \ } \end{tyoung} \ & \begin{tyoung}{3}
\tligne{3}{1 & 3 & \ } \tligne{3}{\ & \ & \ } \tligne{3}{2 & \ & \
} \end{tyoung} \
& \begin{tyoung}{3} \tligne{3}{1 & 3 & 4 } \tligne{3}{\ & \ & \ } \tligne{3}{2 & \ & \ } \end{tyoung} \\
\theta_1 & \theta_2 & \theta_3 & \theta'_4
\end{array}
\]
There is a failure of second type at the fourth step since $4$
belongs to the first column of $T$ whereas the tableaux $\theta'_4$
and $\tau[1,...,4]$ coincide.

\smallskip
\noindent (c) Suppose now
\[
\tau=\young(235,4,1) \quad\mbox{and}\quad T=\young(134,2,5)
\]
\[
\begin{array}{ccccc}
\begin{tyoung}{3} \tligne{3}{\ & \ & \ } \tligne{3}{\ & \ & \ } \tligne{3}{1 & \ & \ } \end{tyoung}
& \begin{tyoung}{3} \tligne{3}{2 & \ & \ } \tligne{3}{\ & \ & \ }
\tligne{3}{1 & \ & \ } \end{tyoung} & \begin{tyoung}{3}
\tligne{3}{2 & 3 & \ } \tligne{3}{\ & \ & \ } \tligne{3}{1 & \ & \
} \end{tyoung} & \begin{tyoung}{3} \tligne{3}{2 & 3 & \ }
\tligne{3}{\ & \ & 4 } \tligne{3}{1 & \ & \ } \end{tyoung}
& \begin{tyoung}{3} \tligne{3}{2 & 3 & 5 } \tligne{3}{4 & \ & \ } \tligne{3}{1 & \ & \ } \end{tyoung} \\
\theta_1 & \theta_2 & \theta_3 & \theta_4 & \theta_5
\end{array}
\]
The pair $(\tau,T)$ is constructible.
\end{example}

We show the following result.

\begin{theorem}
\label{theorem-constructibility-hook} Let
$(\tau,T)\in\mathbf{Y}_{\mbox{\rm\tiny hk}}$. Then, we have
$\tau\in T$ if and only if the pair $(\tau,T)$ is constructible.
\end{theorem}

\begin{proof}
Let $\mathbf{A}_{\mbox{\tiny hk}}$ be the set introduced  in
section \ref{A-hook-case}. According to Theorem
\ref{theorem-A-hook-case}, it is sufficient to show that
$(\tau,T)$ is constructible if and only if
$(\tau,T)\in\mathbf{A}_{\mbox{\tiny hk}}$. Let $a_1=1<a_2<...<a_s$
(resp. $a'_1<a'_2<...<a'_s$) denote the entries of the first row
of $T$ (resp. of $\tau$). 
It is easy to see that,
if a failure of first type occurs in $i$, then $i\in\{a_2,...,a_s\}$,
whereas if the failure is of second type, then $i\in\{a'_2,...,a'_s\}$.

A first type failure occurs at
the $i$-th step for $i=a_q$ ($q\geq 2$) if and only if $\theta_{i-1}$ and
$\tau[1,...,i-1]$ do not coincide. By (h-b), this is equivalent to
the relation $a'_{q-1}>i-1$. Equivalently, $a'_{q-1}\geq a_q$.

Suppose $i=a'_q$ with $q\geq 2$. As 
the tableau $\theta_{i-1}$ satisfies condition (h-b),
we have $a_{q-1}\leq i-1$. There is a second type failure at
the step $i$ if and only if $i$ is in the first column of $T$
whereas $\theta_{i-1}$ and $\tau[1,...,i-1]$ coincide.
Equivalently, $i\notin\{a_2,...,a_q\}$ and $i-1<a_q$. This is
equivalent to have $a'_q=i<a_q$.

We get: the algorithm fails if and only if there is $q\in\{2,...,s\}$
such that $a'_{q-1}\geq a_q$ or $a_q>a'_q$. This is equivalent to: $(\tau,T)\notin \mathbf{A}_{\mbox{\tiny hk}}$. The
proof is then complete.  
\end{proof}

\section{Connection to the problem of intersections of components
in codimension one in the two-row case}

\label{7} \label{inter-two-rows}

In the last two sections of this article, we study the pairs of
irreducible components ${\mathcal K}^T,{\mathcal
K}^S\subset{\mathcal B}_u$ which intersect in codimension one,
i.e. $\mathrm{codim}_{{\mathcal K}^T}{\mathcal K}^T\cap{\mathcal
K}^S=1$ (recall that ${\mathcal B}_u$ is equidimensional, cf.
\ref{young-diagram}). The motivation for investigating this
question is a conjecture by D. Kazhdan and G. Lusztig (cf.
\cite[\S 6.3]{Kazhdan-Lusztig}) where these intersections of
codimension one play a crucial role. The description of pairs of
components ${\mathcal K}^T,{\mathcal K}^S\subset {\mathcal B}_u$
having an intersection in codimension one as well as the proof
that they accord to Kazhdan-Lusztig conjecture have been provided
in the hook, two-row and two-column cases and, up to now, only in
these three cases.

Our purpose is to relate this question to the questions involved
in the previous sections. Namely, still assuming that $T,S$ are of
hook, two-row or two-column type, we will show that
$\mathrm{codim}_{{\mathcal K}^T}{\mathcal K}^T\cap{\mathcal
K}^S=1$ implies ($S\in T$ or $T\in S$): this is Theorem
\ref{lastsection-maintheorem} of the next section.

In the present section, to begin with, we only deal with the
two-row case. The pairs of components intersecting in codimension
one have already been characterized in the two-row case (see
\cite{Fung} or \cite{Wolper}). Here, we provide a new
characterization:

\begin{theorem}
\label{theorem-codim1-section-2row} Let $u\in\mathrm{End}(V)$ be a
nilpotent endomorphism of two-row type. Let ${\mathcal
K}^T,{\mathcal K}^{S}\subset {\mathcal B}_u$ be two components
associated to the standard tableaux $T,S$. Then, the following
conditions are equivalent:
\begin{itemize}
\item[(i)] $\mathrm{codim}_{{\mathcal K}^T}{\mathcal K}^T\cap {\mathcal K}^{S}=1$;
\item[(ii)] ($S\in T$ or $T\in S$) and all the entries of the first row of $T$ but one
lie in the first row of $S$.
\end{itemize}
\end{theorem}

The previous characterization in \cite{Fung} involves the
combinatorics of cup-diagrams which also arises in the theory of
representations of the Temperley-Lieb algebra (see
\cite{Westbury}). We recall the definition of cup-diagrams and
this characterization first.

Implicitly, \cite{Fung} involves a procedure introduced by D.
Vogan \cite{Vogan}. Vogan's transformation ${\mathcal
T}_{\alpha,\beta}$ depends on two adjacent simple roots
$\alpha,\beta$ of a given irreducible root system, the
transformation ${\mathcal T}_{\alpha,\beta}$ is defined on
elements of the Weyl group $w\in W$ such that $w^{-1}(\beta)<0$
and $w^{-1}(\alpha)\not<0$. In type $A$, this transformation on
the Weyl group (which is then the symmetric group) has a
translation in terms of standard tableaux by the means of the
Robinson-Schensted algorithm. See \cite[\S 5]{MelnikovIII} for
more details. Here we recall only the procedure induced on
tableaux.

Then, we show the above theorem, by connecting cup-diagrams and
Vogan's procedure to the combinatorics involved in the previous
sections.

\subsection{The combinatorics of cup-diagrams and meanders}

\label{section-meanders}

Let $T$ be a standard tableau with two rows. Let $n=|T|$. Let
$r\geq s$ be the lengths of its two rows. Let $1=a_1<...<a_r$
(resp. $b_1<...<b_s$) be the entries in the first (resp. second)
row of $T$. Let us define a sequence $a^*_1,...,a^*_s$ by
induction. Put $a^*_1=b_1-1$, and having defined
$a^*_1,...,a^*_{q-1}$ for $q\leq s$, put
$$a^*_q=\mathrm{max}\,\{x\in \{a_1,...,a_r\}\setminus\{a^*_1,...,a^*_{q-1}\}:x<b_q\}.$$

Note that this definition interprets in terms of parenthesis diagrams (cf. \cite[\S 2]{Westbury}).
We construct a word of $n$ letters in the alphabet $\{\bullet,(,)\}$.
First put closing parentheses ``\,$)$\,'' at the positions number $b_1,\ldots,b_s$.
Next, consider the letters ``\,$)$\,'' from left to right, for each one
put its corresponding opening parenthesis ``\,$($\,'' in the rightmost non-assigned letter on its left.
Complete the word with $n-2s$ letters ``\,$\bullet$\,'' in the remaining places. 
The position of the ``\,$($\,'' 
corresponding to the ``\,$)$\,''
at the position $b_q$
gives the number $a^*_q$.

Following \cite{Fung}, the {\em cup-diagram of $T$} is the graph
with $n$ points numbered by $1,2,...,n$, displayed along a
horizontal line, and with an arc connecting $a^*_q$ to $b_q$ for
each $q=1,...,s$. We say that the $a^*_q$'s are the {\em left end
points}, the $b_q$'s are the {\em right end points}, and the
remaining points are called {\em fixed points}.

\begin{example} Let $T=\young(12467,3589)$. The corresponding parenthesis diagram is
``\,$\bullet\,(\,)\,(\,)\,(\,(\,)\,)\,$\,''.
The corresponding cup-diagram is
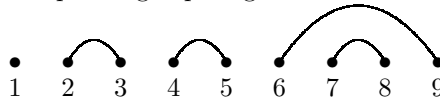
\[
\begin{picture}(180,30)(0,0)
\put(10,10){$\bullet$} \put(30,10){$\bullet$}
\put(50,10){$\bullet$} \put(70,10){$\bullet$}
\put(90,10){$\bullet$} \put(110,10){$\bullet$}
\put(130,10){$\bullet$} \put(150,10){$\bullet$}
\put(170,10){$\bullet$}

\put(10,0){1} \put(30,0){2} \put(50,0){3} \put(70,0){4}
\put(90,0){5} \put(110,0){6} \put(130,0){7} \put(150,0){8}
\put(170,0){9}

\qbezier(32,12)(42,30)(53,12)

\qbezier(72,12)(82,30)(93,12)

\qbezier(132,12)(142,30)(153,12) \qbezier(112,12)(142,56)(173,12)

\end{picture}
\]
%
%
%
%
%
\end{example}

Let $T,S$ be two standard tableaux of same shape with two rows.
Following  \cite{Ml-p}, the {\em meander $M_{T,S}$} is the graph
obtained as follows. We draw the cup-diagrams of $T$ and $S$ on
the same line of points $1,...,n$, the arcs of the cup-diagram of
$T$ being drawn upward and the arcs of the cup-diagram of $S$
being drawn downward.

\begin{example}
\label{example-meander} Let $T=\young(12467,3589)$ and
$S=\young(12567,3489)$. We construct the following meander
$M_{T,S}$:
\[
\begin{picture}(180,37)(0,30)
\put(10,50){$\bullet$} \put(30,50){$\bullet$}
\put(50,50){$\bullet$} \put(70,50){$\bullet$}
\put(90,50){$\bullet$} \put(110,50){$\bullet$}
\put(130,50){$\bullet$} \put(150,50){$\bullet$}
\put(170,50){$\bullet$}

\put(10,40){1} \put(30,40){2} \put(50,40){3} \put(70,40){4}
\put(90,40){5} \put(110,40){6} \put(130,40){7} \put(150,40){8}
\put(170,40){9}

\qbezier(32,52)(42,70)(53,52)

\qbezier(72,52)(82,70)(93,52)

\qbezier(132,52)(142,70)(153,52) \qbezier(112,52)(142,96)(173,52)

\qbezier(132,53)(142,35)(153,53) \qbezier(112,53)(142,9)(173,53)

\qbezier(32,53)(42,35)(53,53) \qbezier(12,53)(42,9)(73,53)

\end{picture}
\]
%
%
%
%
%
%
%
\end{example}

Following \cite[\S 5]{Ml-p} we introduce some terminology. A
connected subset of arcs can be open or closed. We call {\em loop}
a closed subset of arcs. We call {\em interval} an open subset of
arcs. The number of arcs in a path is called its {\em length}.
Note that the length of a loop is always even. The meander
$M_{T,S}$ is said to be {\em even} if all of its intervals are
even. It is said to be {\em odd}, otherwise. For example in the
meander above, there are three loops and one interval of length 2.
The meander is even.

\smallskip

The following theorem is a consequence of \cite[Theorems 7.3 and 7.4]{Fung} and the results
in \cite{Westbury}, as it is explained in \cite[\S 5.5]{Ml-p}.

\begin{theorem}
\label{th-meanders} Let $u\in\mathrm{End}(V)$ be a nilpotent
endomorphism of two-row type. Let ${\mathcal K}^T,{\mathcal
K}^{S}\subset {\mathcal B}_u$ be two components associated to the
standard tableaux $T,S$. Then, we have ${\mathcal K}^T\cap
{\mathcal K}^{S}\not=\emptyset$ if and only if $M_{T,S}$ is even.
Moreover, in this case, $\dim {\mathcal
K}^T\cap {\mathcal K}^{S}$ is the number of loops of $M_{T,S}$.
\end{theorem}

If $r\geq s$ are the lengths of the Jordan blocks of $u$,
then by \ref{young-diagram}, we have 
$\dim{\mathcal K}^T=\mathrm{dim}\,{\mathcal B}_u=s$. It
follows:

\begin{corollary}
Let $r\geq s$ be the lengths of the rows of
$\mathrm{sh}(T)=\mathrm{sh}(S)$. Then, $\mathrm{codim}_{{\mathcal
K}^T}{\mathcal K}^T\cap {\mathcal K}^{S}=1$ if and only if
$M_{T,S}$ is an even meander with $s-1$ loops.
\end{corollary}

\begin{example}
(a) For $T,S$ as in Example \ref{example-meander}, the meander
$M_{T,S}$ is even and has 3 loops.
It results $\mathrm{codim}_{{\mathcal K}^T}{\mathcal K}^T\cap {\mathcal K}^{S}=1$. \\
(b) In addition, let $R=\young(12347,5689)$. We see that the
meander $M_{T,R}$ is even with 3 loops, hence
$\mathrm{codim}_{{\mathcal K}^T}{\mathcal K}^T\cap {\mathcal
K}^{R}=1$. The meander $M_{S,R}$ is even, hence ${\mathcal
K}^S,{\mathcal K}^{R}$ have a nonempty intersection, but $M_{S,R}$
but has only 2 loops, hence this intersection has dimension two.
\end{example}

\smallskip
\noindent
{\em Definition of the set $\mathbf{M}_{\mbox{\rm\tiny
2-r}}$.} Let $T,S$ be standard tableaux of same shape with two rows
of lengths $r\geq s$. Say $(S,T)\in \mathbf{M}_{\mbox{\tiny 2-r}}$
if the meander $M_{T,S}$ is even and has $s-1$ loops.

\subsection{Vogan's ${\mathcal T}_{\alpha,\beta}$ procedure}

\label{T-alpha-beta}

We take the notation of \cite[\S 5]{MelnikovIII}. The set of the
simple roots for the type $A_{n-1}$ is denoted by
$\Pi=\{\alpha_i:i=1,...,n-1\}$ with $\alpha_i=(i,i+1)$. Let $T$ be
a standard tableau with $|T|=n$. For $i\in\{1,...,n\}$, we denote
by $r_T(i)$ the number of the row of $T$ containing $i$ (rows
being numbered from top to bottom). We say that $i\in\{1,...,n-1\}$
is a {\em descent of $T$} if $r_T(i)<r_T(i+1)$.

We define the set ${\mathcal D}_{\alpha,\beta}$ for any two
adjacent simple roots $\alpha,\beta$. Let $i=1,...,n-2$. Define
${\mathcal D}_{\alpha_i,\alpha_{i+1}}$ as the set of standard
tableaux $T$ with $|T|=n$ 
for which $i+1$ is a descent, but not $i$.
Define ${\mathcal D}_{\alpha_{i+1},\alpha_i}$ as the set of
standard tableaux $T$ with $|T|=n$ 
for which $i$ is a descent, but not $i+1$.

Let $T\in {\mathcal D}_{\alpha_i,\alpha_{i+1}}$. We define the standard tableau
${\mathcal T}_{\alpha_i,\alpha_{i+1}}(T)$ as follows. 
\begin{itemize}
\item[(a)]
If $r_T(i)<r_T(i+2)$, then ${\mathcal
T}_{\alpha_i,\alpha_{i+1}}(T)$ is obtained from $T$ by
interchanging $i+1$ and $i+2$.
\item[(b)]
If $r_T(i)\geq r_T(i+2)$, then ${\mathcal
T}_{\alpha_i,\alpha_{i+1}}(T)$ is obtained from $T$ by
interchanging $i$ and $i+1$.
\end{itemize}
Thus ${\mathcal T}_{\alpha_i,\alpha_{i+1}}$ is a bijection from
${\mathcal D}_{\alpha_{i},\alpha_{i+1}}$ to ${\mathcal
D}_{\alpha_{i+1},\alpha_{i}}$, and its inverse bijection is
denoted by ${\mathcal T}_{\alpha_{i+1},\alpha_i}$: we obtain
${\mathcal T}_{\alpha_{i+1},\alpha_i}(S)$ from $S\in {\mathcal
D}_{\alpha_{i+1},\alpha_{i}}$ by interchanging either $i$ and
$i+1$, or $i+1$ and $i+2$, depending on whether $r_T(i)<r_T(i+2)$,
or $r_T(i)\geq r_T(i+2)$.

In addition, for $i=2,...,n-1$, define ${\mathcal D}_i$ as the set
of standard tableaux $T$ with $|T|=n$ such that $i,i+1$ are
neither in the same row nor in the same column of $T$. For
$T\in{\mathcal D}_i$, let ${\mathcal T}_i(T)$ be the tableau
obtained from $T$ by switching $i,i+1$.

If $T$ has two rows, 
then, observe that we have $T\in{\mathcal
D}_{\alpha_i,\alpha_{i+1}}$ if and only if $i+1$ is a descent of
$T$. Likewise, we have $T\in{\mathcal
D}_{\alpha_{i+1},\alpha_{i}}$ if and only if $i$ is a descent of
$T$.

\smallskip
\noindent
{\em Definition of the sets $\mathbf{V}$ and $\mathbf{V}_{\mbox{\rm\tiny
2-r}}$.} We denote by $\mathbf{V}$ the set of pairs of standard
tableaux obtained as follows
\[
(\,{\mathcal T}_{\alpha^{(j)},\beta^{(j)}}\cdots{\mathcal
T}_{\alpha^{(1)},\beta^{(1)}}(T)\,,\ {\mathcal
T}_{\alpha^{(j)},\beta^{(j)}}\cdots{\mathcal
T}_{\alpha^{(1)},\beta^{(1)}}{\mathcal T}_{i}(T)\,)
\]
where $(\alpha^{(k)},\beta^{(k)})_{k=1,...,j}$ is a sequence (possibly empty) of
pairs of adjacent simple roots, and $T\in{\mathcal D}_{i}$ is such
that the tableaux ${\mathcal
T}_{\alpha^{(k)},\beta^{(k)}}\cdots{\mathcal
T}_{\alpha^{(1)},\beta^{(1)}}(T)$ and ${\mathcal
T}_{\alpha^{(k)},\beta^{(k)}}\cdots{\mathcal
T}_{\alpha^{(1)},\beta^{(1)}}{\mathcal T}_{i}(T)$ are well defined
for every $k=1,...,j$. We denote by $\mathbf{V}_{\mbox{\tiny
2-r}}$ the subset of pairs $(S,T)\in\mathbf{V}$ which have two
rows.

\begin{example}
Let $T,S,R$ be as in the previous example. Then we have
\begin{eqnarray}
 & (T,S) = (T,{\mathcal T}_{4}(T))\in\mathbf{V}_{\mbox{\tiny 2-r}}, \nonumber \\
 & (T,R) = ({\mathcal T}_{\alpha_4,\alpha_5}(Q),{\mathcal T}_{\alpha_4,\alpha_5}{\mathcal T}_{3}(Q))
\in\mathbf{V}_{\mbox{\tiny 2-r}} \nonumber
\end{eqnarray}
where $Q=\young(12457,3689)$.
\end{example}

\subsection{A new characterization of the intersections in codimension one}

\label{new-charac-int-codim-one}

Theorem \ref{theorem-codim1-section-2row} follows from the
following

\begin{theorem}
\label{theorem-Vogan-meanders} Let $T,S$ be two standard tableaux
of same shape with two rows. The following conditions are
equivalent:
\begin{itemize}
\item[(i)] $(T,S)\in\mathbf{M}_{\mbox{\rm\tiny 2-r}}$;
\item[(ii)] $(T,S)\in\mathbf{V}_{\mbox{\rm\tiny 2-r}}$;
\item[(iii)] ($S\in T$ or $T\in S$) and every entry in the first row of $T$ but one
lies in the first row of $S$.
\end{itemize}
\end{theorem}

The end of this section is devoted to the proof of Theorem
\ref{theorem-Vogan-meanders}. The proof is obtained by combining
Lemmas \ref{step-1}, \ref{step-2} and \ref{step-3}. The first step
is the following

\begin{lemma}
\label{step-1} $\mathbf{V}_{\mbox{\rm\tiny 2-r}}\subset
\mathbf{M}_{\mbox{\rm\tiny 2-r}}$.
\end{lemma}

\begin{proof}
If $T\in{\mathcal D}_i$ and $S={\mathcal T}_i(T)$, then the fact
that $\mathrm{codim}_{{\mathcal K}^T}{\mathcal K}^T\cap {\mathcal
K}^{S}=1$ follows from \cite[\S 3.1]{Melnikov-Pagnon1} (it can be
shown also by using Theorem \ref{th-meanders}).

Let $T,S\in{\mathcal D}_{\alpha_i,\alpha_{i+1}}$ and put
$T'={\mathcal T}_{\alpha_i,\alpha_{i+1}}(T)$ and $S'={\mathcal
T}_{\alpha_i,\alpha_{i+1}}(S)$. We show:
$$(T,S)\in \mathbf{M}_{\mbox{\rm\tiny 2-r}}
\Leftrightarrow (T',S')\in \mathbf{M}_{\mbox{\rm\tiny 2-r}}.$$ To
do this, let us describe the changes between the cup-diagrams of
$T$ and $T'$.

\smallskip
\noindent (a) Assume $r_T(i)<r_T(i+2)$. Then $i,i+1$ are in the
first row of $T$ and $i+2$ is in the second row, and $T'$ is
obtained by switching $i+1,i+2$. As $i+1$ is a descent of $T$,
there is an arc joining $(i+1,i+2)$ in the cup-diagram of $T$.
Possibly an arc starts at $i$ until some $i_1>i+2$:
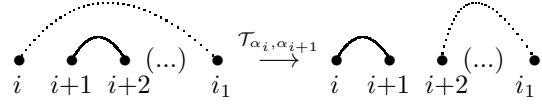
\[
\begin{picture}(210,30)(0,5)
\put(10,10){$\bullet$} \put(30,10){$\bullet$}
\put(50,10){$\bullet$} \put(85,10){$\bullet$}
\put(95,10){$\stackrel{{\mathcal
T}_{\alpha_i,\alpha_{i+1}}}{\longrightarrow}$}
\put(130,10){$\bullet$} \put(150,10){$\bullet$}
\put(170,10){$\bullet$} \put(205,10){$\bullet$}

\put(60,10){$(...)$} \put(180,10){$(...)$}

\put(10,0){$i$} \put(24,0){$i{+}1$} \put(46,0){$i{+}2$}
\put(85,0){$i_1$} \put(130,0){$i$} \put(144,0){$i{+}1$}
\put(166,0){$i{+}2$} \put(200,0){$i_1$}

\qbezier(32,12)(42,30)(53,12) \qbezier[40](12,12)(42,56)(88,12)

\qbezier(132,12)(142,30)(153,12)
\qbezier[30](172,12)(180,56)(208,12)

\end{picture}
\]
%
%
%
%
%
As represented in the picture, it is straightforward to check that
the cup-diagram of $T'$ is obtained by changing the arc joining
$(i+1,i+2)$ into an arc $(i,i+1)$, and changing the arc joining
$(i,i_1)$, if exists, into an arc $(i+2,i_1)$.

\smallskip
\noindent (b) Assume $r_T(i)\geq r_T(i+2)$. Then $i+1$ is in the
first row of $T$ and $i,i+2$ are in the second row, and $T'$ is
obtained by switching $i,i+1$. There is an arc joining $(i+1,i+2)$
in the cup-diagram of $T$. As $i$ is in the second row, there is
an arc $(i_0,i)$ for some $i_0<i$:
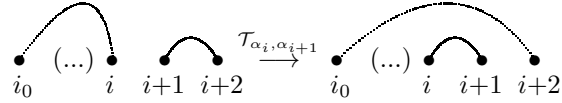
\[
\begin{picture}(210,30)(0,5)
\put(10,10){$\bullet$} \put(45,10){$\bullet$}
\put(65,10){$\bullet$} \put(85,10){$\bullet$}
\put(95,10){$\stackrel{{\mathcal
T}_{\alpha_i,\alpha_{i+1}}}{\longrightarrow}$}
\put(130,10){$\bullet$} \put(165,10){$\bullet$}
\put(185,10){$\bullet$} \put(205,10){$\bullet$}

\put(25,10){$(...)$} \put(145,10){$(...)$}

\put(10,0){$i_0$} \put(45,0){$i$} \put(59,0){$i{+}1$}
\put(81,0){$i{+}2$} \put(130,0){$i_0$} \put(165,0){$i$}
\put(179,0){$i{+}1$} \put(201,0){$i{+}2$}

\qbezier(67,12)(77,30)(88,12) \qbezier[80](12,12)(43,56)(48,12)

\qbezier(167,12)(177,30)(188,12)
\qbezier[80](132,12)(178,56)(208,12)

\end{picture}
\]
%
%
%
%
%
As in the picture, it is straightforward to see that the
cup-diagram of $T'$ is obtained by changing the arc joining
$(i+1,i+2)$ into an arc $(i,i+1)$, and changing the arc joining
$(i_0,i)$ into an arc $(i_0,i+2)$.

\smallskip Now let us describe the changes between the meanders
$M_{T,S}$ and $M_{T',S'}$. In any case, there is an arc  joining
$(i+1,i+2)$ in the cup-diagrams of both $T$ and $S$, and possible
arcs $(i,i_2)$ and $(i,i'_2)$ in both cup-diagrams (with either
$i_2<i$ or $i_2>i+2$, and similarly $i'_2<i$ or $i'_2>i+2$,
possibly $i_2=i'_2$). The following picture illustrates the
changes between the meanders $M_{T,S}$ and $M_{T',S'}$ (the
picture assumes for example $i_2,i'_2>i+2$).
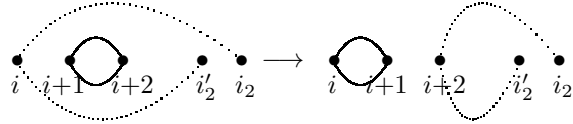
\[
\begin{picture}(220,40)(0,25)
\put(10,40){$\bullet$} \put(30,40){$\bullet$}
\put(50,40){$\bullet$} \put(80,40){$\bullet$}
\put(95,40){$\bullet$} \put(105,40){$\longrightarrow$}
\put(130,40){$\bullet$} \put(150,40){$\bullet$}
\put(170,40){$\bullet$} \put(200,40){$\bullet$}
\put(215,40){$\bullet$}

\put(10,30){$i$} \put(22,30){$i{+}1$} \put(48,30){$i{+}2$}
\put(80,30){$i'_2$} \put(95,30){$i_2$} \put(130,30){$i$}
\put(144,30){$i{+}1$} \put(166,30){$i{+}2$} \put(200,30){$i'_2$}
\put(215,30){$i_2$}

\qbezier(32,42)(42,60)(53,42) \qbezier[40](12,42)(42,86)(98,42)
\qbezier(32,42)(42,24)(53,42) \qbezier[40](12,42)(42,-2)(83,42)

\qbezier(132,42)(142,60)(153,42)
\qbezier[30](172,42)(182,86)(218,42)
\qbezier(132,42)(142,60)(153,42)
\qbezier[30](172,42)(182,-2)(203,42)
\qbezier(132,42)(142,24)(153,42)
\end{picture}
\]
%
%
%
It is clear that the change $M_{T,S}\rightarrow M_{T',S'}$
preserves the number of loops and the lengths of the intervals.
Therefore, we get $(T,S)\in \mathbf{M}_{\mbox{\rm\tiny 2-r}}$
$\Leftrightarrow$ $(T',S')\in \mathbf{M}_{\mbox{\rm\tiny 2-r}}$.
Then, the inclusion $\mathbf{V}_{\mbox{\rm\tiny 2-r}}\subset
\mathbf{M}_{\mbox{\rm\tiny 2-r}}$ follows from the definition of
$\mathbf{V}_{\mbox{\rm\tiny 2-r}}$.  
\end{proof}

Next, we show the relation between Vogan's procedure and
constructible pairs:

\begin{lemma}
\label{step-2} Let $T,S$ be two standard tableaux of same shape
with two rows. Assume $S\in T$, and that every entry in the first
row of $T$ but one lies in the first row of $S$. Then we have
$(T,S)\in\mathbf{V}_{\mbox{\rm\tiny 2-r}}$.
\end{lemma}

\begin{proof}
Following section \ref{constructibility-2row-case}, we consider
the algorithm of constructibility for the pair $(S,T)$. Let
$\theta_1,\theta_2,...$ be the tableaux obtained while applying
the algorithm. 
As $S\in T$, by Theorem \ref{theorem-constructibility-2row}, the pair $(S,T)$ is constructible.
Let $i$ be the minimal entry which has not the same
place in $T$ and $S$.
By definition of the algorithm, the tableau
$\theta_{i-1}$ coincides with the subtableau $S[1,...,i-1]$.
In addition, as the algorithm does not fail at the $i$-th step,
it follows that $i$ is in the first row of $T$ and in the second
row of $S$. 
Thus, $\theta_i$ has the following form:
\[
\theta_{i}=\begin{array}{|c|c|c|} \cline{1-2}
\multicolumn{2}{|c|}{*\ *\ *\ *} \\
\hline
\multicolumn{1}{|c|}{*\ *} & \quad & i \\
\cline{1-1} \cline{3-3}
\end{array}
\]
The shape of this tableau is not a Young diagram since there are
empty boxes on the left of $i$. As $(S,T)$ is
constructible, the final tableau $\theta_n$ coincides with $S$,
hence its shape is a Young diagram. Thus $i$ is pushed to the
left, during the remaining steps. By definition of the algorithm,
it implies that there is some $j>i$ in the second row of $T$. Take
$j$ minimal.

We show $(S,T)\in\mathbf{V}_{\mbox{\rm\tiny 2-r}}$ by induction on
the pair $(n-i,j-i)$ for the lexicographical order. If $j=i+1$,
then $T\in{\mathcal D}_{i}$. Moreover $i+1$ belongs to the
first row of $S$, since the algorithm fails at the $(i+1)$-th
step otherwise. Therefore, $S={\mathcal T}_i(T)$ and we have
$(T,S)=(T,{\mathcal T}_i(T))\in\mathbf{V}_{\mbox{\rm\tiny 2-r}}$.

Now, we suppose $j\geq i+2$. Then $i+1,...,j-1$ are in the first
row of $T$. By hypothesis $i$ is the only entry of the first row
of $T$ which is not in the first row of $S$, hence $i+1,...,j-1$
also belong to the first row of $S$. By definition of the
algorithm, the tableau $\theta_{j-1}$ is as follows:
\[
\theta_{i}=\begin{array}{|c|c|c|c|c|c|} \cline{1-2} \cline{4-6}
\multicolumn{2}{|c|}{*\ *\ *\ *} & \ & \!\!i\!\!+\!\!1\!\! & \cdots & \!\!j\!\!-\!\!1\!\! \\
\hline
\multicolumn{1}{|c|}{*\ *} & \quad & i \\
\cline{1-1} \cline{3-3}
\end{array}
\]

Since $j$ is in the second row of $T$, it is impossible that $j$
is in the first row of $S$, since there is a failure at the $j$-th
step otherwise. Hence $j$ is in the second row of both $T,S$. It
follows $T,S\in{\mathcal D}_{\alpha_{j-2},\alpha_{j-1}}$.

Write $T'={\mathcal T}_{\alpha_{j-2},\alpha_{j-1}}(T)$ and
$S'={\mathcal T}_{\alpha_{j-2},\alpha_{j-1}}(S)$. The tableau $T'$
is obtained from $T$ by switching $j-1,j$. The tableau $S'$ is
obtained from $S$ by switching either $i,i+1$ or $j-1,j$ depending
on whether $j=i+2$ or $j>i+2$. In the former case, $i+2$ is the
minimal entry which has not the same place in $T',S'$. In the latter
case, $i$ is the minimal entry which has not the same place in
$T',S'$ and $j-1$ is the minimal entry, bigger than $i$, which lies
in the second row of $T$. In both cases, the induction hypothesis
applies, so that it is sufficient to show that $S'\in T'$.

To prove $S'\in T'$, we show that the pair $(S',T')$ is
constructible. Let $\theta'_1,\theta'_2,...$ be the tableaux
obtained by applying the algorithm. Since $1,...,i-1$ are in the
same places in $T,S,T',S'$, we have $\theta'_{i-1}=\theta_{i-1}$.

\smallskip
\noindent (a) Suppose $j=i+2$. Then $i$ is in the first row of
$T',S'$, $i+1$ is in the second row of $T',S'$ and $i+2$ is in the
first row of $T'$, in the second row of $S'$. Then we have
\[
\theta'_{i+2}=\begin{array}{|c|c|c|c|c|} \cline{1-4}
\multicolumn{3}{|c|}{*\ *\ *\ *} & i \\
\hline
\multicolumn{1}{|c|}{*\ *} & \!\!i\!\!+\!\!1\!\! & \multicolumn{2}{|c|}{\quad} & \!i\!\!+\!\!2\!\! \\
\cline{1-2} \cline{5-5}
\end{array}
\mbox{ \ whereas \ } \theta_{i+2}=\begin{array}{|c|c|c|c|c|}
\cline{1-4}
\multicolumn{3}{|c|}{*\ *\ *\ *} & \!\!i\!\!+\!\!1\!\! \\
\hline
\multicolumn{1}{|c|}{*\ *} & i & \multicolumn{2}{|c|}{\quad} & \!i\!\!+\!\!2\!\! \\
\cline{1-2} \cline{5-5}
\end{array}
\]
As $\theta_{i+2},\theta'_{i+2}$ have the same shape, and as
$i+3,...,n$ have the same places in $T,T'$ on one hand, and in
$S,S'$ on the other hand, we get that $(S',T')$ is constructible
if and only if $(S,T)$ is constructible.

\smallskip
\noindent (b) Suppose $j>i+2$. Then $i$ is in the first row of
$T'$, in the second row of $S'$, $i+1,...,j-2,j$ are in the first
row of $T'$ and $S'$, $j-1$ is in the second row of $T'$ and $S'$.
Then we have
\begin{eqnarray}
\theta'_{j}=\begin{array}{|c|c|c|c|c|c|c|} \cline{1-4} \cline{6-7}
\multicolumn{3}{|c|}{*\ *\ *\ *} &  \!\!i\!\!+\!\!1\!\! & \  & \cdots & j \\
\hline
\multicolumn{1}{|c|}{*\ *} & i & \multicolumn{2}{|c|}{\quad} & \!\!j\!\!-\!\!1\!\! \\
\cline{1-2} \cline{5-5}
\end{array}
\mbox{ \ whereas \ } \theta_{j}=\begin{array}{|c|c|c|c|c|c|c|}
\cline{1-4} \cline{6-7}
\multicolumn{3}{|c|}{*\ *\ *\ *} &  \!\!i\!\!+\!\!1\!\! & \  & \cdots & \!\!j\!\!-\!\!1\!\! \\
\hline
\multicolumn{1}{|c|}{*\ *} & i & \multicolumn{2}{|c|}{\quad} & j \\
\cline{1-2} \cline{5-5}
\end{array} \nonumber
\end{eqnarray}
Thus $\theta_{j},\theta'_{j}$ have the same shape, moreover
$j+1,...,n$ have the same places in $T,T'$ on one hand, and in
$S,S'$ on the other hand. Therefore, $(S',T')$ is constructible if
and only if $(S,T)$ is constructible.

\smallskip
In both cases, we infer that $S'\in T'$. This argument completes the
proof.  
\end{proof}

Finally, we prove:

\begin{lemma}
\label{step-3} Let $(T,S)\in\mathbf{M}_{\mbox{\rm\tiny 2-r}}$.
Then, we have ($S\in T$ or $T\in S$), and every entry in the first
row of $T$ but one lies in the first row of $S$.
\end{lemma}

\begin{proof}
Let $r\geq s$ be the lengths of the rows of
$\mathrm{sh}(T)=\mathrm{sh}(S)$, the common shape of $T,S$. First
notice that the meander $M_{T,S}$ contains $s-1$ loops. For each
loop, the rightmost point of the loop is a right end point of the
cup-diagrams of both $T$ and $S$. It follows that the second rows
of $T$ and $S$ have $s-1$ common entries. Equivalently, every
entry in the first row of $T$ but one lies in the first row of
$S$.

Say that the minimal entry which is not at the same place in $T,S$
lies in the first row of $T$ and in the second row of $S$. By
\cite[Lemma 2.3]{F} and \ref{section-vanLeeuwen}, the size of the
second row of the diagram $Y_{j/i}^T$ is the number of arcs
$(a,b)\subset \{i+1,...,j\}$ in the cup-diagram of $T$. On the
other hand, the size of the second row of the diagram $Y_{j/i}(S)$
is the minimum between the number of entries among $i+1,...,j$ in
the first row of $S$ and the number of entries among $i+1,...,j$
in the second row of $S$. Thus it is the minimum between the
number of right end points $a\in \{i+1,...,j\}$ in the cup-diagram
of $S$ and the number of points $b\in \{i+1,...,j\}$ which are not
right end points in the cup-diagram of $S$.

We aim to show $S\preceq T$. By the previous observation, it is
sufficient to show that, for any $0\leq i<j\leq n$, the number of
arcs $(a,b)\subset \{i+1,...,j\}$ in the cup-diagram of $T$ is
less than or equal to the number of right end points $a\in
\{i+1,...,j\}$ and than the number of points $b\in \{i+1,...,j\}$
which are not right end points in the cup-diagram of $S$. To do
this, it is sufficient to show that, for any arc joining $(a,b)$
in the cup-diagram of $T$, $\{a,b\}$ contains exactly one right
end-point of the cup-diagram of $S$.

The meander $M_{T,S}$ contains $2s$ arcs and $s-1$ loops. It
follows that at least $s-2$ loops have length $2$. Moreover, if
$M_{T,S}$ contains an (open) interval, then it contains $s-1$
loops of length $2$. If the arc $(a,b)$ of the cup-diagram of $T$
belongs to a loop of length two of the meander $M_{T,S}$, then it
is also an arc in the cup-diagram of $S$. Then, $\{a,b\}$ contains
$b$ as unique right end-point of the cup-diagram of $S$. Remove
all the loops of length two of the meander $M_{T,S}$ and all the
common fixed points of the cup-diagrams of $T$ and $S$, then there
are two possible configurations for the remaining arcs:
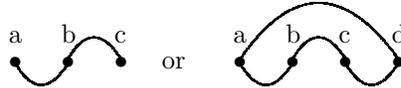
\[
\begin{picture}(50,25)(0,15)
\put(10,22){$\bullet$} \put(30,22){$\bullet$}
\put(50,22){$\bullet$}

\put(10,32){a} \put(30,32){b} \put(50,32){c}

\qbezier(12,25)(22,7)(33,25)

\qbezier(32,25)(42,43)(53,25)
\end{picture}
\begin{picture}(35,25)(0,15)
\put(18,22){\mbox{or}}
\end{picture}
\begin{picture}(70,25)(0,15)
\put(10,22){$\bullet$} \put(30,22){$\bullet$}
\put(50,22){$\bullet$} \put(70,22){$\bullet$}

\put(10,32){a} \put(30,32){b} \put(50,32){c} \put(70,32){d}

\qbezier(12,25)(22,7)(33,25)

\qbezier(52,25)(62,7)(73,25)

\qbezier(32,25)(42,43)(53,25)

\qbezier(12,25)(42,69)(73,25)
\end{picture}
\]
%
%
%
%
%
%
%
%
(recall that the arcs of the cup-diagram of $T$ are upward, the
arcs of the cup-diagram of $S$ are downward). In each one of these
two configurations, exactly one point among the end points of any
arc of the cup-diagram of $T$ is a right end-point of the
cup-diagram of $S$. This observation completes the proof.  
\end{proof}

\section{Connection to the problem of intersections of components
in codimension one}

\label{8}

In the previous section, we have shown in particular that for two
components ${\mathcal K}^T,{\mathcal K}^{S}\subset {\mathcal B}_u$
of two-row type, we have:
\[
\mathrm{codim}_{{\mathcal K}^T}{\mathcal K}^T\cap{\mathcal K}^S=1
\Rightarrow (S\in T \mbox{\ or\ } T\in S).
\]
In the present section, we show this implication also in the hook
and two-column cases. In addition, we deduce some topological
properties of the components.

\subsection{On the intersections of codimension one in the three cases}

\label{codim1-hook-two-column-irreducibility}

To begin with we recall from \cite{Ml-p} and \cite{Vargas} the
description of the pairs of components intersecting in codimension
one, for the hook and two-column cases, and we recall the already
known property of irreducibility of the intersections in
codimension one for the hook, two-row and two-column cases.

\smallskip

First, let $u\in\mathrm{End}(V)$ be nilpotent of hook type. Let
${\mathcal K}^T,{\mathcal K}^{S}\subset {\mathcal B}_u$ be the
components associated to the standard tableaux $T,S$. Let
$a_1=1<a_2<...<a_s$ (resp. $a'_1=1<a'_2<...<a'_s$) be the entries
in the first row of $T$ (resp. of $S$). By \cite{Vargas} (or
Proposition \ref{proposition-nonempty-intersection} and Theorem
\ref{theorem-A-hook-case} above) the intersection ${\mathcal
K}^T\cap{\mathcal K}^{S}$ is nonempty if and only if
\[\max(a_{q-1},a'_{q-1})<\min(a_{q},a'_{q})\ \ \forall q=2,...,s.\]
Assume that the intersection ${\mathcal K}^T\cap{\mathcal K}^{S}$
is nonempty. Then, from \cite{Vargas}, we have the following
formula:
\[\mathrm{codim}_{{\mathcal K}^T}{\mathcal K}^T\cap{\mathcal K}^{S}=
\sum_{q=2}^s|a_q-a'_q|\] and we can state:

\begin{theorem}
\label{theorem-codim1-hook-case} Let $u\in\mathrm{End}(V)$ be
nilpotent of hook type. The intersection ${\mathcal
K}^T\cap{\mathcal K}^{S}$ is nonempty of codimension one if and
only if $T,S$ are obtained one from the other by switching two
entries $i,i+1$ for some $i=2,...,n-1$.
\end{theorem}

Let $\mathbf{V}_{\mbox{\tiny hk}}$ denote the subset of pairs
$(T,S)\in\mathbf{V}$ which are of hook type (cf.
\ref{T-alpha-beta}). It is easy to see that
$\mathbf{V}_{\mbox{\tiny hk}}$ is the set of pairs $(T,{\mathcal
T}_i(T))$ for $T\in{\mathcal D}_i$ of hook type and $i=2,...,n-1$.
Then, we get:

\begin{corollary}
Let $u\in\mathrm{End}(V)$ be
nilpotent of hook type.
We have $\mathrm{codim}_{{\mathcal K}^T}{\mathcal K}^T\cap
{\mathcal K}^S=1$ if and only if $(T,S)\in
\mathbf{V}_{\mbox{\rm\tiny hk}}$.
\end{corollary}

Intersections in codimension one are closely connected in the
two-row and two-column cases. Let $u\in\mathrm{End}(V)$ be
nilpotent of two-column type. Let ${\mathcal K}^T,{\mathcal
K}^{S}\subset {\mathcal B}_u$ be the components associated to the
standard tableaux $T,S$. Let $T^t,S^t$ be the corresponding
tableaux with two rows, obtained by transposition (see section
\ref{transposee}). By \cite[\S 5.4]{Ml-p}, we have:

\begin{theorem}
\label{theorem-codim1-2column-case} Let $u\in\mathrm{End}(V)$ be
nilpotent of two-column type. Then, we have
$\mathrm{codim}_{{\mathcal K}^T}{\mathcal K}^T\cap {\mathcal
K}^{S}=1$ if and only if $\mathrm{codim}_{{\mathcal
K}^{S^t}}{\mathcal K}^{T^t}\cap {\mathcal K}^{S^t}=1$.
\end{theorem}

As in section \ref{T-alpha-beta}, let $\alpha,\beta\in\{\alpha_1,\ldots,\alpha_{n-1}\}$ be two consecutive simple roots.
Then it is easy to see that, if $T$ is a standard tableau of two-column type, then
$T\in{\mathcal D}_{\alpha,\beta}$ if and only if $T^t\in{\mathcal D}_{\beta,\alpha}$,
and moreover in this case ${\mathcal T}_{\beta,\alpha}(T^t)=S^t$, where $S={\mathcal T}_{\alpha,\beta}(T)$.
Likewise, for $i\in\{1,\ldots,n-1\}$ we have $T\in{\mathcal D}_i$ if and only
if $T^t\in {\mathcal D}_i$, and in this case ${\mathcal T}_i(T^t)=({\mathcal T}_i(T))^t$.
Let $\mathbf{V}_{\mbox{\tiny 2-c}}$ denote the subset of pairs
$(T,S)\in\mathbf{V}$ which have two columns. 
It follows that $(T,S)\in \mathbf{V}_{\mbox{\tiny 2-c}}$ if and only if $(T^t,S^t)\in \mathbf{V}_{\mbox{\tiny 2-r}}$.
Combining this observation with Theorems \ref{theorem-Vogan-meanders} and \ref{theorem-codim1-2column-case}, we get:


\begin{corollary}
Let $u\in\mathrm{End}(V)$ be
nilpotent of two-column type.
Then, we have $\mathrm{codim}_{{\mathcal K}^T}{\mathcal K}^T\cap
{\mathcal K}^{S}=1$ if and only if
$(T,S)\in\mathbf{V}_{\mbox{\rm\tiny 2-c}}$.
\end{corollary}

Thus it holds that, for $T,S$ a pair of standard tableaux of same
shape of hook, two-row or two-column type, the components
${\mathcal K}^T$ and ${\mathcal K}^S$ intersect in codimension one
if and only if $(T,S)\in\mathbf{V}$.

\smallskip

Finally, let us recall that the intersections in codimension one
are irreducible in the three cases we consider here.

\begin{theorem}
\label{theorem-irreducible-intersection} Let $u\in\mathrm{End}(V)$
be nilpotent of hook, two-row or two-column type. Let ${\mathcal
K}^T,{\mathcal K}^{S}\subset {\mathcal B}_u$ be the components
associated to the standard tableaux $T,S$. If
$\mathrm{codim}_{{\mathcal K}^T}{\mathcal K}^T\cap {\mathcal
K}^S=1$, then ${\mathcal K}^T\cap {\mathcal K}^S$ is irreducible.
\end{theorem}

This result is proved in \cite{Vargas} for the hook case, in
\cite{Fung} for the two-row case and in \cite{Melnikov-Pagnon} for
the two-column case. Notice that the intersection ${\mathcal
K}^T\cap {\mathcal K}^S$ can be reducible when
$\mathrm{codim}_{{\mathcal K}^T}{\mathcal K}^T\cap {\mathcal
K}^S>1$ (see \cite{Melnikov-Pagnon}).

\subsection{A necessary condition for an intersection in codimension one}

The main result of this section is the following

\begin{theorem}
\label{lastsection-maintheorem} Let $u\in\mathrm{End}(V)$ be a
nilpotent endomorphism of hook, two-row or two-column type. Let
${\mathcal K}^T,{\mathcal K}^{S}\subset {\mathcal B}_u$ be the
components associated to the standard tableaux $T,S$. We have:
\[\mathrm{codim}_{{\mathcal K}^T}{\mathcal K}^T\cap {\mathcal K}^{S}=1
\Rightarrow \mbox{($S\in T$ or $T\in S$)}.\]
\end{theorem}

\begin{proof}
The result holds in the two-row case, by Theorem
\ref{theorem-codim1-section-2row}. Then, we deduce that the
statement holds in the two-column case, using Proposition
\ref{proposition-transposee} and Theorem
\ref{theorem-codim1-2column-case}. It remains to show the result
in the hook case. Suppose that $T,S$ are of hook type, and assume
that $\mathrm{codim}_{{\mathcal K}^T}{\mathcal K}^T\cap {\mathcal
K}^{S}=1$. By Theorem \ref{theorem-codim1-hook-case}, the tableaux
$T$ and $S$ are obtained one from the other by switching $i,i+1$
for some $i=2,...,n-1$. Since both tableaux are standard, $i,i+1$
do not lie both in the first column or in the first row of $T$
(resp. of $S$). Thus $i$ is either in the first row of $T$ or in
the first row of $S$. Say that it is in the first row of $T$.
Then, applying Theorem \ref{theorem-A-hook-case}, we easily obtain
that $S\in T$.  
\end{proof}

\subsection{Topological properties of the intersections in codimension one
in the hook, two-row and two-column cases}

\label{section-topological-properties} \label{third-corollary}

We now establish some corollaries of Theorem
\ref{lastsection-maintheorem}. 
We consider two components
${\mathcal K}^T,{\mathcal K}^{S}\subset {\mathcal B}_u$ 
associated to standard tableaux $T,S$. Recall that they are
obtained as the closures of the subsets ${\mathcal
B}_u^T,{\mathcal B}_u^S\subset {\mathcal B}_u$ (see
\ref{young-diagram}).

\begin{corollary}
\label{first-corollary} 
Let $u\in\mathrm{End}(V)$ be of hook, two-row or two-column type. If
${\mathcal K}^T\cap {\mathcal K}^{S}$ has codimension $1$ in ${\mathcal K}^T$, then
one of its open subsets ${\mathcal K}^T\cap {\mathcal B}_u^{S}$ and ${\mathcal B}_u^T\cap
{\mathcal K}^{S}$ is nonempty and dense.
\end{corollary}

\begin{proof}
Using Lemma \ref{lemma-standardization} and applying Theorem
\ref{lastsection-maintheorem}, we get that ${\mathcal
B}_u^T\cap {\mathcal K}^{S}$ or ${\mathcal K}^T\cap {\mathcal
B}_u^{S}$ is nonempty. 
Then it is dense because, 
by Theorem
\ref{theorem-irreducible-intersection}, the intersection
${\mathcal K}^T\cap {\mathcal K}^{S}$ is irreducible. 
\end{proof}

Let ${\mathcal K}^R$ be a third component associated to another
standard tableau $R$.

\begin{proposition}
\label{second-corollary} 
Let $u\in\mathrm{End}(V)$ be of hook, two-row or two-column type.
Suppose that $\mathrm{codim}_{{\mathcal
K}^T}{\mathcal K}^T\cap{\mathcal K}^S\cap{\mathcal K}^R=1$, then
two tableaux among $T,S,R$ coincide.
\end{proposition}

\begin{proof}
We define a total order on standard tableaux: let $i\in\{1,...,n\}$ be
minimal which has not the same place in $T,S$ and let $Y_i^T$
(resp. $Y_i^S$) be the shape of the subtableau of $T$ (resp. of
$S$) of entries $1,...,i$. We have either $Y_i^T\prec Y_i^S$ or
$Y_i^S\prec Y_i^T$. Then, write $S< T$ if $Y_i^S\prec Y_i^T$. By
Theorem \ref{theorem_necessary_condition}, $S< T$ implies $T\notin
S$.

Assume by contradiction that $T,S,R$ are pairwise distinct. We may
assume $R< S< T$. Then, by Theorem \ref{lastsection-maintheorem},
we have $R\in S$, $R\in T$ and $S\in T$. By Corollary
\ref{first-corollary}, $\mathrm{codim}_{{\mathcal K}^T}{\mathcal
K}^T\cap {\mathcal B}_u^{S}=1$. As ${\mathcal K}^T\cap {\mathcal
K}^S\cap {\mathcal B}_u^{R}$ is a nonempty open subset of
${\mathcal K}^T\cap {\mathcal K}^S\cap {\mathcal K}^{R}$, we have
$\mathrm{codim}_{{\mathcal K}^T}{\mathcal K}^T\cap {\mathcal
K}^S\cap {\mathcal B}_u^{R}=1$. The sets ${\mathcal K}^T\cap
{\mathcal B}_u^{S}$ and ${\mathcal K}^T\cap {\mathcal K}^S\cap
{\mathcal B}_u^{R}$ are disjoint, locally closed subsets
of ${\mathcal K}^T\cap {\mathcal K}^S$ and they have the same
dimension as ${\mathcal K}^T\cap {\mathcal K}^S$. It contradicts
the irreducibility of ${\mathcal K}^T\cap {\mathcal K}^S$.  
\end{proof}

\subsection{A subset $\hat{\mathcal K}^T$ associated to generalized Spaltenstein's partitions}

\label{last-section}

The subset ${\mathcal B}_u^T\subset {\mathcal B}_u$ is the set of
$u$-stable flags $(V_0,...,V_n)\in{\mathcal B}_u$ such that the
Jordan shapes of the restrictions $u_{|V_i}$ of $u$ coincide with
the shapes of the subtableaux $T[1,...,i]$. The set ${\mathcal
B}_u^T$ is thus defined by simple relations. Nevertheless, the
difference between ${\mathcal B}_u^T$ and its closure, the
component ${\mathcal K}^T\subset{\mathcal B}_u$, is big in
general: one can see that, whenever ${\mathcal B}_u$ is reducible,
it admits a component ${\mathcal K}^T$ with
$\mathrm{codim}_{{\mathcal K}^T}{\mathcal K}^T\setminus {\mathcal
B}_u^T=1$. We aim to construct a subset $\hat{\mathcal K}^T\subset
{\mathcal K}^T$ bigger than ${\mathcal B}_u^T$, which is also
defined by simple relations. To do this we rely on a
generalization of the sets ${\mathcal B}_u^T$. Then we show that,
in the hook, two-row and two-column cases, we have
$\mathrm{codim}_{{\mathcal K}^T}{\mathcal K}^T\setminus
\hat{\mathcal K}^T\geq 2$.

Recall a construction from \cite{F-Betti}. Let ${\mathcal R}_n$
denote the set of double sequences of integers
$(i_k,j_k)_{k=0,...,n}$ with $(i_k)_k$ weakly decreasing,
$(j_k)_k$ weakly increasing, $0\leq i_k\leq j_k\leq n$ and
$j_k-i_k=k$ for every $k$. Let $\rho=(i_k,j_k)_{k}\in{\mathcal
R}_n$. Instead of considering the restrictions of $u$ to the
subspaces of the flag, we consider the maximal chain of
subquotients
\[
0=V_{j_0}/V_{i_0}\hookrightarrow V_{j_1}/V_{i_1}\hookrightarrow \ldots\hookrightarrow
V_{j_n}/V_{i_n}=V.
\]
For $k=1,\ldots,n$ we consider the Young diagram
$Y(u_{|V_{j_k}/V_{i_k}})$ associated to the nilpotent
$u_{|V_{j_k}/V_{i_k}}\in\mathrm{End}(V_{j_k}/V_{i_k})$ induced by
$u$. In addition, recall that $Y_{j_k/i_k}^T$ is defined in
\ref{Y-j/i-T}. We denote by ${\mathcal B}_{u,\rho}^T$ the set of
$u$-stable flags $(V_0,...,V_n)$ such that
$Y(u_{|V_{j_k}/V_{i_k}})=Y_{j_k/i_k}^T$ for every $k$. We get a
partition ${\mathcal B}_u=\bigsqcup_T{\mathcal B}_{u,\rho}^T$
parameterized by standard tableaux of shape $Y(u)$. By
\ref{Y-j/i-T}, the set ${\mathcal B}_{u,\rho}^T$ is a dense open
subset of the component ${\mathcal K}^T$.


\begin{proposition}
\label{proposition_avec_lemme} 
Let $u\in\mathrm{End}(V)$ be of hook, two-row or two-column type.
If ${\mathcal K}^T\cap {\mathcal K}^{S}$ has codimension $1$ in ${\mathcal K}^T$,
then there is $\rho\in {\mathcal R}_n$ such that
one of the open subsets
${\mathcal B}_{u,\rho}^T\cap {\mathcal B}_u^{S}$ and ${\mathcal B}_u^{T}\cap {\mathcal B}_{u,\rho}^S$
of ${\mathcal K}^T\cap {\mathcal K}^{S}$ is nonempty and dense.
\end{proposition}

We need the following technical

\begin{lemma}
\label{lemme_avec_proposition} Let $T,S$ be two standard tableaux
of common shape of hook, two-row or two-column type. Suppose
$T\not=S$. Then there are $i<j$ such that $Y_{j/i}^T\prec
Y_{j/i}^S$.
\end{lemma}

\begin{proof}[Proof of Lemma \ref{lemme_avec_proposition}]
Let $n=|T|=|S|$. Consider first the hook case. Let $L_T$ be the
set of elements $i\in\{1,...,n-1\}$ such that $i$ is situated on
the left of $i+1$ in the tableau $T$. Observe that $i\in L_T$ if
and only if $i+1$ is in the first row of $T$. Hence 
$\#L_T+1$ is the length of the first row of $T$. Suppose
that we have $Y_{j/i}^S\preceq Y_{j/i}^T$ for any $i<j$. Then in
particular $Y_{i+1/i-1}^S\preceq Y_{i+1/i-1}^T$ for $i\in L_S$
implies $L_S\subset L_T$. Thus $L_S=L_T$. It follows that the
first rows of $T$ and $S$ coincide, hence $T=S$.

Suppose that $T,S$ have two rows of lengths $r\geq s$. Argue by
induction on $n$, with immediate initialization for $n=1$. Assume
that the property holds until in $n-1$. Suppose $Y_{j/i}^S\preceq
Y_{j/i}^T$ for any $i<j$. We distinguish three cases.

\noindent (a) Suppose that there is $i<n$ such that the subtableau
$T[1,...,i]$ has a rectangular shape with two rows of length
$i/2$. As $Y_{i/0}^S\preceq Y_{i/0}^T$, the subtableaux
$T[1,...,i]$ and $S[1,...,i]$ have the same shape, and the skew
subtableaux $T[i+1,...,n]$ and $S[i+1,...,n]$ have the same shape,
which is a Young diagram. By induction hypothesis, we get
$T[1,...,i]=S[1,...,i]$ and $T[i+1,...,n]=S[i+1,...,n]$, hence
$T=S$.

\noindent (b) Suppose that $r=s=n/2$. Then $n$ has the same place
in $T,S$. Let $T'=T[1,...,n-1]$ and $S'=S[1,...,n-1]$. By
induction hypothesis, we get $T'=S'$. It follows $T=S$.

\noindent (c) Suppose that there is no $i\leq n$ for which
$T[1,...,i]$ is rectangular. Let $T'$ (resp. $S'$) be the
rectification by jeu de taquin of the skew subtableau $T[2,...,n]$
(resp. $S[2,...,n]$). Then, as observed in \ref{section-A-2row}\,(c),
$T'$ has two rows of lengths $(r-1,s)$.
Since $Y_{n/1}^S\preceq Y_{n/1}^T$, the tableaux $T',S'$ have the
same shape. By induction hypothesis, we get $T'=S'$. It follows
$T=S$. The argument in the two-row case is now complete.

Finally, suppose that $T,S$ have a common shape of two-column type
and assume $T\not=S$. Let $T^t,S^t$ be the transposed tableaux
with two rows (see \ref{transposee}). Then there are $i<j$ such
that $Y_{j/i}^{S^t}\prec Y_{j/i}^{T^t}$. It follows
$Y_{j/i}^{T}\prec Y_{j/i}^{S}$.  
\end{proof}

\begin{proof}[Proof of Proposition \ref{proposition_avec_lemme}]
By Corollary \ref{first-corollary}, we may
assume ${\mathcal K}^T\cap{\mathcal B}_u^{S}\not=\emptyset$. By Lemma \ref{lemme_avec_proposition},
there are $i<j$ such that $Y_{j/i}^T\prec Y_{j/i}^S$. We can find
$\rho=(i_k,j_k)_{k=1,...,n}\in {\mathcal R}_n$ such that $i_{k}=i$
and $j_{k}=j$ for $k=j-i$. Then, by Lemma
\ref{lemma-semi-continuity} and the definition of ${\mathcal
B}_{u,\rho}^S$, we have ${\mathcal K}^T\cap{\mathcal
B}_{u,\rho}^S=\emptyset$. On the other hand, as the ${\mathcal
B}_{u,\rho}^R$'s form a partition of ${\mathcal B}_u$ into locally
closed subsets, there is $R$ standard such that
${\mathcal K}^T\cap {\mathcal B}_u^{S}\cap{\mathcal B}_{u,\rho}^R$
is dense in ${\mathcal K}^T\cap{\mathcal B}_u^S$. It implies
$\mathrm{codim}_{{\mathcal K}^T}{\mathcal K}^T\cap {\mathcal
K}^{S}\cap{\mathcal K}^R=1$, hence by Proposition
\ref{second-corollary} we have $R\in\{T,S\}$. Moreover, 
as ${\mathcal K}^T\cap{\mathcal B}_{u,\rho}^S=\emptyset$, necessarily $R=T$.
Therefore, ${\mathcal B}_{u,\rho}^T\cap {\mathcal B}_u^{S}$ is nonempty.
By Theorem \ref{theorem-irreducible-intersection}, it is dense in ${\mathcal K}^T\cap{\mathcal K}^S$.  
\end{proof}

\begin{remark}
\label{remark-third-corollary} Observe that the lemma is not
necessarily true, when $T,S$ are not supposed of hook, two-row or
two-column type. Indeed, consider
\[
T=\young(125,34,6) \quad S=\young(125,36,4)
\]
Then we have $Y_{j/i}^S\preceq Y_{j/i}^T$ for any $i<j$, although
$T\not=S$. One can also see that for any
$\rho=(i_k,j_k)_{k=1,...,6}\in{\mathcal R}_6$ we have
$Y_{j_5/i_5}^S\prec Y_{j_5/i_5}^T$.
Then, by Lemma \ref{lemma-semi-continuity} and the definition of
${\mathcal B}_{u,\rho}^T$, we have ${\mathcal
B}_{u,\rho}^T\cap{\mathcal K}^S=\emptyset$. Hence ${\mathcal
B}_{u,\rho}^T\cap{\mathcal B}_u^S=\emptyset$. 
However, we can compute that the components
${\mathcal K}^T$ and  ${\mathcal K}^S$ have a nonempty
intersection which has codimension one. 
Therefore, the proposition does not hold
in this case. Adding boxes to $T$ and $S$, we can show likewise that
the proposition does not hold whenever the Jordan shape of $u$
contains {\scriptsize $\yng(3,2,1)$} as a subdiagram.
\end{remark}

\medskip

Finally, we define
$$\hat{\mathcal K}^T=\bigcup_{\rho\in{\mathcal R}_n}{\mathcal B}_{u,\rho}^T.$$
This is an open subset of the component ${\mathcal K}^T$. We show:

\begin{theorem}
Assume that $T$ is of hook, two-row or two-column type. Then, we
have $\mathrm{codim}_{{\mathcal K}^T}{\mathcal
K}^T\setminus\hat{\mathcal K}^T\geq 2$.
\end{theorem}

\begin{proof}
Arguing by contradiction, we assume $\mathrm{codim}_{{\mathcal
K}^T}{\mathcal K}^T\setminus\hat{\mathcal K}^T=1$. Then there is
$S$ standard such that $\mathrm{codim}_{{\mathcal K}^T}({\mathcal
K}^T\setminus\hat{\mathcal K}^T)\cap {\mathcal B}_u^S=1$. In
particular $\mathrm{codim}_{{\mathcal K}^T}{\mathcal K}^T\cap
{\mathcal B}_u^S=1$. As ${\mathcal K}^T\cap{\mathcal K}^S$ is
irreducible and ${\mathcal B}_u^S\cap{\mathcal B}_u^T=\emptyset$,
we have ${\mathcal K}^S\cap {\mathcal B}_u^T=\emptyset$. By
Proposition \ref{proposition_avec_lemme}, there is
$\rho\in{\mathcal R}_n$ such that $\mathrm{codim}_{{\mathcal
K}^T}{\mathcal B}_{u,\rho}^T\cap {\mathcal B}_u^S=1$. We have thus
$\mathrm{codim}_{{\mathcal K}^T}({\mathcal
K}^T\setminus\hat{\mathcal K}^T)\cap {\mathcal B}_u^S=1$ and
$\mathrm{codim}_{{\mathcal K}^T}\hat{\mathcal K}^T\cap {\mathcal
B}_u^S=1$. However, the set ${\mathcal K}^T\cap {\mathcal B}_u^S$
is irreducible, since it is open in ${\mathcal K}^T\cap {\mathcal
K}^S$. This brings a contradiction.  
\end{proof}

\begin{remark}
Let $T,S$ be as in Remark \ref{remark-third-corollary}. Then we
have ${\mathcal B}_{u,\rho}^T\cap {\mathcal K}^S=\emptyset$ for
any $\rho\in{\mathcal R}_n$, hence ${\mathcal K}^T\cap{\mathcal
K}^S\subset {\mathcal K}^T\setminus\hat{\mathcal K}^T$. Since the
components ${\mathcal K}^T,{\mathcal K}^S$ have an intersection in
codimension one, it follows $\mathrm{codim}_{{\mathcal
K}^T}{\mathcal K}^T\setminus\hat{\mathcal K}^T=1$. The theorem
does not hold in this case. Adding boxes to $T,S$, we show
likewise that the theorem does not hold whenever the Jordan shape
of $u$ contains {\scriptsize $\yng(3,2,1)$} as a subdiagram.
\end{remark}



\section*{Index of the notation}

\begin{itemize}
\item[\ref{introduction-introduction}\ ] ${\mathcal B}$, ${\mathcal B}_u$
\item[\ref{young-diagram}\ ] $Y(u)$, $Y_i^T$, ${\mathcal B}_u^T$, ${\mathcal K}^T$
\item[\ref{section_base_Jordan}\ ] $F_\tau$
\item[\ref{propositions}\ ] $F(\underline{e})$, ${\mathcal Z}_\tau$, $Z(u)$,
$\mathrm{sh}(\tau)$, $\mathrm{sh}(T)$, $|\tau|$, $|T|$, $\tau\in T$, $\mathbf{Y}$, $\mathbf{K}$
\item[\ref{standardization}\ ] $\mathrm{st}(\tau)$
\item[\ref{subsection-inductive-property}\ ] $\tau'\subseteq \tau$, $T'\subseteq T$
\item[\ref{section-schutzenberger}\ ] $T^S$, $S\cdot \tau$
\item[\ref{semi-continuity}\ ] $Y\preceq Y'$
\item[\ref{sequences-subdiagrams}\ ] $Y_{j/i}(\tau)$, $Y_{j/i}^T$, $\tau\preceq T$, $\mathbf{R}$
\item[\ref{A-hook-case}\ ] $\mathbf{Y}_{\mbox{\tiny hk}}$, $\mathbf{K}_{\mbox{\tiny hk}}$, $\mathbf{R}_{\mbox{\tiny hk}}$, $\mathbf{A}_{\mbox{\tiny hk}}$
\item[\ref{section-A-2row}\ ] $\mathbf{Y}_{\mbox{\tiny 2-r}}$, $\mathbf{K}_{\mbox{\tiny 2-r}}$, $\mathbf{R}_{\mbox{\tiny 2-r}}$, $\mathbf{\hat A}_{\mbox{\tiny 2-r}}$, $\mathbf{A}_{\mbox{\tiny 2-r}}$, $\eta(\tau,T)$ (two-row case)
\item[\ref{A-2column-case}\ ] $\mathbf{Y}_{\mbox{\tiny 2-c}}$, $\mathbf{K}_{\mbox{\tiny 2-c}}$, $\mathbf{R}_{\mbox{\tiny 2-c}}$,
$C_q(\tau)$, $C_q(T)$, $\omega_\tau(i)$, $\nu_\tau(j)$,\\
$\mathbf{\hat A}_{\mbox{\tiny 2-c}}$, $\mathbf{A}_{\mbox{\tiny
2-c}}$, $\eta(\tau,T)$ (two-column case)
\item[\ref{connection-1-2row-2column}\ ] $Y^t$, $T^t$, $S^t$
\item[\ref{section-constructibility-2column-case}\ ] $\overline{\mathbb{N}}$, $f_i(p)$, $P_i$, $Q_i$
\item[\ref{section-meanders}\ ] $M_{T,S}$, $\mathbf{M}_{\mbox{\tiny 2-r}}$
\item[\ref{T-alpha-beta}\ ] $r_T(i)$, ${\mathcal D}_{\alpha,\beta}$, ${\mathcal T}_{\alpha,\beta}$, ${\mathcal D}_i$, ${\mathcal T}_i$,
$\mathbf{V}$, $\mathbf{V}_{\mbox{\tiny 2-r}}$
\item[\ref{codim1-hook-two-column-irreducibility}\ ] $\mathbf{V}_{\mbox{\tiny hk}}$, $\mathbf{V}_{\mbox{\tiny 2-c}}$
\item[\ref{last-section}\ ] ${\mathcal R}_n$, ${\mathcal B}_{u,\rho}^T$, $\hat{\mathcal K}^T$
\end{itemize}

\end{document}